\documentclass[a4paper]{article}
\usepackage{amsmath}
\usepackage{amsthm}
\usepackage{amsfonts}
\usepackage{bm}
\usepackage[]{xcolor}

\theoremstyle{plain}
\newtheorem{theorem}{Theorem}[section]
\newtheorem{corollary}[theorem]{Corollary}
\newtheorem{lemma}[theorem]{Lemma}
\newtheorem{proposition}[theorem]{Proposition}
\newtheorem{remark}[theorem]{Remark}
\newtheorem{assumption}[theorem]{Assumption}
\theoremstyle{definition}
\newtheorem{definition}{Definition}[section]

\AtBeginEnvironment{note}{\small}

\newcommand{\conj}[1]{\overline{#1}}
\newcommand{\operator}[1]{\mathcal{#1}}
\newcommand{\norm}[1]{\|#1\|}
\newcommand{\dx}{{\,\mathrm{d}\vect x}}
\newcommand{\dt}{{\,\mathrm{d}t}}
\newcommand{\dS}{{\,\mathrm{d}S}}
\newcommand{\Div}[1]{\nabla \cdot #1}
\newcommand{\Grad}[1]{\nabla #1}
\newcommand{\Lap}[1]{\Delta #1}
\newcommand{\Gradbound}[1]{\Grad _{\partial \Omega} {#1}}
\newcommand{\TimeDer}{{\partial^{}_t}}
\newcommand{\SecondTimeDer}{{\partial^2_{tt}}}
\newcommand{\NormDer}{{\partial_{\bm{n}}}}
\newcommand{\vect}[1]{\bm{#1}}
\newcommand{\normal}{\bm{n}}
\newcommand{\mat}[1]{\underline{\underline{\mathrm{#1}}}}
\newcommand{\R}{\mathbb{R}}

\newcommand{\N}{\mathbb{N}}

\newcommand{\Ll}{{F_{\ell}}}
\newcommand{\weakto}{{\rightharpoonup}}
\newcommand{\wrt}{{\textrm{w.r.t.\ }}}
\newcommand{\LI}{L_{\mathrm I}}
\newcommand{\LD}{L_{\mathrm D}}
\newcommand{\Vs}{{V_{\star}}}
\newcommand{\Ws}{{W_{\star}}}
\newcommand{\deltaI}{{\delta_{\mathrm I}}}
\newcommand{\deltaD}{{\delta_{\mathrm D}}}

\newcommand{\GD}{{\Gamma_{\mathrm D}}}
\newcommand{\GI}{{\Gamma_{\mathrm I}}}
\newcommand{\SD}{{\Sigma_{\mathrm D}}}
\newcommand{\SI}{{\Sigma_{\mathrm I}}}
\newcommand{\gD}{{g_{\mathrm D}}}
\newcommand{\gI}{{g_{\mathrm I}}}
\newcommand{\deO}{{\partial\Omega}}
\newcommand{\oL}{{\operator{L}}}
\newcommand{\oM}{{\operator{M}}}
\newcommand{\oW}{{\operator{W}}}
\newcommand{\vx}{{\vect x}}
\newcommand{\vn}{{\vect n}}
\newcommand{\OT}{{\Omega_T}}
\newcommand{\OZ}{{\Omega_0}}
\newcommand{\GradGD}{{\Grad_{\GD}}}
\newcommand{\GradGI}{{\Grad_{\GI}}}

\newcommand{\HoneImp}{{H^1_{\mathrm I}}}

 \usepackage[]{mathtools}
\usepackage[]{amssymb}
\usepackage[english]{babel}
\usepackage{cite}
\usepackage[]{csquotes}
\usepackage[left=2.5cm,right=2.5cm,top=3cm,bottom=3.5cm]{geometry}
\usepackage[]{tikz}
\usepackage[bookmarks=true, bookmarksopen=true, bookmarksopenlevel=5]{hyperref}
\usepackage[]{pgfplots}
\usepackage{pgfplotstable}
\usepackage[]{subcaption}
\usepackage{multirow}
\usepackage{tabularray}
\usepackage{xcolor}
\usepackage{bmpsize}
\usepackage{bbm}
\usepackage{booktabs}
\allowdisplaybreaks[4]
\pgfplotsset{
        table/search path={figures, data}, legend style={font=\footnotesize}, tick label style={font=\footnotesize},
    }
\graphicspath{{figures/}}

\pgfplotsset{compat=newest}
\colorlet{Vcolor}{green!65!black}
\tikzstyle{L2error style}=[red, mark=square, line width=.8pt]
\tikzstyle{H1error style}=[blue, mark=star, line width=.8pt]
\tikzstyle{Verror style}=[Vcolor, mark=triangle, line width=.8pt]
\tikzstyle{L2bestapprox style}=[red, dashed]
\tikzstyle{H1bestapprox style}=[blue, dashed]
\tikzstyle{Vbestapprox style}=[Vcolor, dashed]
\tikzstyle{L2order style}=[red, dotted]
\tikzstyle{H1order style}=[blue, dotted]
\tikzstyle{Vorder style}=[Vcolor, dotted]

\newcommand{\logLogSlopeTriangle}[5]
{
    \pgfplotsextra
    {
        \pgfkeysgetvalue{/pgfplots/xmin}{\xmin}
        \pgfkeysgetvalue{/pgfplots/xmax}{\xmax}
        \pgfkeysgetvalue{/pgfplots/ymin}{\ymin}
        \pgfkeysgetvalue{/pgfplots/ymax}{\ymax}

        \pgfmathsetmacro{\xArel}{#1}
        \pgfmathsetmacro{\yArel}{#3}
        \pgfmathsetmacro{\xBrel}{#1-#2}
        \pgfmathsetmacro{\yBrel}{\yArel}
        \pgfmathsetmacro{\xCrel}{\xArel}

        \pgfmathsetmacro{\lnxB}{\xmin*(1-(#1-#2))+\xmax*(#1-#2)}
        \pgfmathsetmacro{\lnxA}{\xmin*(1-#1)+\xmax*#1}
        \pgfmathsetmacro{\lnyA}{\ymin*(1-#3)+\ymax*#3}
        \pgfmathsetmacro{\lnyC}{\lnyA+#4*(\lnxA-\lnxB)}
        \pgfmathsetmacro{\yCrel}{\lnyC-\ymin)/(\ymax-\ymin)}

        \coordinate (A) at (rel axis cs:\xArel,\yArel);
        \coordinate (B) at (rel axis cs:\xBrel,\yBrel);
        \coordinate (C) at (rel axis cs:\xCrel,\yCrel);

        \draw[#5]   (A)-- node[pos=0.5,anchor=north] {1}
        (B)--
        (C)-- node[pos=0.5,anchor=west] {#4}
        cycle;
    }
}
\newcommand{\logLogCondSlopeTriangle}[5]
{
    \pgfplotsextra
    {
        \pgfkeysgetvalue{/pgfplots/xmin}{\xmin}
        \pgfkeysgetvalue{/pgfplots/xmax}{\xmax}
        \pgfkeysgetvalue{/pgfplots/ymin}{\ymin}
        \pgfkeysgetvalue{/pgfplots/ymax}{\ymax}

        \pgfmathsetmacro{\xArel}{#1}
        \pgfmathsetmacro{\yArel}{#3}
        \pgfmathsetmacro{\xBrel}{#1-#2}
        \pgfmathsetmacro{\yBrel}{\yArel}
        \pgfmathsetmacro{\xCrel}{\xArel}

        \pgfmathsetmacro{\lnxB}{\xmin*(1-(#1-#2))+\xmax*(#1-#2)}
        \pgfmathsetmacro{\lnxA}{\xmin*(1-#1)+\xmax*#1}
        \pgfmathsetmacro{\lnyA}{\ymin*(1-#3)+\ymax*#3}
        \pgfmathsetmacro{\lnyC}{\lnyA+#4*(\lnxA-\lnxB)}
        \pgfmathsetmacro{\yCrel}{\lnyC-\ymin)/(\ymax-\ymin)}
        \pgfmathsetmacro{\rateText}{abs(#4)}

        \coordinate (A) at (rel axis cs:\xArel,\yArel);
        \coordinate (B) at (rel axis cs:\xBrel,\yBrel);
        \coordinate (C) at (rel axis cs:\xCrel,\yCrel);

        \draw[#5]   (A)-- node[pos=0.5,anchor=north] {1}
        (B)--
        (C)-- node[pos=0.5,anchor=east] {\rateText}
        cycle;
    }
}

\hypersetup{
    colorlinks=true,
    linkcolor=blue,
    filecolor=blue,
    urlcolor=blue,
    pdftitle={A space-time continuous and coercive formulation for the wave equation},
    pdfauthor={Bignardi, Moiola},
    citecolor=blue, }

\usetikzlibrary{shapes, arrows, chains, calc, patterns, plotmarks, matrix}
\usepgfplotslibrary{groupplots}

\title{A space--time continuous and coercive formulation for the wave equation}
\author{Paolo Bignardi\thanks{Dipartimento di Matematica ``F. Casorati'', Universit\`a di Pavia, Italy, \texttt{paolo.bignardi01@universitadipavia.it}}
    \and
    Andrea Moiola\thanks{Dipartimento di Matematica ``F. Casorati'', Universit\`a di Pavia, Italy, \texttt{andrea.moiola@unipv.it}}
}
\date{\today}

\begin{document}

\maketitle

\begin{abstract}
We propose a new space--time variational formulation for wave equation initial--boundary value problems.
The key property is that the formulation is coercive (sign-definite) and continuous in a norm stronger than $H^1(Q)$, $Q$ being the space--time cylinder.
Coercivity holds for constant-coefficient impedance cavity problems posed in star-shaped domains, and for a class of impedance--Dirichlet problems.
The formulation is defined using simple Morawetz multipliers and its coercivity is proved with elementary analytical tools, following earlier work on the Helmholtz equation.
The formulation can be stably discretised with any $H^2(Q)$-conforming discrete space, leading to quasi-optimal space--time Galerkin schemes.
Several numerical experiments show the excellent properties of the method.


\end{abstract}

\section{Introduction}
\subsection{Space--time methods for wave problems}

Most numerical schemes for the approximation of initial--boundary value problems (IBVPs) for evolution PDEs rely on separate discretisations of the space and time variables, i.e.\ either the method of lines or Rothe's method.
\emph{Space--time methods}, instead, consist of simultaneous discretisations of both variables.
Even if they may require the solution of large algebraic systems, space--time schemes can be advantageous in terms of local mesh refinement, adaptivity, parallelisation, treatment of moving boundaries and interfaces, and to obtain accurate approximations at all time instants.
While the first methods of this kind date to the late 1960's \cite{F1969,O1969},
space--time schemes were studied extensively only more recently; see e.g.\ \cite{LangerSteinbach2019,Oberwolfach22} for recent overviews.

Restricting ourselves to linear transient \emph{wave} problems, we note that the majority of space--time methods belong to the discontinous-Galerkin (DG) family, e.g.\ \cite{MoRi05}, or require discontinuous test functions, e.g.\ \cite{French1993}.
Fewer space--time \emph{conforming} Galerkin methods have been proposed: e.g.\
the stabilised FEM and isogeometric schemes in \cite{FraschiniLoliMoiolaSangalli2023,SteinbachZank2019stabilized,Zank2021higherorder},
the formulation involving a trial-to-test ``transformation operator'' in \cite[sect.~5]{steinbach2020coercive},
the first-order systems in \cite{BalesLasiecka1994}, and the first-order system of least squares (FOSLS) in \cite{FuhrerGonzalezKarkulik2025}.
We also mention the method proposed in \cite{HePaSiUr2022}, which is conforming with respect to an ultra-weak variational formulation, and uses $C^1$-continuous test and discontinuous trial functions.

Most of these formulations (an exception is the FOSLS in \cite{FuhrerGonzalezKarkulik2025}) impose strong constraints on the discrete spaces in order to obtain stable methods.
For instance, \cite[sect.~5]{steinbach2020coercive} requires piecewise multi-linear ($\mathbb Q^1$) spaces on tensor-product meshes satisfying a CFL condition;
the stabilised formulations in  \cite{SteinbachZank2019stabilized,Zank2021higherorder} and  \cite{FraschiniLoliMoiolaSangalli2023} apply to tensor-product spline spaces (piecewise $\mathbb Q^p$ with global $C^k$ continuity), but numerical experiments show that \cite{SteinbachZank2019stabilized,Zank2021higherorder} is unconditionally stable only for spaces with $C^0$ regularity in time (and not for $C^1$ or higher), while \cite{FraschiniLoliMoiolaSangalli2023} only for spaces with highest regularity ($k=p-1$).
These are strong limitations if one wants, for instance, to locally adapt (either a priori or a posteriori) the computational mesh or the discrete space to achieve better accuracy.

Ideally, one would like to have a space--time variational problem similar to the classical formulation of the Poisson--Dirichlet boundary value problem, for which any $H^1_0$-conforming discrete space gives well-posed, stable, and quasi-optimal Galerkin schemes, thanks to Lax--Milgram theorem and C\'ea's lemma.
This holds because the classical Poisson variational formulation is coercive (also called sign-definite, or elliptic).
On the other hand, the standard space--time variational formulation for the wave equation is not even inf-sup stable, see \cite[Thm.~1.1]{steinbach2022generalized}.
At a first glance, finding a space--time variational formulation for wave IBVPs that are continuous and coercive in the same norm, to ensure the applicability of Lax--Milgram and C\'ea, without resorting to least-squares schemes, seems hopeless.
However, also for the Helmholtz equation (i.e.\ the acoustic wave equation in frequency domain $\Delta u+k^2u=f$, for large $k>0$) this might appear impossible, but a coercive formulation has been proposed in \cite{Moiola2014} and implemented in \cite{DiwanMoiolaSpence2019}.
This is possible thanks to the use of Morawetz multipliers, which are special test functions that have been used to analyse wave problems, prove energy estimates and describe the solution decay, since the work of Cathleen S.~Morawetz in the 1960's \cite{morawetz1961decay,morawetz1975decay}.

\subsection{A coercive formulation that uses Morawetz multipliers}
The goal of this paper is to use techniques similar to those in \cite{Moiola2014} to derive a space--time variational formulation of some IBVPs that is continuous and coercive in a given norm.

The IBVPs considered are the impedance problem for the acoustic wave equation $\SecondTimeDer{u}-c^2\Delta u=f$ with constant wave speed $c$, and a mixed impedance--Dirichlet problem for the same PDE.
The impedance condition $\NormDer u + (\theta c)^{-1}\TimeDer{u} = \gI$ is the simplest absorbing boundary condition in acoustics, and corresponds to imposing that the fluid normal velocity on an obstacle is proportional to the acoustic pressure.

The formulation is obtained by integrating by parts and choosing as test function the linear-coefficient, first-order multiplier $\oM v(\vx,t):=-\xi\vx\cdot\nabla v(\vx,t)+\beta(t-T^*)\TimeDer v(\vx,t)$, where $\xi,\beta,T^*$ are real parameters.
The formulation, defined in \eqref{eq:def:bilin}--\eqref{eq:def:lin} and \eqref{eq:vpV} in the pure impedance case, only involves integrals (on the space--time cylinder $Q$ and parts of its boundary) of trial and test functions, their partial derivatives, volume and boundary data, together with affine coefficients.

Theorem~\ref{thm:coerc_b} and Proposition \ref{prop:continuity_b} prove the coercivity and the continuity of the formulation in a norm \eqref{eq:norm_def} stronger than (space--time) $H^1$.
All estimates are proved using only elementary vector calculus tools (such as Cauchy--Schwarz and weighted inequalities, the knowledge of the sign of the normal component of the position vector on the space boundary) and are explicit in all parameters.
We refer to Remark~\ref{remark:V=W?} for a more delicate point concerning the function space in which the variational problem is posed; this is related to the density of smooth functions in the graph space associated with the coercivity norm.

In the pure impedance case, coercivity holds for bounded, Lipschitz, star-shaped domains, see Assumption~\ref{ass:StarS}.
More generally, star-shaped impedance cavities with star-shaped Dirichlet obstacles are allowed, see Assumption~\ref{ass:StarSD}.
This kind of restrictions is common when Morawetz multipliers are used and is related to the absence of trapped rays; see e.g.~\cite[sect.~6]{Moiola2014} and \cite{morawetz1961decay}.

This formulation can be stably discretised with \emph{any} $H^2(Q)$-conforming finite element space giving well-posed and quasi-optimal schemes.
This is the ideal starting point to develop efficient solvers, matrix-compression techniques, and adaptive algorithms.
We demonstrate the excellent numerical performance of the formulation with several experiments involving cubic spline spaces: we propose simple recipes for the choice of the numerical parameters, we show that the method is unconditionally stable and parameter-robust, it provides errors close to the best-approximation, it has little dissipation and it can approximate rough solutions.
The Matlab code implemented is freely available online.
This work is part of the PhD project of the first author: see the dissertation \cite{Bignardi2025},
which includes the extension of the formulation to more general problems, the a-posteriori error analysis, an adaptive scheme, and a related $C^0$-interior penalty method.
More extensive and challenging experiments will be reported in a subsequent work.

\subsection{Outline of the paper}
In section~\ref{sec:IBVPandRegularity} we state the IBVP under consideration, study in detail its well-posedness in the impedance case and prove a regularity result on its solution, adapting a Faedo--Galerkin technique from~\cite{ladyzhenskaya-bvp}.

In section~\ref{sec:abstract_framework}, we propose an abstract framework for coercive formulations arising from multiplier techniques.
This framework includes both what is proposed in the remainder of the paper for the wave equation, and the coercive formulation for the Helmholtz equation in \cite{Moiola2014}.

Section~\ref{sec:Identities} contains the pointwise and integral identities related to the Morawetz multipliers.
These identities are proved using only the elementary vector calculus product rules and integration by parts.

Section~\ref{sec:interior_problem} contains the main results for the interior impedance problem.
In particular, it contains the variational formulation and the corresponding norm, the proofs of the (explicit) coercivity and continuity bounds for the bilinear and the linear form.
In section~\ref{sec:NumImplications} we describe some consequences that concern Galerkin discretisations: we estimate the quasi-optimality constant and the Galerkin matrix condition number, show that conforming discrete spaces essentially coincide with $C^1(\overline Q)$-conforming ones, and study the energy-dissipation properties of the formulation.
Since a possible obvious criticism of space--time methods concerns the size of the linear systems involved, in Remark~\ref{rem:DOFcount} we also compare the number of degrees of freedom needed to achieve a given accuracy against those required by a time-stepping scheme.

Section~\ref{sec:Scattering} extends the previous definitions and results to the more general case involving Dirichlet conditions on part of the boundary (e.g.\ a sound-soft scatterer in an impedance cavity).

In section~\ref{sec:numExperiments} we show several numerical experiments for a cubic-spline discretisation of the proposed formulation.
We study the sensitivity of the formulation on some numerical parameters in its definitions, the optimality of the convergence rates, the approximation of smooth and singular solutions, the sharpness of the quasi-optimality estimates, the approximation of the solution energy.

Finally, in section~\ref{sec:Conclusions} we draw some conclusions and list several possible extensions of the proposed formulations and open problems.

\section{Model problem, well-posedness and solution regularity}\label{sec:IBVPandRegularity}

\subsection{Model problem}\label{sec:IBVP}

Let $\Omega \subset \R^d$ be a Lipschitz bounded domain whose boundary is partitioned in (relatively open) impedance and Dirichlet parts as $\deO=\conj{\GI\cup\GD}$, with $\GI\cap\GD=\emptyset$ and $\GI\ne\emptyset$, and let $T>0$.
Define
\begin{align*}
	Q        & := \Omega \times (0,T), \quad            &
	\Omega_t & := \Omega \times \{t\}, \quad            &
	\Sigma_t & := \Sigma \times \{t\},        \\
	\SI      & := \GI \times (0,T), \quad               &
	\SD      & := \GD \times (0,T),\quad                &
    \Sigma   & :=\overline{\SI\cup\SD}=\deO\times[ 0,T ],
\end{align*}
where $t \in [0,T]$. Consider the interior impedance--Dirichlet problem
\begin{equation}\label{eq:IBVP}
	\begin{aligned}
\SecondTimeDer{u} -c^2 \Lap u            & = f \qquad   &  & \text{on } Q,       \\
		\NormDer u  + (\theta c)^{-1}\TimeDer{u} & = \gI \qquad &  & \text{on }\SI,      \\
		u                                        & = \gD \qquad &  & \text{on }\SD,      \\
		u                                        & = u_0        &  & \text{on }\Omega_0, \\
		\TimeDer{u}                              & = u_1        &  & \text{on }\Omega_0,
	\end{aligned}
\end{equation}
where the wavespeed $c > 0$ and the impedance parameter $\theta > 0$ are constants, and $f,\gI,\gD, u_0$ and $u_1$ are appropriate functions defined on $Q$, $\SI$, $\SD$, $\Omega_0$ and $\Omega_0$, respectively.
We denote by $\Delta$ ($\Grad$ and $\Div$) the Laplacian (the gradient and the divergence, respectively) in the space variable only.
The notation $\NormDer$ denotes the normal derivative $\NormDer u=\vn\cdot\Grad u$ on $\Sigma$, where $\vn$ is the (space) outward-pointing unit normal field on $\partial\Omega$.
We indicate with $\GradGI$ and $\GradGD$ the tangential gradient over $\GI$ and $\GD$, respectively. Moreover, let $\LD := \max\{|\vx|: \vx\in\GD\}$ and $\LI := \max\{|\vx|: \vx\in\GI\}$.
We allow the case $\GD=\emptyset$, which corresponds to the impedance cavity problem.

\subsection{Generalised solutions and well-posedness of the impedance problem}
\label{sec:Regularity}
The content of this section is an extension of the results by Ladyzhenskaya \cite[Ch.~IV]{ladyzhenskaya-bvp} to the case of impedance boundary conditions. Similar results can be proven also using the Laplace transform (see e.g. \cite{ChaumontFrelet2022}) or contraction semigroup theory (see e.g. \cite{ChaumontFreletEtAl2022}).

In this and the next subsection we assume that $\GD=\emptyset$: this is not a strong limitation because the Dirichlet part of the boundary could be treated in the same way as in \cite{ladyzhenskaya-bvp}.
Let
\begin{align*}
	\HoneImp(Q):=&\{u\in H^1(Q): \TimeDer u|_\SI\in L^2(\SI)\},
	\\
	\widehat{H}^1(Q):=&\{u \in H^1(Q): u(\cdot, T) = 0\}.
\end{align*}
We indicate with $(\cdot, \cdot)_\Omega$ the scalar product in $L^2(\Omega)$ and with $(\cdot, \cdot)_{\GI}$ the scalar product in $L^2(\GI)$; for $u, v \in H^1(Q)$,
both scalar products are
function of $t$.
Similarly, $\| \cdot \|_E$ is the $L^2(E)$ norm, where $E$ is any (relatively) open subset of $Q,\partial Q$, or $\partial\Omega$; when $E=\Omega$ or $E=\GI$, then $\|\cdot\|_E$ depends on $t$.
We use the same notation when the argument of the norm is a vector field.

Next we define a class of solutions to problem \eqref{eq:IBVP}.
\begin{definition}\label{def:imp_gsol}
	A function $u\in \HoneImp(Q)$ is a
	\textbf{generalised solution} of problem \eqref{eq:IBVP}
	with $\GD=\emptyset$ if $u = u_0\in H^1(\Omega)$ on $\OZ$ and, for all $v \in \widehat{H}^1(Q)$, the following
	variational problem
	is satisfied
	\begin{equation}
		\label{eq:imp_varprob}
		\int_Q\big( -\TimeDer u \TimeDer v + c^2 \Grad u \cdot \Grad v\big) \dx \dt + \frac{c}{\theta} \int_\SI \TimeDer u v \dS dt = \int_Q f v \dx \dt + c^2 \int_\SI \gI v \dS \dt + \int_{\Omega_0}u_1 v \dx,
	\end{equation}
	where $f \in L^1(0, T; L^2(\Omega))$, $\gI \in L^2(\SI)$ and $u_1 \in L^2(\Omega)$.
\end{definition}
Existence and uniqueness of a generalised solution is guaranteed by the following theorem, which is a modification of Theorems IV.3.1--3.2 in \cite{ladyzhenskaya-bvp} for the Dirichlet problem.

\begin{theorem}[Existence of a solution to the impedance BVP]\label{thm:Existence}
	Problem \eqref{eq:IBVP} admits a solution in the sense of Definition~\ref{def:imp_gsol}. Moreover, such solution is unique.
\end{theorem}
\begin{proof}
	The proof relies on the same argument of \cite[Thm.~IV.3.2]{ladyzhenskaya-bvp}, and is essentially the same except for the stability estimates.
	Let $\{\varphi_k\}_{k\in\N}$ be the eigenfunctions of $-\Lap$ in $\Omega$ with Neumann boundary conditions, and recall that they are orthogonal in $H^1$ and orthonormal in $L^2$. Let $u^N(x, t) = \sum_{k=1}^N \hat{u}_k^N(t)\varphi_k(x)$ be the solution of the finite-dimensional ODE system \begin{equation}
		\label{eq:Ngalerkin}
		(\SecondTimeDer u^N, \varphi_k)_\Omega + c^2 (\Grad u^N, \Grad \varphi_k)_\Omega + \frac{c}{\theta}(\TimeDer u^N, \varphi_k)_{\deO}
		= (f, \varphi_k)_\Omega + c^2 (\gI, \varphi_k)_{\deO} \qquad \forall k=1, \dots, N
	\end{equation}
	with initial conditions
	\begin{equation*}
		\hat{u}_k^N(0) = (u_0, \varphi_k)_{H^1(\Omega)}, \qquad \TimeDer \hat{u}_k^N(0) = (u_1, \varphi_k)_\Omega \qquad \forall k = 1,\dots,N.
	\end{equation*}
Multiplying each equation by $\TimeDer \hat{u}_k^N(t)$ and summing over $k = 1, \dots, N$, yields
	\begin{equation*}
		\frac12 \TimeDer \|\TimeDer u^N\|^2_\Omega + c^2 \frac12 \TimeDer \|\Grad u^N\|^2_\Omega + \frac{c}{\theta}\|\TimeDer u^N\|^2_{\GI} = (f, \TimeDer u^N)_\Omega + c^2 (\gI, \TimeDer u^N)_{\GI} \qquad \forall t\in[0,T],
	\end{equation*}
Then, the weighted Young inequality yields
	\begin{equation}
		\label{eq:est0}
		\frac12 \TimeDer \|\TimeDer u^N\|^2_\Omega + c^2 \frac12 \TimeDer \|\Grad u^N\|^2_\Omega +\frac12\frac{c}{\theta}\|\TimeDer u^N\|^2_{\GI}
		\le \frac12 T\|f\|^2_\Omega
		+\frac1{2 T}\|\TimeDer u^N\|^2_\Omega
		+\frac{c^3\theta}2 \|\gI\|^2_\GI
		\qquad \forall t \in [0, T]
	\end{equation}
	which in turn yields, neglecting the term $\|\TimeDer u^N\|_\GI$ and using Gronwall's inequality \cite[sect.~B.2.j]{EvansPDE},
\begin{equation}\label{eq:est1}
\|\TimeDer u^{ N}\|_{\Omega_t}^2 +  c^2\|\Grad u^{ N}\|_{\Omega_t}^2\le
e^1\big(\|u_1\|_\OZ^2+c^2\|\Grad u_0\|_\OZ^2+ T\|f\|_Q^2+c^3\theta\|\gI\|_\SI^2\big)\qquad\forall t \in[0,T].
\end{equation}
Note that the right-hand side in \eqref{eq:est1} is independent of $N$.

	To obtain a bound for $\|\TimeDer u^N\|_\GI$, integrating \eqref{eq:est0} over $[0,t]$, and dropping $\|\TimeDer u^N\|^2_{\Omega_t}$ and $\|\Grad u^N\|^2_{\Omega_t}$ from the left-hand side, yields
	\begin{equation*}
		\frac{c}{\theta}\int_0^t \|\TimeDer u^N\|_\GI^2 \mathrm{d}s
		\le \|f\|^2_Q
		+\|\TimeDer u^N\|^2_Q
		+c^3\theta \|\gI\|^2_\SI
		+\frac12\|\Grad u_0\|^2_\OZ
		+\frac12\|u_1\|^2_\OZ
		\qquad \forall t \in (0, T].
	\end{equation*}
	Then, $\|\TimeDer u^N\|^2_Q$ can be bounded by $C_1 T$ because of \eqref{eq:est1}, so it holds that
	\begin{equation*}
		\|\TimeDer u^N\|^2_{\SI} \leq C_2
	\end{equation*}
	where $C_2$ has the same dependences of $C_1$ and is independent of $N$.
The $L^2(Q)$ norm $\|u^N\|_Q$ can be bounded using the initial datum $u_0$ and $\|\TimeDer u\|_Q$, therefore $\|u^N\|_{H^1(Q)}+\|\TimeDer u^N\|_\SI$ is uniformly bounded \wrt $N$.
	Hence, up to a subsequence, $u^N \weakto u$, $\TimeDer u^N \weakto \TimeDer u$ and $\Grad u^N \weakto \Grad u$ in $L^2(Q)$, $\TimeDer u^N \weakto \TimeDer u$ in $L^2(\SI)$ for some $u \in \HoneImp(Q)$. It remains to prove that $u$ is a generalised solution.
	Indeed, let $\mathcal{D}_N := \{v = \sum_{k=1}^N \hat{v}_k(t) \varphi_k(x), \hat{v}_k \in H^1(0,T), \hat{v}_k(T) = 0\}$.
	Multiply each equation in \eqref{eq:Ngalerkin} by $\hat{v}_k$ (as in the definition of $\mathcal{D}_N$), sum across $k = 1, \dots, N$ and integrate over $[0,T]$; integration by parts yields that
	\begin{equation*}
		\int_Q \big(-\TimeDer u^N \TimeDer v + c^2 \Grad u^N \cdot \Grad v \big)\dx \dt - \int_{\Omega_0}\TimeDer u^N v \dx + \frac{c}{\theta} \int_\SI \TimeDer u^N v \dS \dt = \int_Q f v \dx \dt + c^2 \int_\SI \gI v \dS \dt
	\end{equation*}
	holds for all $v \in \mathcal{D}_{M}$, $\forall M\leq N$.
Now let $N \rightarrow \infty$; by weak convergence and because $\TimeDer u^N(\cdot, 0) \rightarrow u_1$ in $L^2(\Omega)$ we obtain
	\begin{equation*}
		\int_Q \big(-\TimeDer u \TimeDer v + c^2 \Grad u \cdot \Grad v \big)\dx \dt - \int_{\Omega_0}u_1 v \dx + \frac{c}{\theta} \int_\SI \TimeDer u v \dS \dt = \int_Q f v \dx \dt + c^2 \int_\SI \gI v \dS \dt
	\end{equation*}
	for all $v \in \mathcal{D}_{M}$, for $M$ fixed. But $\bigcup_{M = 1}^{\infty} \mathcal{D}_M$ is dense in $\widehat{H}^1(Q)$, and therefore the variational identity above holds for all $v \in \widehat{H}^1(Q)$.

	The argument to prove uniqueness is the same as in \cite[Thm.~IV.3.1]{ladyzhenskaya-bvp}.
\end{proof}

\subsection{Regularity of solutions for impedance problem}
We study the regularity of generalised solutions because it allows to establish the relation with classical solutions and to understand under which assumptions on the problem data the solution belongs to the function spaces that will be defined and used in \S\ref{sec:spaces}.
Following \cite[Thm.~IV.4.1]{ladyzhenskaya-bvp}, we control the terms involving second-order time derivatives by repeating the Faedo--Galerkin argument for $\TimeDer u$, while the second-order space derivatives and the tangential gradient are controlled using classical elliptic regularity results at fixed times.
The $C^{1,1}$ regularity assumption on the domain $\Omega$ is used only to ensure that the solution of a Neumann problem with boundary datum in $H^{\frac12}(\GI)$ belongs to $H^2(\Omega)$ (in parts (iv) and (v) of the proof); we expect the same to hold under different assumptions on $\Omega$ such as convexity.

\begin{theorem}
	\label{thm:t_regularity}
	Assume that $\GD=\emptyset$, that $\Omega$ has a $C^{1,1}$ boundary,
	\begin{equation}\label{eq:cond_reg0}
		f\in H^1\big(0,T;L^2(\Omega)\big), \qquad \gI\in H^1\big(0,T;H^\frac12(\GI)\big),
\qquad u_0 \in H^2(\Omega),\qquad u_1 \in H^1(\Omega),
	\end{equation}
	and that the initial and boundary data are compatible, in the sense that \begin{equation}
		\label{eq:compatibility}
		\NormDer u_0 + (c\theta)^{-1} u_1 = \gI(\cdot, 0) \qquad \text{on}\; \partial\Omega.
	\end{equation}
	Then the generalised solution $u$ to problem \eqref{eq:imp_varprob} belongs to $H^2(Q)$ and the trace $\SecondTimeDer u$ to $L^2(\SI)$.

\end{theorem}

\begin{proof}
{\bf Part (i): $\SecondTimeDer{u},\Grad \TimeDer{u}\in L^2(Q)$, $\SecondTimeDer u\in L^2(\SI)$. \;}
    Following \cite[Thm.~IV.4.1]{ladyzhenskaya-bvp},
	as in the proof of Theorem~\ref{thm:Existence} we use the Faedo--Galerkin method to prove the existence of a solution $u$ with $\SecondTimeDer u$ and $\Grad \TimeDer{u}$ in $L^2(Q)$, and conclude using the uniqueness of such solution.

	Let $\{\varphi_k\}_k$ be a fundamental system of $H^2(\Omega)$ and $V^N = \mathrm{span}(\varphi_1, \dots, \varphi_N)$.
In Part (v) below we show that there exist $w_0^N, z_0^N, w_1^N\in V^N$ such that $w_0^N\rightarrow u_0$ in $H^2(\Omega)$, $z_0^N\rightarrow0$ in ${H^2(\Omega)}$, $w_1^N\rightarrow u_1$ in $H^1(\Omega)$ and
	\begin{align}\label{eq:w0w1z0}
		\int_\GI (\NormDer w_0^N+(c\theta)^{-1}w_1^N-\gI(\cdot, 0))\varphi_M\dS=\int_\Omega(\Lap z_0^N-z_0^N)\varphi_M\dx\qquad
		\forall M=1,\ldots,N. \end{align}

Consider the Galerkin ODE system defined as in \eqref{eq:Ngalerkin}, with initial conditions $u^N(\cdot, 0)=w^N_0$ and $\TimeDer u^N(\cdot, 0)=w^N_1$, indicating again as $u^N(x, t) = \sum_{k=1}^N \hat u^N_k(t) \varphi_k(x)$ its solution.

Taking the time-derivative of \eqref{eq:Ngalerkin}, multiplying by $\SecondTimeDer{\hat{u}_k^N}(t)$ and summing over $k$, for a.e. $t\in(0,T)$ we obtain
\begin{equation*}
		(\partial^3_{ttt} u^N, \SecondTimeDer{u^N})_\Omega+c^2 (\Grad{\TimeDer{u^N}} ,\Grad \SecondTimeDer{u^N})_\Omega + \frac{c}{\theta} (\SecondTimeDer{u^N}, \SecondTimeDer{u^N})_{\GI}
		= (\TimeDer f, \SecondTimeDer{u^N})_\Omega + c^2 (\TimeDer{\gI}, \SecondTimeDer{u^N})_{\GI}.
	\end{equation*}
Using weighted Young's and Gronwall's inequalities yields
	\begin{align}
		\label{eq:regStability}
		\|\SecondTimeDer u^N\|^2_{\Omega_t} + c^2 \|\Grad \TimeDer u^N\|^2_{\Omega_t} \le C\big(\|\TimeDer f\|^2_Q + \|\SecondTimeDer u^N\|^2_\OZ + c^2 \|\TimeDer \Grad u^N\|^2_{\Omega_0} + c^3\theta\|\TimeDer g\|^2_\SI \big)
	\end{align}
	where $C$ is independent of $N$.
	The norms of $u^N(\cdot, 0)=w_0^N$ and $\TimeDer{u^N}(\cdot, 0)= w_1^N$ can be bounded uniformly \wrt $N$ using
	that they converge to $u_0$ and $u_1$, respectively.
Moreover, $\|\SecondTimeDer u^N\|_\OZ$ can be uniformly bounded with respect to $N$ by testing \eqref{eq:Ngalerkin} against $\SecondTimeDer u^N$, evaluating at $t=0$ and integrating by parts in space to obtain
	\begin{equation*}
		\|\SecondTimeDer u^N\|_\OZ^2 = c^2\big(\Lap u^N(\cdot, 0) + f(\cdot, 0), \SecondTimeDer u^N(\cdot, 0)\big)_\Omega + c^2\big(\NormDer u^N(\cdot, 0)+(c\theta)^{-1}\TimeDer u^N(\cdot, 0)-\gI(\cdot, 0),\SecondTimeDer u^N(\cdot, 0)\big)_\GI.
	\end{equation*}
	The initial conditions of the ODE system, and the relation \eqref{eq:w0w1z0} satisfied by $w_0^N$ and $w_1^N$ imply that
	\begin{equation*}
		\|\SecondTimeDer u^N\|_\OZ^2 = c^2\big(\Lap w_0^N + f(\cdot, 0), \SecondTimeDer u^N(\cdot, 0)\big)_\Omega +
		c^2\big(\Lap z_0^N - z_0^N, \SecondTimeDer u^N(\cdot, 0)\big).
	\end{equation*}
	Using Young's inequality and the fact that $z_0^N\rightarrow0$ in $H^2(\Omega)$, it holds that
	\begin{equation*}
		\|\SecondTimeDer u^N\|_\OZ^2 \le C\big( \|\Lap u_0\|_\OZ^2 + \|f(\cdot, 0)\|_\OZ^2\big),
	\end{equation*}
	where $C$ is independent of $N$.
	Then from \eqref{eq:regStability}, for every $t \in (0,T)$
	\begin{equation*}
		\|\SecondTimeDer{u^N}\|^2_{\Omega_t} + \|\Grad \TimeDer u^N\|^2_{\Omega_t} \leq C,
	\end{equation*}
	where $C$ depends only on $\|f\|_{H^1(0,T;L^2(\Omega))}$, $\|\TimeDer \gI\|_\SI$, $\|u_1\|_{H^1(\OZ)}$, $\|u_0\|_{H^2(\OZ)}$, $c,\theta,T$, therefore $\SecondTimeDer{u^N}$, $\Grad \TimeDer u^N$ are bounded in $L^2(Q)$, uniformly in $N$.
	Similarly, $\|\SecondTimeDer u^N\|^2_{\SI}$ is controlled uniformly.
Then, as in the proof of the existence theorem, the assertion of Part~(i) follows using compactness and showing that the limit of the sequence $u^N$ is a solution of the generalised problem \eqref{eq:imp_varprob}.

{\bf Part (ii): elliptic problem at fixed time. \;}
We now follow the strategy of \cite[Cor.~IV.4.1]{ladyzhenskaya-bvp}.
	Since $u$ is a generalised solution as in Definition~\ref{def:imp_gsol}, \begin{equation*}
    \int_Q \big(\SecondTimeDer{u}\, v + c^2 \Grad u \cdot \Grad v\big) \dx \dt + \frac{c}{\theta} \int_\SI \TimeDer{u}\, v \dS \dt
    = \int_Q f v \dx \dt + c^2 \int_\SI \gI v \dS \dt
	\end{equation*}
	for any $v \in L^2(0, T; H^1(\Omega))$.
	This follows from integration by parts in time of \eqref{eq:imp_varprob} and the density of $\widehat H^1(Q)$ in $L^2(0,T;H^1(\Omega))$.
Assume now that $v(\vx,t) = \chi(t) \phi(\vx)$, where $\chi$ is an arbitrary element of $L^2(0,T)$ and $\phi$ an arbitrary element of $H^1(\Omega)$. Using Fubini's theorem, split the integrals over $Q$ and $\SI$ as integrals over $(0,T)$ and $\Omega$, $\partial \Omega$ respectively. Since $\chi \in L^2(0,T)$ is arbitrary, it follows that, for a.e.~$t\in [0,T]$,
	$u(\cdot,t)$ is solution of the following inhomogeneous elliptic Neumann problem:
	\begin{equation}\label{eq:ellipticVP}
		\int_{\Omega} c^2 \Grad u \cdot \Grad \phi \dx
		= \int_{\Omega} (f - \SecondTimeDer{u})\phi \dx + c^2 \int_{\GI} \big(\gI - (c\theta)^{-1} \TimeDer{u}\big)\phi \dS
		\qquad \forall\phi\in H^1(\Omega).
	\end{equation}
	We recall that, following the notation used so far, each of these integrals depends on $t$ (since $u,f,\gI$ depend on $t$ as well as on $\vect x$).

	{\bf Part (iii): tangential gradient $\GradGI u\in L^2(\SI)$. \;}
	We make use of the elliptic problem \eqref{eq:ellipticVP}.
	For a.e.~$t\in(0,T)$,
	\begin{align*}
    \|\GradGI u\|_\GI
    &\le C \big(\|f-\SecondTimeDer u\|_\Omega+\|\gI-(c\theta)^{-1}\TimeDer u\|_\GI\big)\\
    &\le C \big(\|f\|_\Omega +\|\SecondTimeDer u\|_\Omega+ \|\gI\|_\GI + \|\TimeDer u\|_{H^1(\Omega)}\big)
	\end{align*}
	where $C>0$ is a constant independent of $N$, and the first inequality follows from the elliptic regularity result of Ne\v cas for the Neumann-to-Dirichlet map (see e.g.~\cite[sect.~5.2.1]{necas2011direct}, \cite[Thm.~4.24(ii)]{McLean00}) and the second from the continuity of the Dirichlet trace in $H^1(\Omega)$.
	Then, integrating in $t$,
	$$
	\|\GradGI u\|_\SI \le \tilde C \big(\|f\|_Q +\|\SecondTimeDer u\|_Q+ \|\gI\|_\SI + \|\TimeDer u\|_{L^2(0,T;H^1(\Omega))}\big)
	$$
	which is controlled by the problem data thanks to Part~(i).

	{\bf Part (iv): $u\in H^2(Q)$. \;}
	The only terms that are left to be controlled are the second-order space derivatives.
	We apply the elliptic regularity theorem \cite[Thm.~2.3.3.6]{grisvard1985elliptic}
to \eqref{eq:ellipticVP}: since $\Omega$ is $C^{1,1}$ regular, the volume source $(f(\cdot, t)-\SecondTimeDer u(\cdot, t))\in L^2(\Omega)$, the boundary term $(\gI(\cdot, t)-(c\theta)^{-1}\TimeDer u(\cdot, t))\in H^\frac12(\GI)$ (from the assumptions on $\gI$, Part (i), and the trace theorem), we can control
	$\|u\|_{H^2(\Omega)}$ with the problem data uniformly for a.e.~$t$.
	Integrating in $t$ and recalling that second-order derivatives involving $\TimeDer$ were bounded in Part (i), we control $\|u\|_{H^2(Q)}$ with the problem data.

	{\bf Part (v): existence of $w_0^N,w_1^N,z_0^N$.} \;
	To prove this, first take $v_0^N,v_1^N\in V^N$ such that $v_0^N\rightarrow u_0$ in $H^2(\Omega)$ and $v_1^N\rightarrow u_1$ in $H^1(\Omega)$, thanks to the density of $H^2(\Omega)\subset H^1(\Omega)$.
	By the compatibility condition \eqref{eq:compatibility} and the trace theorem,
\begin{align*}
		\|\NormDer v_0^N + (c\theta)^{-1}v_1^N - \gI(\cdot, 0)\|_{H^\frac12(\GI)}& = \|\NormDer v_0^N + (c\theta)^{-1}v_1^N - \gI(\cdot, 0)-\NormDer u_0-(c\theta)^{-1}u_1+\gI(\cdot, 0)\|_{H^\frac12(\GI)}\\
		&\leq \|v_0^N-u_0\|_{H^2(\Omega)}+(c\theta)^{-1}\|v_1^N-u_1\|_{H^1(\Omega)}\rightarrow0.
	\end{align*}
	Define $z_0^N\in V^N$ as the solution of the following discrete Neumann problem:
\begin{equation*}
		\int_\Omega (\Grad z_0^N \cdot \Grad \varphi_M + z_0^N\varphi_M) \dx = \int_\GI -\big(\NormDer v_0^N+(c\theta)^{-1}v_1^N-\gI(\cdot,0)\big)\varphi_M\dS \qquad \forall M=1,\ldots,N. \end{equation*}
	Since $\Omega$ has boundary with $C^{1,1}$ regularity, then, by \cite[Thm.~2.3.3.6, 2.4.2.7]{grisvard1985elliptic}, $z_0^N\in H^2(\Omega)$ and
	\begin{equation*}
		\|z_0^N\|_{H^2(\Omega)}\leq C\|\NormDer v_0^N + (c\theta)^{-1}v_1^N - \gI(\cdot, 0)\|_{H^\frac12(\GI)}\rightarrow0,
	\end{equation*}
	for a constant $C$ independent of $N$. Then define $w_0^N:=v_0^N+z_0^N$, $w_1^N:=v_1^N$; it holds that $w_0^N\rightarrow u_0$ in $H^2(\Omega)$ (since $z_0^N\rightarrow0$ in $H^2(\Omega)$) and that $w_1^N\rightarrow u_1$ in $H^1(\Omega)$.
	Moreover, because $z_0^N$ is solution of the problem described above, it holds that
	\begin{align*}
		\int_\GI \big(\NormDer w_0^N + (c\theta)^{-1}w_1^N-\gI(\cdot,0)\big)\varphi_M\dS
		&=\int_\GI\big(\NormDer v_0^N + \NormDer z_0^N + (c\theta)^{-1}v_1^N-\gI(\cdot,0)\big)\varphi_M\dS\\
		&=\int_\Omega (-\Grad z_0^N\cdot\Grad\varphi_M -z_0^N\varphi_M)\dx + \int_\GI \NormDer z_0^N\varphi_M\dS\\
		&=\int_\Omega(\Lap z_0^N-z_0^N)\varphi_M\dx
		\qquad\qquad\forall M=1,\ldots,N.
	\end{align*}
\end{proof}
\begin{remark}
Theorem~\ref{thm:t_regularity} also shows that, when $\Omega$, $f$, $\gI$, $u_0$ and $u_1$ satisfy the conditions stated in the theorem,
the solution to \eqref{eq:imp_varprob} satisfies the wave equation $\SecondTimeDer u-c^2\Delta u=f$ for almost every $(\vect x, t) \in Q$.
\end{remark}

\section{Abstract framework}
\label{sec:abstract_framework}

Since the derivation of Morawetz identities and their use in the construction of variational formulation are quite involved, we propose a simple abstract framework that can be used to derive coercive formulations arising from multipliers.
This highlights the common structure of the space--time formulation proposed in \S\ref{sec:interior_problem} and of the formulation for the Helmholtz equation in \cite{Moiola2014}.

\begin{remark}[Real and complex problems]\label{rem:ComplexReal}
	In this section, we assume that all the functions and the spaces are complex-valued, unless otherwise specified, thus we use sesquilinear forms and antilinear functionals.
	This allows to include the Helmholtz impedance boundary value problem in Remark~\ref{rem:Helmholtz}.
	In the following sections, which concern the wave equation, all spaces are real, so we use bilinear forms and linear functionals, but all our results immediately extend to the complex-valued case.
\end{remark}

Consider the following abstract boundary value problem:
\begin{equation}
	\label{eq:abstract_problem}
	\begin{cases}
		\operator{L} u = f \qquad \text{on } {D} \\
		\operator{B} u = g \qquad \text{on } \partial {D}
	\end{cases}
\end{equation}
with $D \subset \R^d$,
$f\in L^2(D)$,
$\operator L:V\to L^2(D)$ a differential operator where $V\subset$ $L^2(D)$,
$g\in V_{\partial D}$ where $V_{\partial D}$ is a function space on $\partial D$ and
$\operator{B}:V\to V_{\partial D}$ a trace operator.

Consider an operator $\operator{M}:V\to L^2(D)$, which will play the role of the Morawetz multiplier, and introduce a Hermitian sesquilinear form $m$  defined as:
\begin{align}\label{eq:m_def}
	m(u,v) :=\int_D (\operator{L}u \conj{\operator{M}v} + \operator{M}u\conj{\operator{L}v})\dx,
	\qquad u,v\in V.
\end{align}
For all $u\in V$, it holds that
\begin{align}\label{eq:real_part}
	\Re \int_D \operator{L}u\conj{\operator{M}u} \dx = \frac{1}{2} m(u,u).
\end{align}

We denote by $r$ and $s$ two sesquilinear forms on $V$ such that
\begin{subequations}\label{eq:assum:boundary_identity}
	\begin{align}\label{eq:assum:splitId}
		 & m(u,v) = r(u,v) + s(u,v),                                                  \intertext{and}
		\label{eq:assum:S=Sg}
		 & \text{if }\quad\operator{B} u = g \quad\text{ then }\quad s(u,v) = S_g(v),
	\end{align}
\end{subequations}
for an appropriate antilinear form $S_g$ that depends on the boundary datum $g$.
In the concrete cases below, the sesquilinear forms $r$ and $s$ are obtained through integration by parts and other manipulations of $m(u,v)$.
Let $\ell$ be a least-squares-type sesquilinear form, i.e.\ it is Hermitian, it defines a seminorm over the space $V$, and
\begin{equation}
	\label{eq:assum:ell=L}
	\text{if} \quad \oL u = f \quad \text{and} \quad \operator{B} u = g \quad \text{then} \quad
	\ell(u, v) = \Ll(v) \qquad \forall v \in V,
\end{equation}
where $\Ll$ is an antilinear form that is defined solely using $f$ and $g$.
The sesquilinear form $\ell$ is used to introduce positive-semidefinite terms that help controlling some components of the norm on $V$, which in turn are needed to ensure the continuity of $r$ and $s$.
To satisfy Assumption \eqref{eq:assum:ell=L} one could choose for example $\ell(u,v) = \int_D \oL u \oL v \dx$ and $\Ll(v) = \int_D f \oL v \dx$.

We use $r$ and $\ell$ to define another sesquilinear form $b$ as:
\begin{align}\label{eq:b}
	b(u,v) :=
	\int_D \operator{M}u\conj{\operator{L}v} \dx - r(u,v) + \ell(u, v),
	\qquad u,v\in V.
\end{align}
The sesquilinear form $b$ can be used to solve the abstract problem \eqref{eq:abstract_problem} together with an appropriately defined antilinear form $F$.
Using relation \eqref{eq:assum:splitId} and the definition of $m$ \eqref{eq:m_def}, it holds that
\begin{align*}
	b(u,v) = \int_D -\operator{L}u\conj{\operator{M}v} \dx + s(u,v) + \ell(u, v).
\end{align*}
Moreover, if $u\in V$ is a solution to problem \eqref{eq:abstract_problem}, by assumptions \eqref{eq:assum:S=Sg} and \eqref{eq:assum:ell=L},
\begin{align*}
	b(u,v) = \int_D -f\conj{\operator{M}v} \dx + S_g(v) + \Ll(v).
\end{align*}
This suggests to define the antilinear form $F$ as
\begin{align}\label{eq:F}
	F(v) := \int_D -f\conj{\operator{M}v} \dx+ S_g(v) + \Ll(v),
\end{align}
so that if $u\in V$ is a solution of problem \eqref{eq:abstract_problem} then
\begin{equation}\label{eq:buv=Fv}
	b(u,v) = F(v) \qquad \forall v\in V.
\end{equation}

A sufficient condition for  the well-posedness of \eqref{eq:buv=Fv} is the coercivity of $b$.
Recalling the definitions of $b$ \eqref{eq:b} and $m$ \eqref{eq:m_def}, and the identity \eqref{eq:real_part},
\begin{align}\begin{aligned}
		\Re b(v,v) & = \ell(v,v) + \Re \int_D \operator{M}v \conj{\operator{L}v} \dx - \Re r(v,v)               \\
		           & = \ell(v, v) + \frac{1}{2} \Re r(v,v) + \frac{1}{2} \Re s(v,v) - \Re r(v,v)                \\
		           & = \ell(v, v) - \frac{1}{2}\Re r(v,v) + \frac{1}{2} \Re s(v,v) \qquad\qquad \forall v\in V.
	\end{aligned}\label{eq:abstract_coercivity_condition}\end{align}
Let $\|\cdot\|_V$ be a norm on $V$ such that it is a Hilbert space, and that $\operator L,\operator M,\operator B$ are continuous operators.
Assume that $r$, $s$ and $\ell$ are continuous sesquilinear forms and $S_g$ and $L$ are continuous antilinear functionals.
Then the variational problem \eqref{eq:buv=Fv} is well-posed by Lax--Milgram theorem if $b$ is coercive: to ensure this one needs to provide a lower bound for $2 \ell(v,v) + \Re\{s(v,v)-r(v,v)\}$.
It turns out that this approach allows to choose $r$, $s$ and $\ell$ in \eqref{eq:assum:boundary_identity} for both the Helmholtz equation, as we summarise in Remark~\ref{rem:Helmholtz}, and the wave equation, as we describe in the rest of the paper.

\begin{remark}[Coercive formulation for the Helmholtz equation]\label{rem:Helmholtz} \;
	A reason to consider this abstract framework is that we can recover the coercive formulation of \cite{Moiola2014} for the impedance Helmholtz problem
	$$
		\operator{L}u = \Lap u + k^2 u = f \quad \text{in }\Omega, \qquad
		\operator{B}u = \NormDer{u} -\textrm{i} k u = g \quad \text{on }\partial \Omega
	$$
	where $k$ is a positive constant and $\Omega \subset \R^d$ is a bounded star-shaped domain.
	We choose
	$\operator{M}u = \vect{x} \cdot \Grad u - \textrm{i} k\beta u + \frac{d - 1}{2}u$, where $\beta$ is a positive coefficient, and $\ell(u,v) := \frac1{3k^2}
	\int_D \oL u \oL v \dx$ ($\Ll$ is defined accordingly).
	We define the sesquilinear forms $r$ and $s$ as
\begin{align*}
		r(u,v) = & - \int_{\Omega} \left(\Grad u \cdot \conj{\Grad v} + k^2 u \conj{v}\right) \dx                     \\
		         & + \int_{\partial \Omega}\bigg( \textrm{i} k
		u \conj{\operator{M}v} + \left(\vect{x} \cdot \Grad_{\partial \Omega} u - \textrm{i}k\beta u + \frac{d-1}2 u\right) \conj{\NormDer{v}}
+(\vect{x} \cdot \normal) \left(k^2 u \conj{v} - \Grad_{\partial \Omega} u \cdot \conj{\Grad_{\partial \Omega} v}\right)\bigg)\dS,
		\\
		s(u,v) = & \int_{\partial \Omega} \left(\NormDer{u} - \textrm{i}  k u \right) \conj{\operator{M}v} \dS.
	\end{align*}
	where $\Grad_{\partial \Omega}$ is the tangential component of the gradient on $\GI$.
	Note that the assumptions \eqref{eq:assum:splitId}--\eqref{eq:assum:S=Sg} are satisfied.
	This, together with \eqref{eq:b}--\eqref{eq:F}, defines $b$ and $F$:
	\begin{align*}
		b(u,v) = & \int_{\Omega} \left(\Grad u \cdot \conj{\Grad v} + k^2 u \conj{v} + \left(\operator{M}u + \frac{1}{3k^2} \operator{L}u\right) \conj{\operator{L}v}\right) \dx + \\
		         & -  \int_{\partial \Omega} \left(\textrm{i} ku \conj{\operator{M}v} + \Big(\vect{x} \cdot \Gradbound{u} - \textrm{i}                                       k \beta u + \frac{d-1}{2}u\Big) \conj{\NormDer{v}} + (\vect{x} \cdot \normal) \left(k^2 u\conj{v} - \Gradbound{u} \cdot \conj{\Gradbound{v}}\right) \right)\dS,
		\\
		F(v) =   & \int_{\Omega}\left(\conj{\operator{M}v} - \frac{1}{3k^2}\conj{\operator{L}v}\right) f \dx + \int_{\partial \Omega} \conj{\operator{M}v}g \dS,
	\end{align*}
	which coincide with those defined in \cite[eq.~(1.23), (1.24)]{Moiola2014}.
	All forms involved are continuous in the space
	\begin{equation*}
		V = \overline{C^{\infty}(\overline{\Omega})}^{\| \cdot \|_V}
		=\left\{v\in H^1(\Omega) : \;\Lap v \in L^2(\Omega),\; v \in H^1(\partial \Omega),\;
		\NormDer{v} \in L^2(\partial \Omega) \right\}
	\end{equation*}
	(see \cite[eq.~(1.21) and Lemma~A.1]{Moiola2014}),
	equipped with the norm
	\begin{align*}\|v\|^2_V = \  & k^2\|v\|^2_{L^2(\Omega)} + \|\Grad v\|^2_{L^2(\Omega)} + k^{-2}\|\Lap v\|^2_{L^2(\Omega)} \\&
		+ \mathrm{diam}(\Omega) \left(k^2\|v\|^2_{L^2(\partial \Omega)} + \|\Grad_{\partial \Omega} v\|^2_{L^2(\partial \Omega)} + \left\|\NormDer{v}\right\|^2_{L^2(\partial \Omega)} \right).
	\end{align*}
	The main result of \cite{Moiola2014} is that $b$ (and actually a more general version) is coercive in $V$.
	The proof (\cite[Thm.~3.4]{Moiola2014}) essentially relies on controlling $\Re\{s(v,v)-r(v,v)\}$ from below and requires $\Omega$ to be star-shaped with respect to a ball centered at the origin.
\end{remark}

\section{Space--time Morawetz multiplier and identities}
\label{sec:Identities}

We define the space--time Morawetz multiplier $\operator{M}$ as the operator
\begin{equation}\label{eq:multiplier_def}
\operator{M} u : = - \xi \vect{x} \cdot \Grad u + \beta (t - T^*)\TimeDer{u},
\qquad\text{where } \xi, \beta > 0 \text{ and }T^* = \nu T,\; \nu > 1,
\end{equation}
for either $u\in H^1(Q)$ or $u\in C^1(Q)$.

To simplify the notation, we denote the wave operator by $\operator{W}u: = \SecondTimeDer{u} - c^2 \Lap u$.

\begin{remark}[Other multipliers]\label{remark:energy_estimation}
	The definition of the multipliers is problem-dependent.
	In the context of energy estimation for the wave equation, the multiplier used is  different, and is defined as $\operator{Z}u = t\vect{x} \cdot \Grad u + \frac{1}{2}(|\vect{x}|^2 + t^2)\partial_tu + t u$ (see \cite[App.~3]{lax1990scattering} or \cite{Morris2018}). One would then compute $q$ and $\vect{p}$ such that $\operator{Z}u \operator{W}u = \TimeDer{q} + \Div \vect{p}$,  integrate $\TimeDer{q} + \Div \vect{p}$ over the space--time domain, and use the divergence theorem to obtain an estimate of the energy that decays  as $\frac{1}{t}$.
	For the Helmholtz equation the ``Rellich multiplier'' $\operator Mu=\vect x\cdot\Grad u$ is instrumental to prove wavenumber-explicit stability bounds for impedance boundary value problems, \cite[Prop.~8.1.4]{melenk1995generalized}.
	The coercive formulation for the Helmholtz equation in \cite{Moiola2014} relies on the ``Morawetz multiplier'' $\operator{M}v = \vect{x} \cdot \Grad v - \textrm{i}k\beta v + \frac{d-1}{2}v$.
	See \cite[sect.~1.4]{Moiola2014} for a brief summary of the uses of this kind of multipliers for wave problems.
\end{remark}

The Morawetz identity allows to write the sum $\operator{M}u \operator{W}v +  \operator{W}u\operator{M}v$ as the sum of a time derivative, a divergence, and two terms ``with a sign'', i.e.\ whose signs are determined by the parameters $\beta$ and $\xi$ when $v=u$, for all $u\ne0$.
\begin{lemma}[Pointwise Morawetz identity]
	For all $u,v\in C^2(Q)$ the following identity holds:
	\begin{equation}
		\label{eq:diff_identity}
		\begin{split}
			\operator{M}u \operator{W}v + \operator{W}u\operator{M}v
=&\TimeDer{} \Big[
				\operator{M}u \, \TimeDer{v} +\TimeDer{u} \,\operator{M}v
				-\beta (t-T^*)\left(\TimeDer{u} \TimeDer{v} - c^2\Grad u \cdot \Grad v\right)
				\Big] \\
&-\Div \Big[
			c^2{\operator{M}u\, \Grad v} + c^2\Grad u\,\operator{M}v
			- \xi \vect{x} \left( \TimeDer{u}\TimeDer{v} - c^2\Grad u \cdot \Grad v \right)
			\Big] \\
& - \big(\beta + \xi d \big) \TimeDer{u} \TimeDer{v}
-c^2 \big(
			\beta + \xi (2 - d)
			\big) \Grad u \cdot \Grad v.
		\end{split}
	\end{equation}
\end{lemma}
\begin{proof}

The proof of \eqref{eq:diff_identity} only requires the use of some standard vector calculus identities.
Indeed, the definition \eqref{eq:multiplier_def} of $\oM$ and repeated applications of the Leibniz product rule, together with $\Div\vx=d$, give the time component of the identity:
\begin{equation}
		\label{eq:time_component}\begin{split}
			(\operator{M}u) \SecondTimeDer{v}
			+ (\operator{M}v) \SecondTimeDer{u}
			&=
			\TimeDer{} \big[
				(\operator{M} u)\TimeDer{v} +
				\TimeDer{u}(\operator{M} v)
				\big]
			- (\TimeDer\oM u) \TimeDer v
			- \TimeDer u (\TimeDer \oM v) \\ & =
			\TimeDer{} \big[
				(\operator{M} u)\TimeDer{v} + \TimeDer{u}(\operator{M} v)
				\big]
			+ \xi (\vect{x} \cdot \Grad \TimeDer{u})  \TimeDer{v}
			+ \xi \TimeDer{u} (\vect{x} \cdot \Grad \TimeDer{v}) \\ &
			\quad  - 2 \beta \TimeDer{u} \TimeDer{v}
			- \beta (t-T^*)\SecondTimeDer{u} \TimeDer{v}
			- \beta (t-T^*)\TimeDer{u} \SecondTimeDer{v}
			\\ & =
			\TimeDer{} \big[
				(\operator{M} u)\TimeDer{v} + \TimeDer{u}(\operator{M} v)
				\big]
			+ \xi \big(
			\Div [\vx  \TimeDer u \TimeDer v]
			- (\Div \vx)\TimeDer u \TimeDer v
			\big) \\ & \quad
			-\beta\TimeDer{u} \TimeDer{v}
			-\beta \TimeDer \big[
				(t-T^*) \TimeDer u \TimeDer v
				\big]
			\\ & =
			\TimeDer{} \big[
				(\operator{M} u)\TimeDer{v} + \TimeDer{u}(\operator{M} v) - \beta (t-T^*) \TimeDer u \TimeDer v
				\big] \\ & \quad
			+ \Div [\xi \vx  \TimeDer u \TimeDer v]
			- (\beta+\xi d)\TimeDer u \TimeDer v.
		\end{split}
	\end{equation}
The space component of the identity is treated similarly.
	Recalling that the gradient of the scalar product of two gradients is
	$\Grad(\Grad u\cdot\Grad v)=(D^2u)\Grad v+(D^2v) \Grad u$ where $D^2$ denotes the (symmetric) Hessian matrix,
	and that $\vx$ is the gradient of $\frac12|\vx|^2$ whose Hessian is the identity matrix, we have
	\begin{align*}
		\Grad(\vx\cdot\Grad u)\cdot\Grad v+\Grad u\cdot\Grad(\vx\cdot\Grad v)
		= & \ 2\Grad u\cdot\Grad v + \big((D^2u) \vx\big)\cdot\Grad v+\Grad u\cdot\big((D^2 v) \vx\big) \\
		= & \ 2\Grad u\cdot\Grad v + \vx\cdot \Grad(\Grad u\cdot\Grad v)                                \\
		= & \ (2-d)\Grad u\cdot\Grad v + \Div\big[\vx (\Grad u\cdot\Grad v)\big].
	\end{align*}
	Then
	\begin{align}
		\nonumber
		\operator{M}u \Lap v + \Lap u \operator{M}v
		=\  & \Div \big[\operator{M}u \Grad v + \Grad u\operator{M}v \big] - \Grad \oM u \cdot \Grad v- \Grad u \cdot \Grad\oM v     \\
		\nonumber
		=\  & \Div \big[\operator{M}u \Grad v + \Grad u\operator{M}v \big]                                                           \\
		\nonumber
		    & +\xi\big(\Grad(\vx\cdot\Grad u)\cdot\Grad v+\Grad u\cdot\Grad(\vx\cdot\Grad v)\big)                                    \\ &
		\label{eq:space_component}
		- \beta(t-T^*)\big(\Grad u \cdot \Grad \TimeDer{v} + \Grad v \cdot \Grad \TimeDer{u}\big)
		\\
		\nonumber
		=\  & \Div \left[\operator{M}u \Grad v + \operator{M}v \Grad u + \xi \vect{x} \Grad u \cdot \Grad v\right]                   \\
		\nonumber
		    & - \TimeDer{} \left[\beta (t-T^*) \Grad u \cdot \Grad v\right] + \bigl(\beta + \xi (2 - d)\bigr) \Grad u \cdot \Grad v.
	\end{align}
The identity \eqref{eq:diff_identity} is obtained multiplying \eqref{eq:space_component} by $c^2$ and subtracting it from \eqref{eq:time_component}.
\end{proof}

Integrating by parts the identity \eqref{eq:diff_identity}, we write the bilinear form $m(\cdot,\cdot)$ defined in \eqref{eq:m_def} (with $\oL=\oW$) as the sum of an integral over $Q$ of first-order derivatives, and an integral on the boundary $\partial Q$.

\begin{lemma}[Integrated Morawetz identity]\label{lem:integrated_identity}
If $u,v\in H^2(Q)$, then
\begin{equation}
		\label{eq:integrated_identity}
		\begin{split}
			m(u,v) :=& \int_Q \big(\operator{M}u \operator{W}v  + \operator{W}u\operator{M}v \big) \dx \dt \\
			=&- \int_Q \big[(\beta + \xi d)\TimeDer{u}\TimeDer{v} + c^2 (\beta + 2 \xi - d\xi)\Grad u \cdot \Grad v \big] \dx \dt \\
&- \int_{\Omega_T} \big[\xi \vect{x} \cdot ( \TimeDer{u} \Grad v + \Grad u \TimeDer{v})
- \beta (T-T^*)(\TimeDer{u} \TimeDer{v} + c^2 \Grad u \cdot \Grad v)\big] \dx  \\
&+ \int_{\Omega_0} \big[\xi \vect{x} \cdot ( \TimeDer{u}\Grad v + \Grad u \TimeDer{v} )
+\beta T^* (\TimeDer{u} \TimeDer{v} + c^2 \Grad u \cdot \Grad v ) \big] \dx\\
& -\int_\Sigma \big[c^2 \operator{M} u\NormDer{v} + c^2\NormDer{u} \operator{M} v
				+ \xi \vect{x} \cdot \vect{n} (- \TimeDer{u} \TimeDer{v}
				+ c^2 \,\Grad u \cdot \Grad v) \big] \dS\dt.
		\end{split}
	\end{equation}
\end{lemma}
\begin{proof}
	If $u,v\in C^2(\conj Q)$,
then \eqref{eq:integrated_identity} is immediately obtained by integrating the identity \eqref{eq:diff_identity} over $Q$, applying the divergence theorem \cite[Thm.~3.34]{McLean00} in space to the term $\Div[\bullet]$ and the fundamental theorem of calculus to the term $\TimeDer[\bullet]$.

	The assertion follows from the density of $C^2(\conj Q)$ in $H^2(Q)$, e.g.\ \cite[Thm.~3.29]{McLean00}, and the continuity in $H^2(Q)$ in the sense of bilinear forms (using the trace theorem, e.g.\ \cite[Thm.~3.37]{McLean00}) of both the left- and the right-hand side of \eqref{eq:integrated_identity}.
\end{proof}

\section{Interior impedance problem}
\label{sec:interior_problem}

We derive and analyse a space--time variational formulation in the form \eqref{eq:buv=Fv} for the interior impedance problem, i.e.\ for~\eqref{eq:IBVP} with $\GD=\emptyset$.

\subsection{Bilinear and linear form definition}\label{sec:forms}
We follow the abstract approach described in section~\ref{sec:abstract_framework}, taking $\oL=\oW$, $D=Q$ and $\operator M$ as in \eqref{eq:multiplier_def}.
Recalling Remark~\ref{rem:ComplexReal}, from now on all functions are assumed to be real.
The bilinear form $m(\cdot,\cdot)$ in \eqref{eq:m_def} has been computed in \eqref{eq:integrated_identity}.
In order to define $b(\cdot,\cdot)$ and $F(\cdot)$, it is necessary to choose $r(\cdot,\cdot)$, $s(\cdot,\cdot)$ and $\ell(\cdot, \cdot)$ first:
\begin{align}	\label{eq:r_def}
	 & \begin{aligned}
		   r(u,v) := - & \int_Q \left[
			   (\beta + \xi d)\TimeDer{u}\TimeDer{v} + c^2  (\beta + 2 \xi - d\xi)\Grad u \cdot \Grad v
		   \right] \dx \dt                                                                                              \\
+           & \int_{\Omega_T} \big[-\xi \vect x \cdot (  \TimeDer{u} \,\Grad v+ \Grad u\, \TimeDer{v} )
+\beta (T-T^*)(\TimeDer{u} \TimeDer{v} + c^2 \Grad u \cdot \Grad v)\big] \dx                              \\
-           & \int_{\SI} \big[c^2 \operator{M} u\,\NormDer{v} -\frac{c}{\theta} \TimeDer{u}\, \operator{M} v
		   + \xi \,\vect x \cdot \vect{n}\,(-\TimeDer{u} \TimeDer{v}+  c^2\,\Grad u \cdot \Grad v)\big] \dS\dt,   \\
	   \end{aligned}        \\
\label{eq:s_def}
	 & \begin{aligned}
		   s(u,v) := & \int_{\Omega_0} \big[ \xi \vect x \cdot \big( \TimeDer{u} \,\Grad v+ \Grad u \,\TimeDer{v}\big)
		   +\beta T^* \left(\TimeDer{u} \TimeDer{v} + c^2 \Grad u \cdot \Grad v\right)\big] \dx                        \\
& -  \int_{\SI} c^2 \left[(\theta c)^{-1} \TimeDer{u} + \NormDer{u}\right]\operator{M} v  \dS \dt,
	   \end{aligned} \\
	\label{eq:ell_def}
	 & \begin{aligned}
		   \ell(u,v) := \; & A_QT^2\int_Q \oW u \oW v \dx \dt + A_\OZ T^{-1}\int_{\OZ} u v \dx,
	   \end{aligned}
\end{align}
where $A_Q, A_\OZ > 0$.
The first two bilinear forms satisfy assumptions \eqref{eq:assum:splitId}--\eqref{eq:assum:S=Sg}: they sum to $m(\cdot,\cdot)$, and $s(u,v)$ depends on $u$ only through the IBVP \eqref{eq:IBVP} boundary data, i.e.\ $u$ and $\TimeDer u$ on $\OZ$ and $(\theta c)^{-1} \TimeDer{u} + \NormDer{u}$ on ${\SI}$.
Moreover, $\ell$ is positive-semidefinite, satisfies assumption \eqref{eq:assum:ell=L} and depends on $u$ only through the source term and the boundary data, i.e.\ $\oW u$ on $Q$ and $u$ on $\OZ$, with $F_\ell(v):= A_QT^2\int_Q f \oW v \dx \dt + A_\OZ T^{-1}\int_{\OZ} u_0 v \dx$.

The bilinear and linear forms of the variational formulation follow from \eqref{eq:b}--\eqref{eq:F} and \eqref{eq:r_def}--\eqref{eq:ell_def}:
\begin{align}\label{eq:def:bilin}
	 & \begin{aligned}
		   b(u,v) =
		    & \int_Q \big(
		   \operator{M} u \operator{W} v
		   +(\beta + \xi d)\TimeDer{u}\TimeDer{v} + c^2  (\beta + 2\xi - d\xi)\Grad u \cdot \Grad v
		   \big) \dx \dt                                                                                              \\
		    & + \int_{\Omega_T}\big[ \xi \vect{x} \cdot (\TimeDer{u}\,\Grad v + \Grad u\, \TimeDer{v})
- \beta (T-T^*)(\TimeDer{u} \TimeDer{v}	+ c^2 \Grad u \cdot \Grad v)\big] \dx                           \\
& + \int_{\SI} c^2 \operator{M} u\,\NormDer{v} - c \theta^{-1} \TimeDer{u} \operator{M} v
		   + \xi \vect{x} \cdot \vect{n} ( - \TimeDer{u} \TimeDer{v} + c^2 \Grad u \cdot \Grad v )\big] \dS \dt \\
		    & + A_Q T^2 \int_Q \operator{W} u \operator{W} v \dx \dt + A_\OZ T^{-1}\int_\OZ u v \dx,
	   \end{aligned}
	\\ \label{eq:def:lin}
	 & \begin{aligned}
		   F(v) =
		    & \int_Q -f\operator{M} v \dx\dt
		   - \int_{\SI} c^2 \,\gI\,\operator{M} v \dS \dt                                                  \\
		    & + \int_{\Omega_0}\big[\xi \vect{x} \cdot \big( u_1\,\Grad v+ \Grad u_0\, \TimeDer{v}\big)
		   +\beta T^* \left(u_1 \TimeDer{v}+ c^2 \Grad u_0\cdot \Grad v\right)\big]\dx                  \\
		    & +A_Q T^2 \int_Q  f \oW v \dx \dt + A_\OZ T^{-1}\int_\OZ u_0 v \dx.
	   \end{aligned}
\end{align}

\subsection{Norm, function spaces and variational problem}\label{sec:spaces}

In order to write a variational problem and to prove a coercivity result, a function space and an accompanying norm are needed.

We choose a norm that is dimensionally homogeneous and has the minimum number of terms to ensure continuity of $b$ and $F$:
for $v \in C^{\infty}(\overline{Q})$,
\begin{equation}
	\label{eq:norm_def}
	\begin{split}
		\|v\|_V^2 :=& \qquad\; \|\TimeDer{v}\|_{Q}^2 + \quad c^2\|\Grad v\|_Q^2 + T^2 \|\operator{W} v\|_Q^2 \\
		& + T \|\TimeDer{v}\|_{\Omega_T}^2 + c^2 T \|\Grad v\|_{\Omega_T}^2 \\
		& + T \|\TimeDer{v}\|_{\Omega_0}^2 \:+ c^2 T \|\Grad v\|_{\Omega_0}^2 + T^{-1}\|v\|_{\OZ}^2 \\
		& + \LI \|\TimeDer{v}\|_{\SI}^2 + c^2 \LI \|\Grad v\|_{\SI}^2 .
	\end{split}
\end{equation}
The boundary terms appearing in \eqref{eq:norm_def} are the full $L^2$ norms of $\TimeDer v$ and $c\Grad v$ on the whole space--time boundary $\partial Q$, weighted with the domain sizes $\LI$ and $T$.
In particular, the $\|\Grad v\|_{\SI}$ term involves both the tangential and the normal components of the space gradient of $v$.

Note that $\|\cdot\|_V$ does not explicitly contain the $L^2(Q)$ norm of the function $v$, but only the $L^2(\OZ)$ norm.
However, when combined with the time derivative over $Q$, this is sufficient to give an upper bound on the $L^2(Q)$ norm:
\begin{align}\nonumber
	\|v\|_Q^2=
	\int_Q\Big(v(\vx,0)+\int_0^t \TimeDer v(\vx,s) \mathrm ds\Big)^2\dx\dt
& \leq 2 \int_Q v^2(\vx, 0) \dx \dt +  2 \int_Q \left( \int_0^t \TimeDer v (\vx,s)\mathrm{d}s \right)^2 \dx \dt \\
    & \leq 2T \|v\|_\OZ^2 + T^2\|\TimeDer v\|_Q^2
	\;\leq\; 2 T^2 \|v\|_V^2.
	\label{eq:L2V}
\end{align}

The space $V$ of test and trial functions is defined as the closure in the $\|\cdot\|_V$ norm of the smooth function space:
\begin{equation}\label{eq:space_def}
	V := \overline{C^\infty(\overline Q)}^{\|\cdot\|_V}.
\end{equation}

Following \S\ref{sec:abstract_framework}, we pose the variational problem:
\begin{equation}   \label{eq:vpV}
	\text{Find $u \in V$ such that}\qquad     b(u,v) = F(v) \qquad \forall v \in V,
\end{equation}
where $b(\cdot,\cdot)$ and $F(\cdot)$ are as in \eqref{eq:def:bilin} and \eqref{eq:def:lin}.

\begin{proposition}[Consistency]\label{prop:EquivLadyzVP}
	If $u \in V$ is a generalised solution of the interior impedance problem in the sense of Definition~\ref{def:imp_gsol}, then it is also a solution of \eqref{eq:vpV}.
\end{proposition}
\begin{proof}
	The operators $\operator{M,L}:V\to L^2(Q)$ and the bilinear form at the right-hand side of the integrated identity \eqref{eq:integrated_identity} are continuous in the norm $\|\cdot\|_V$. Thus, the proof of Lemma~\ref{lem:integrated_identity} shows that all
	$u,v\in V$ satisfy \eqref{eq:integrated_identity}.
	If $u\in V$ is a generalised solution in the sense of Definition~\ref{def:imp_gsol}, the manipulations in \S\ref{sec:forms} lead to the relation $b(u,v)=F(v)$ for all $v\in V$.
\end{proof}

\begin{remark}[The space $W$]\label{remark:V=W?}
We introduce a Hilbert space $W$ that is (possibly) larger and defined more explicitly than the space $V$ in \eqref{eq:space_def}, which instead is defined as the completion of a space of smooth functions.
In this remark we temporarily write explicitly all trace operators.
First let
	\begin{align*}
		H(Q; \oW) :=\; & \left\{u \in H^1(Q): \oW u \in L^2(Q)\right\}.
\end{align*}
For all $v \in H(Q;\oW)$ the Dirichlet trace $\gamma_Q v$ is well-defined since $\gamma_Q:H^1(Q)\to H^{\frac12}(\partial Q)$, \cite[Thm.~3.37]{McLean00}; in particular  $(\gamma_Q v)|_D \in L^2(D)$ for $D \in \{\OZ, \OT, \SI\}$.
	We show that also the Neumann trace can be defined on $H(Q;\oW)$ using integration by parts.
	Indeed, for $v \in H(Q;\oW)$, define $D_n u\in H^{-\frac12}(\partial Q)=(H^{\frac12}(\partial Q))^*$ by
\begin{equation*}
		\langle D_n u , \varphi \rangle_{H^{-\frac12}\times H^{\frac12}} :=
		\int_Q (\TimeDer v \TimeDer \phi - c^2 \Grad v \cdot \Grad \phi) \dx \dt - \int_Q \oW v \phi \dx \dt \qquad \forall \varphi \in H^{\frac12}(\partial Q),
	\end{equation*}
	where $\phi$ is any element of $\phi \in H^1(Q)$ such that $\varphi = \gamma_Q \phi$.
	Exactly as in \cite[Lemma~4.3]{McLean00}, one can prove that $D_n u$ is the Neumann trace of $v$ and is well-defined on $\partial Q$, in particular its action on $\varphi$ is independent of the choice of the lifting $\phi$.
	To define the desired space, we require further regularity on both traces:
$$
		W:=\Big\{v\in H^1(Q;\oW):\; (\gamma_Qv)|_\OZ\in H^1(\OZ),\; (\gamma_Qv)|_\OT\in H^1(\OT),
		\; (\gamma_Qv)|_\SI\in H^1(\SI),
		\; D_n v\in L^2(\partial Q)\Big\}.
	$$
Since $L^2(\partial Q)$ functions can be split over the different parts of $\partial Q$, differently from $H^{-\frac12}(\partial Q)$ distributions, and since $(\gamma_Qv)|_\SI\in H^1(\SI)$ is equivalent to the $L^2(\SI)$ regularity of $\TimeDer v$ and of the tangential component of $\nabla v$,
	we can characterise $W$ (neglecting the trace notation) as
	\begin{align*}
		W & =\Big\{v\in H^1(Q):\;\oW v\in L^2(Q),\; \Grad v,\TimeDer v\in L^2(\partial Q)\Big\} \\
		  & =\Big\{v \in H^1(Q): \|v\|_V < \infty \Big\}.
	\end{align*}
Then, the spaces involved are related by the following dense and closed inclusions:
$$ H^2(Q) \overset{\text{dense}}\subset V\overset{\text{closed}}\subset W\overset{\text{dense}}\subset H^1(Q). $$
The important question of whether $V = W$, i.e.\ whether $C^\infty(\conj Q)$ is dense in $W$, is still unclear and is the subject of ongoing research.
The density of smooth function in a related space--time graph space for a wave operator with $L^2$ traces has recently been studied in \cite[Thm.~5.10]{BerggrenHagg2021}; see also \cite[Thm.~3.6]{FuhrerGonzalezKarkulik2025} and \cite[Thm.~4.1]{GopalakrishnanSepulveda2019} for spaces with homogeneous traces.
\end{remark}
Both $V$ and $W$ are Hilbert spaces with norm $\|\cdot\|_V$, so, together with \eqref{eq:vpV}, we can consider the more general problem:
\begin{equation}\label{eq:vpW}
	\text{Find $u \in W$ such that}\qquad     b(u,v) = F(v) \qquad \forall v \in W.
\end{equation}
We prove in \S\ref{sec:CC} that \eqref{eq:vpV} and \eqref{eq:vpW} are well-posed when $\Omega$ is star-shaped and the parameters in $b$ are chosen judiciously.
Under these assumptions, which of the variational problems \eqref{eq:vpV} and \eqref{eq:vpW} gives the generalised solution of Definition~\ref{def:imp_gsol}?
We can observe three situations:
\begin{itemize}
\item If the IBVP data only have the low regularity stated in Definition~\ref{def:imp_gsol}, then the generalised solution is not guaranteed to be in $W$ (in particular the traces on $\SI$ and $\OT$ may not be sufficiently smooth).
Thus the solutions to \eqref{eq:vpV} and \eqref{eq:vpW} are not necessarily related to \eqref{eq:IBVP}.

\item If $u\in W\setminus V$---which we do not know if it is possible at all---then, in principle, the solutions of \eqref{eq:IBVP} and \eqref{eq:vpW} might differ because Proposition~\ref{prop:EquivLadyzVP} only holds for $u\in V$.
    Then the solution of \eqref{eq:vpV} is a projection on $V$ (through the bilinear form $b$) of the solution of \eqref{eq:vpW}.
    We expect that a regularity analysis more refined than that in Theorem~\ref{thm:t_regularity} can ensure $u\in W$ under assumptions on the data weaker than \eqref{eq:cond_reg0}--\eqref{eq:compatibility}.

\item If the generalised solution $u$ belongs to $V$, then $u$ is solution of both variational problems \eqref{eq:vpV} and \eqref{eq:vpW} (if the conditions for their well-posedness in \S\ref{sec:CC} are met).
This is ensured, e.g.\ if the data regularity and compatibility assumptions \eqref{eq:cond_reg0} and \eqref{eq:compatibility} hold,
and $\Omega$ has a $C^{1,1}$ boundary, since then $u$ belongs to $H^2(Q)$ and thus to $V$ by Theorem~\ref{thm:t_regularity}, and is a strong solution of \eqref{eq:IBVP} almost everywhere.
\end{itemize}

\subsection{Coercivity and continuity}
\label{sec:CC}
The main result of this paper is the coercivity of the bilinear form \eqref{eq:def:bilin}.
This is proved in Theorem~\ref{thm:coerc_b} using only the identities of \S\ref{sec:abstract_framework}--\ref{sec:Identities}--\ref{sec:forms} and elementary inequalities.
The coercivity allows to use Lax--Milgram theorem and obtain well-posedness of the variational problems \eqref{eq:vpV}--\eqref{eq:vpW}.

The key geometrical assumption is the following.
\begin{assumption}[Star-shaped $\Omega$]\label{ass:StarS}
	There exists $\deltaI > 0 $ such that $\vect{x}\cdot\vect{n} \geq \deltaI \LI$ for almost every $\vect x\in \GI$.
\end{assumption}
Assumption~\ref{ass:StarS} is equivalent to the requirement that the space domain $\Omega$ is star-shaped with respect to the ball $B_{\deltaI\LI}$ of radius $\deltaI \LI$ centred at the origin, \cite[Rem.~3.5]{Moiola2014} (i.e.\ all straight segments between any point in $B_{\deltaI\LI}$
and any point in $\Omega$ are contained in $\Omega$).
Recall that $\Omega$ is Lipschitz, so the outward-pointing unit vector field $\vect n$ is defined almost everywhere on $\GI$.

The coercivity constant of $b(\cdot,\cdot)$ and the conditions under which coercivity holds depend explicitly on several parameters related to the domain $Q$
(the space dimension $d$,
the final time $T$,
the space diameter $\LI$,
the star-shapedness parameter $\deltaI$),
the IBVP \eqref{eq:IBVP}
(the wave speed $c$,
the impedance parameter $\theta$),
and the formulation \eqref{eq:vpV}
(the penalty parameters $A_Q$, $A_\OZ$,
the parameters $\xi$, $\beta$, $\nu=T^*/T$ entering the multiplier $\oM$ in \eqref{eq:multiplier_def}).

\begin{theorem}[Coercivity of $b$]\label{thm:coerc_b}
	Under Assumption~\ref{ass:StarS} on $\Omega$, if $\xi > 0$, $A_Q, A_\OZ > 0$ and
	\begin{equation}
		\label{eq:coefficient_coercivity_cond}
		\beta \geq \max
		\begin{cases}
			\xi(d-1),                                               \\
			\xi\frac{1}{\nu - 1}\left(\frac{\LI}{cT}+ 1\right), \\
			\xi\frac{1}{\nu - 1}\frac{\LI}{cT}\left(\theta + \frac1{\deltaI \theta}\right),
		\end{cases}
	\end{equation}
	then for all $u \in W$
	\begin{equation} \label{eq:coercivity_condition}
		b(u,u) \geq {\alpha_b} \|u\|_V^2
		\qquad \text{with}\qquad
		\alpha_b := \min \left\{\frac{\xi\deltaI}{4}, \; A_Q, \; A_\OZ\right\}.
	\end{equation}
\end{theorem}
\begin{proof}
	Let $u \in W$.
	We recall from \eqref{eq:abstract_coercivity_condition} that, in order to prove the coercivity of $b(\cdot,\cdot)$, it is sufficient to show a lower bound for $-r(u,u) + s(u,u) + 2\ell(u,u)$. For simplicity we can find a lower bound for $-r(u,u) + s(u,u)$ and then improve on it by including the terms in $\ell(u,u)$.
	From \eqref{eq:r_def}--\eqref{eq:s_def}, simple calculations lead to
	\begin{align}\nonumber
		-r(u,u) + s(u,u) & =
		\int_{Q} \big[
			(\beta + \xi d) (\TimeDer{u})^2
			+ c^2 ( \beta + 2\xi - d\xi)|\Grad u|^2
		\big] \dx \dt                                           \\ \nonumber
		                 & + \int_{\OT} \big[
			2 \xi (\vect{x} \cdot \Grad u) \TimeDer{u}
			- \beta (T-T^*)\left(
			(\TimeDer{u})^2
			+ c^2 |\Grad u|^2
		\right) \big] \dx                                       \\
		                 & + \int_{\OZ}
		\big[\beta T^* \left(
			(\TimeDer{u})^2
			+ c^2 |\Grad u|^2\right)
			+ 2 \xi (\vect{x} \cdot \Grad u) \TimeDer{u} \big]\dx
		\label{eq:coercivity_step1}                             \\
		                 & + \int_{{\SI}} \left[
			2 \frac c\theta\xi (\vect{x} \cdot \Grad u) \TimeDer{u}
+ c^2 \xi (\vect{x} \cdot \vect{n}) |\Grad u|^2
			- \left(\xi (\vect{x} \cdot \vect{n})
			+ 2\frac{c}{\theta}{\beta} (t-T^*)\right) (\TimeDer{u})^2
			\right] \dS \dt
		\nonumber
		\\
		                 & =: I_Q+I_\OT+I_\OZ+I_{\SI}.\nonumber
	\end{align}
	The mixed terms containing $\Grad u \TimeDer{u}$ can be lower-bounded by negative terms using the H\"older inequality, and then the products of the norms of $\Grad u$ and $\TimeDer{u}$ (on $\OZ$, $\OT$ and ${\SI}$) can be controlled by the weighted Young inequality:
\begin{align*}
		\int_{D}2  (\vect{x} \cdot \Grad u) \TimeDer{u} \ge - 2 \LI \norm{\Grad u}_D \norm{\TimeDer{u}}_D
		 & \geq -\epsilon_D \LI \norm{\Grad u}^2_D -\frac{1}{\epsilon_D} \LI \norm{\TimeDer{u}}^2_D,
		\qquad D\in\{\Omega_0,\Omega_T,{\SI}\},\; \epsilon_D>0.
	\end{align*}
Here we also used that $\LI = \max\{|\vx\|: x\in\Omega\}$.
	Combining these inequalities into \eqref{eq:coercivity_step1} it is possible to determine conditions for the coefficients $\beta$, $\xi$, $T^*$.
	We tackle the different integrals independently.

	\textit{Integrals over $\OT$ and over $\OZ$.}
	Recalling that $T^* = \nu T$, the integrals over $\OT$ and $\OZ$ are bounded from below as follows:
	\begin{align*}
		 & I_\OT\geq \left(
		\beta (\nu-1) - \frac{\xi}{\epsilon_{\Omega_T}}\frac{\LI}{T}
		\right) T \norm{\TimeDer{u}}^2_{\OT}
		+ \left(\beta (\nu-1) - \epsilon_{\Omega_T} \xi \frac{\LI}{c^2 T}
		\right) c^2 T \norm{\Grad u}^2_{\OT},
		\\
		 & I_\OZ\geq \left(
		\beta \nu - \frac{\xi}{\epsilon_{\Omega_0}} \frac{\LI}{T}
		\right) T \norm{\TimeDer{u}}^2_{\OZ}
		+ \left(
		\beta \nu - \epsilon_{\Omega_0} \xi \frac{\LI}{c^2 T}
		\right) c^2 T \norm{\Grad u}^2_{\OZ}.
	\end{align*}
	Taking $\epsilon_{\Omega_T} = \epsilon_{\Omega_0} = c$, the coefficients in the two brackets of each inequality are dimensionless and coincide.
From the assumption \eqref{eq:coefficient_coercivity_cond} on $\beta$,
\begin{equation*}
		\beta \nu - \xi \frac{\LI}{cT} \geq \beta (\nu - 1) - \xi \frac{\LI}{cT} \geq \xi > 0,
	\end{equation*}
	thus
	\begin{equation}\label{eq:coerc:OTOZ}
		I_\OT\geq \xi T \norm{\TimeDer{u}}^2_\OT+\xi c^2 T \norm{\Grad u}^2_\OT,\qquad
		I_\OZ\geq \xi T \norm{\TimeDer{u}}^2_\OZ+\xi c^2 T \norm{\Grad u}^2_\OZ.
	\end{equation}

	\textit{Integral over ${\SI}$.}
	Using the same strategy as for the other integrals and the fact that $\vect{x} \cdot \vect{n} \geq \deltaI \LI$ for all $\vect{x} \in \partial {\Omega}$ by Assumption~\ref{ass:StarS}, it holds that
	\begin{align*}
		I_{\SI} \geq \left(
		\xi \deltaI - \frac{\epsilon_\SI \xi}{\theta} \frac{1}{c}
		\right) c^2 \LI \norm{\Grad u}^2_{{\SI}}
		+ \left(
		\frac2\theta \beta (\nu - 1) \frac{cT}{\LI} - \xi \Big(1 + \frac{c}{\theta \epsilon_{\SI}}\Big)
		\right) \LI \norm{\TimeDer{u}}_{{\SI}}^2.
	\end{align*}
	We take $\epsilon_{\SI} = \frac{\theta \deltaI c}{2}$, so that $ \xi \deltaI - \frac{\epsilon_{\SI}}{\theta c} \xi = \xi\deltaI/2 > 0$. From the the last condition \eqref{eq:coefficient_coercivity_cond} on $\beta$,
	\begin{equation*}
			\frac2\theta  \beta (\nu - 1) \frac{cT}{\LI} - \xi \left(1 + \frac{2}{\deltaI \theta^2} \right)
\geq \xi
	\end{equation*}
	thus
	\begin{equation}\label{eq:coerc:S}
		I_{\SI}\ge \frac{\xi\deltaI}2 c^2 \LI \norm{\Grad u}^2_{{\SI}}  + \xi \LI \norm{\TimeDer{u}}_{{\SI}}^2.
	\end{equation}

	\textit{Integral over ${Q}$.}
	Using again the assumption \eqref{eq:coefficient_coercivity_cond},
	\begin{equation}\label{eq:coerc:Q}
		I_Q \ge \xi\norm{\TimeDer{u}}^2_{{Q}} +\xi c^2\norm{\Grad u}^2_{{Q}}.
	\end{equation}

	\textit{Coercivity constant.}
To obtain the coercivity bound, we recall \eqref{eq:abstract_coercivity_condition}, sum the bounds \eqref{eq:coerc:OTOZ}, \eqref{eq:coerc:S}, and \eqref{eq:coerc:Q}, use that $\deltaI\le1$ and the definition \eqref{eq:norm_def} of the norm $\|\cdot\|_V$:
\begin{align*}
		b(u,u)
		 & = \ell(u,u) -\frac12 r(u,u)+\frac12s(u,u)                                                                      \\
		 & =A_Q T^2\norm{\oW u}_Q^2 + A_\OZ\frac1T\|u\|_\OZ^2+\frac12( I_Q+I_\OT+I_\OZ+I_\SI)                             \\
		 & \ge A_Q T^2\norm{\oW u}_Q^2 + A_\OZ\frac1T\|u\|_\OZ^2+\frac\xi2\Big(
		T \norm{\TimeDer{u}}_{\OT}^2
		+ c^2 T \norm{\Grad u}_{\OT}^2
		+ T \norm{\TimeDer{u}}_{\OZ}^2
		+ c^2 T \norm{\Grad u}_{\OZ}^2                                                                                    \\
		 & \hspace{50mm} +  \LI \norm{\TimeDer{u}}_{{\SI}}^2 + \frac{\deltaI}{2} c^2 \LI \norm{\Grad u}_{{\SI}}^2
		+ \norm{\TimeDer{u}}_{{Q}}^2 + c^2 \norm{\Grad u}_{{Q}}^2\bigg)                                                   \\
		 & \ge \min\Big\{\frac{\xi\deltaI}4, A_Q, A_\OZ\Big\} \norm{u}_V^2.
	\end{align*}
\end{proof}

To apply Lax--Milgram theorem we need the boundedness of the linear and bilinear forms $F(\cdot)$ and $b(\cdot,\cdot)$ in $\|\cdot\|_V$.
These two conditions are clearly satisfied, but in order to estimate the quasi-optimality constant of the Galerkin method, the continuity constant of $b(\cdot,\cdot)$ is needed.

\begin{proposition}[Continuity of $b$ and $F$]\label{prop:continuity_b}
	For all $u,v\in W$,
\begin{equation}\label{eq:boundedness_bilinear}
		|b(u,v)| \leq C_b \|u\|_V \|v\|_V
		\qquad \text{with}\qquad
		C_b := \sqrt{3} \max
		\begin{cases}
			\beta + \xi d + \beta \nu,                                             \\
			\xi \frac{\LI}{cT} + \beta + 2 \xi - d \xi,                        \\
			\beta  (\nu-1) + \xi \frac{\LI}{cT},                               \\
			\big(\frac1\theta +1\big) \big(\beta\nu \frac{cT}{\LI} + \xi\big), \\
			2 \xi,                                                                 \\
			A_Q,                                                                   \\
			A_\OZ,
		\end{cases}
	\end{equation}
	and
	\begin{equation*}
|F(v)|\leq C_F \|(f, u_0, u_1, \gI)\|_{\mathtt{d}}\;\|v\|_V
\qquad \text{with}\qquad
C_F:=\max
\begin{cases}
		\xi \frac{\LI}{c T}+ \beta \nu + A_Q,\\
		\xi + \beta \nu \frac{c T}{\LI},\\
		A_\OZ,\\
		\sqrt2 \Big(\beta \nu + \xi\frac{\LI}{cT}\Big),
\end{cases}
	\end{equation*}
	where the data of the IBVP \eqref{eq:IBVP} are measured in the norm
	\begin{equation}\label{eq:dnorm}
		\|(f, u_0, u_1, \gI)\|_{\mathtt{d}}^2
		:= T^2 \|f\|_Q^2  + T^{-1} \|u_0\|_\OZ^2 + T\|u_1\|_\OZ^2 + c^2T\|\Grad u_0\|_\OZ^2+ c^2\LI\|\gI\|_\SI^2.
	\end{equation}
\end{proposition}
\begin{proof}
	The proof follows the same argument of \cite[Thm.~4.5]{Moiola2014}.
	For $w\in W$, define
	\begin{equation*}\begin{split}
			\vect m(w) = \Big(
			&\|\TimeDer{w}\|_Q,\;
			c \|\Grad w \|_Q,\;
			T\|\SecondTimeDer w - c^2 \Lap w\|_Q,\;
			T^{\frac{1}{2}} \|\TimeDer w\|_\OT,\;
			c T^{\frac{1}{2}} \|\Grad w\|_\OT, \\
			& \, T^{-\frac12} \|w\|_\OZ,\;
			T^{\frac{1}{2}} \|\TimeDer w\|_\OZ,\;
			c T^{\frac{1}{2}} \|\Grad w\|_\OZ,\;
			\LI^{\frac{1}{2}} \|\TimeDer w\|_\SI,\;
			c \LI^{\frac{1}{2}} \|\Grad w\|_\SI
			\Big)^\top \in \R^{10},
		\end{split}
	\end{equation*}
	so that $\|\vect m(w)\|_2=\|w\|_V$.
	Let
	$$
		\begingroup
		\setlength\arraycolsep{2pt}
		\mat M = \begin{pmatrix}
			\mat M_Q &            &            &            \\
			         & \mat M_\OT &            &            \\
			         &            & \mat M_\OZ &            \\
			         &            &            & \mat M_\SI
		\end{pmatrix}
		\endgroup
	$$
	and
	\begin{alignat*}{2}
		 & \begingroup
		\setlength{\arraycolsep}{4pt}
		\mat M_Q = \begin{pmatrix}
			           \beta + \xi d & 0                       & 0   \\
			           0             & \beta + 2 \xi - d \xi   & 0   \\
			           \beta \nu     & \xi \frac{\LI}{c T} & A_Q
		           \end{pmatrix},
		\endgroup \qquad
&             & \begingroup
		\renewcommand{\arraystretch}{1.2}
		\mat M_\OT = \begin{pmatrix}
			             \beta (\nu-1)          & \xi\frac{ \LI}{c T} \\
			             \xi \frac{\LI}{cT} & \beta (\nu-1)
		             \end{pmatrix},
		\endgroup                    \\
& \begingroup
		\setlength{\arraycolsep}{5pt}
		\mat M_\OZ = \begin{pmatrix}
			             A_\OZ & 0 & 0 \\
			             0     & 0 & 0 \\
			             0     & 0 & 0
		             \end{pmatrix},
		\endgroup \qquad
&             & \begingroup
		\renewcommand{\arraystretch}{1.5}
		\setlength{\arraycolsep}{8pt}
		\mat M_\SI = \begin{pmatrix}
			             \frac{\beta \nu}{\theta} \frac{c T}{\LI} + \xi  & 0     \\
			             \beta \nu \frac{c T}{ \LI} + \frac{\xi}{\theta} & 2 \xi
		             \end{pmatrix}.
		\endgroup
	\end{alignat*}
	Then, from the definition \eqref{eq:def:bilin} of $b(\cdot,\cdot)$, it holds that
	\begin{align*}
		|b(u,v)| & \leq \left|\vect m(v)^\top \mat M \vect m(u) \right| \leq \|\mat M\|_2 \|\vect m(v)\|_2 \|\vect m(u)\|_2 \\
		         & \leq \max\big(
		\|\mat M_Q\|_2,
		\|\mat M_\OT\|_2,
		\|\mat M_\OZ\|_2,
		\|\mat M_\SI\|_2
		\big) \|u\|_V \|v\|_V                                                                                               \\
		         & \leq \sqrt{3} \max\big(
		\|\mat M_Q\|_1,
		\|\mat M_\OT\|_1,
		\|\mat M_\OZ\|_1,
		\|\mat M_\SI\|_1
		\big) \|u\|_V \|v\|_V.
	\end{align*}
	The proof of \eqref{eq:boundedness_bilinear} is complete observing that
	\begin{align*}
		 & \big\|\mat M_Q\big\|_1 = \max \Big\{
		\beta + \xi d + \beta \nu,\;
		\xi \LI (c T)^{-1} + \beta + 2\xi - d \xi,\;
		A_Q\Big\},                                                           \\
		 & \big\|\mat M_\OT\big\|_1 = \beta(\nu-1) + \xi \LI (c T)^{-1}, \\
		 & \big\|\mat M_\OZ\big\|_1 = A_\OZ,                                 \\
		 & \big\|\mat M_\SI\big\|_1 = \max \Big\{
		2\xi,\;
		( \theta^{-1} + 1) (\beta \nu c T \LI^{-1} + \xi )
		\Big\}.
	\end{align*}

	In order to control $F(v)$, let
	\begin{align*}
		\vect F_F = \Big(
		 & \beta \nu T\|f\|_Q, \;
		\xi \LI c^{-1} \|f\|_Q, \;
		A_Q T \|f\|_Q, \;
		0, \;
		0, \;
		A_\OZ T^{-\frac12} \|u_0\|_\OZ, \;
		\beta\nu T^{\frac{1}{2}}  \|u_1\|_\OZ + \xi \LI T^{-\frac12}\|\Grad u_0\|_\OZ,             \\
		 & \beta \nu cT^{\frac{1}{2}} \|\Grad u_0\|_\OZ + \xi \LI(c T^\frac12)^{-1}\|u_1\|_\OZ, \;
		\beta \nu c T \LI^{-\frac{1}{2}} \|\gI\|_\SI, \;
		\xi c\LI^{\frac{1}{2}} \|\gI\|_\SI \;
		\Big)^\top\in\R^{10},
	\end{align*}
	from which
	\begin{equation*}
		|F(v)| \leq \vect F_F \cdot \vect m(v)
		\leq \|\vect F_F\|_2 \|\vect m(v)\|_2 =
		\|\vect F_F\|_2 \|v\|_V,
	\end{equation*}
	and
	\begin{align*}
		\|\vect F_F\|_2^2 = \  & \left(\xi^2 \LI^2(c T)^{-2} + \beta^2 \nu^2 + A_Q^2\right)T^2\|f\|_Q^2          \\
		                       & + A_\OZ^2 T^{-1} \|u_0\|_\OZ^2
		+ 2 \big(\beta^2 \nu^2 + \xi^2 \LI^2 (c T)^{-2}\big) T \|u_1\|_\OZ^2                                     \\
		                       & + 2 \big(\beta^2 \nu^2 + \xi^2 \LI^2 (c T)^{-2} \big) c^2 T \|\Grad u_0\|_\OZ^2
		+ c^2\LI\left(\xi^2 + \beta^2\nu^2 c^2T^2\LI^{-2}\right)  \|\gI\|_\SI^2
		\\
		\leq                   & \max\bigg\{
		\xi^2\frac{\LI^2}{c^2T^2} + \beta^2 \nu^2 + A_Q^2,\;
		A_\OZ^2, \;
		2 \Big(\beta^2 \nu^2 + \xi^2 \frac{\LI^2}{c^2T^2}\Big), \;
		\xi^2 + \beta^2 \nu^2 \frac{c^2T^2}{\LI^2}
		\bigg\} \|(f, u_0, u_1, \gI)\|_{\mathtt{d}}^2.
	\end{align*}
\end{proof}

Lax--Milgram theorem \cite[Thm.~2.7.7]{BrennerScottRidgway2008}, together with Theorem~\ref{thm:coerc_b} and Proposition~\ref{prop:continuity_b}, immediately gives the well-posedness of the space--time variational problems.
\begin{corollary}[Stability]
	Let $\Omega$ satisfy Assumption \ref{ass:StarS},
	$\xi>0$, $A_Q>0$, $A_\OZ>0$, $\nu>1$,
	and $\beta$ satisfy \eqref{eq:coefficient_coercivity_cond}.
	Then each variational problem \eqref{eq:vpV} and \eqref{eq:vpW} admits a unique solution, which satisfies the bound
	\begin{align*}
		\|u\|_V \leq
		 & \max\left\{
			\xi \frac{\LI}{c T}+ \beta \nu + A_Q,\;
			\xi + \beta \nu \frac{c T}{\LI},\;
			A_\OZ,\;
			\sqrt2 \Big(\xi\frac{\LI}{cT} + \beta \nu \Big)
	 \right\}  \\
		 & \cdot\max\left\{\frac1{A_Q},\;\frac1{A_\OZ},\; \frac4{\xi\deltaI}\right\}\|(f, \gI,u_0,u_1)\|_{\mathtt{d}}.
	\end{align*}
\end{corollary}
Recall that the data norm $\|\cdot\|_{\mathtt d}$ has been defined in \eqref{eq:dnorm}, and that the solutions of \eqref{eq:vpV} and \eqref{eq:vpW} coincide at least when the conditions in Theorem~\ref{thm:t_regularity} are met.

\subsection{Numerical implications}
\label{sec:NumImplications}
C\'ea's lemma \cite[Thm.~2.8.1]{BrennerScottRidgway2008}, together with Theorem~\ref{thm:coerc_b} and Proposition~\ref{prop:continuity_b}, gives the quasi-optimality of all conforming Galerkin discretisations of \eqref{eq:vpV}.
\begin{proposition}[Galerkin quasi-optimality]\label{lemma:cea}
Let $\Omega$ satisfy Assumption \ref{ass:StarS} and
\begin{equation}\label{eq:ParamChoice}
		\xi=1, \qquad \nu=2, \qquad \beta\ge\beta^\#:=\max\Big\{
		d-1,\;
		1+\frac{\LI}{cT},\;
		\big(\theta+(\theta\deltaI)^{-1}\big)\frac{\LI}{cT}
\Big\}.
	\end{equation}
Let $u \in V$ be the solution of \eqref{eq:vpV} and let $V_N \subset V$ be a closed subspace.
	Then the Galerkin method
	\begin{equation}\label{eq:Galerkin}
		\text{Find }u_N\in V_N \qquad\text{such that}\qquad b(u_N, v_N) = F(v_N) \qquad \forall v_N \in V_N
	\end{equation}
	is well-posed and its solution satisfies
	\begin{equation}
		\label{eq:QOest}
		\|u-u_N\|_V \leq C_{qo} \inf_{v_N \in V_N} \|u - v_N\|_V
	\end{equation}
	with
	\begin{equation}\label{eq:Cqo}
		C_{qo} = \sqrt3\max\left\{A_Q^{-1},A_\OZ^{-1},4\deltaI^{-1}\right\}\max\left\{A_Q,\;A_\OZ,\;
		3\beta+d,\;(1+\theta^{-1})\Big(1+2\beta\frac{cT}\LI\Big)
\right\}.
	\end{equation}
\end{proposition}
As shown in the numerical tests
of section~\ref{sec:numExperiments}, tuning the parameters entering the definition of $b$ can be beneficial for the accuracy of a numerical method, but the choice made in \eqref{eq:ParamChoice} is a good compromise between simplicity and performance of the method.

With a sensible choice of the numerical parameters (e.g.\ $A_Q=A_\OZ=\xi=\nu-1=1$ and $\beta=\beta^\#$), the quasi-optimality constant $C_{qo}$ grows linearly both in the ``length'' and the ``width'' of the space--time cylinder $Q$, measured by the dimensionless ratios $\frac{cT}\LI$ and $\frac\LI{cT}$.

\begin{remark}[$V$-conforming discrete spaces]
Which finite element spaces $V_N$ are conforming in $V$?
Since $H^2(Q)\subset V$, any $H^2(Q)$-conforming space can be used to discretise~\eqref{eq:vpV}.

In particular, assume we are given a non-overlapping partition of $Q$ in Lipschitz $(d+1)$-dimensional elements.
Then, element-wise smooth and globally $C^1(\overline Q)$ elements are $H^2(Q)$-conforming \cite[Thm.~II.5.2]{Braess2007}, and therefore also $V$-conforming.

From the quasi-optimality and the density of the inclusion $H^2(Q)\subset V$, any sequence $\{V_N\}_{N\in\N}$ of $H^2(Q)$-conforming discrete spaces that is able to approximate any $v\in H^2(Q)$ gives a sequence $\{u_N\}_{N\in\N}$ of Galerkin solutions that converges to $u$, if this belongs to $V$.
\end{remark}

Are $C^1$ elements necessary though? Is it possible to use less regular elements? The following lemma clarifies that this is \textit{not} possible.
\begin{lemma}[$C^1$-conformity]
	\label{lemma:C1conformity}
	Let $\mathcal T$ be a space--time triangulation of $Q$. If $v \in V$ and $\forall K \in \mathcal T$ $v|_{\overline{K}} \in C^2(\overline{K})$, then $v \in C^1(\overline{Q}) \cap H^2(Q)$.
\end{lemma}
\begin{proof}
	The proof is the same as in \cite[Lemma~5.1]{Moiola2014}, using the Green's identity associated to the wave operator.
Let $v$ be as in the statement, then, applying Green's identities in space--time,
$$
	0=\int_Q (\phi \oW v - v \oW \phi) \dx\dt
	 =\sum_{K\in \mathcal T}\int_K (\phi \oW v - v \oW \phi) \dx\dt
	 =\sum_{K\in \mathcal T} \int_{\partial K} \left(\phi \partial_{\widetilde{\vn}} v - v \partial_{\widetilde{\vn}}\phi \right) \dS
	 \qquad\forall\phi \in \mathcal D(Q),
	$$
	where $\widetilde{\vn} = (-c^2\vn_{\vx}, n_t)$, and $(\vn_{\vx}, n_t)$ is the outward-pointing normal vector to $K$.
	Let $\mathcal F$ be the set of the facets of the triangulation.
	For every interior facet $F$, let $K^1_F$ and $K^2_F$ be the elements that share $F$.
	The above identity becomes
\begin{equation}
		\label{eq:femJumps}
		\sum_{F\in\mathcal F,\; F\subset Q} \int_F \left(\phi \partial_{\widetilde{\vn}}v|_{K^1_F} - \phi \partial_{\widetilde{\vn}}v|_{K^2_F} - v\partial_{\widetilde{\vn}} \phi|_{K^1_F} + v\partial_{\widetilde{\vn}}\phi|_{K^2_F} \right)\dS
		+ \sum_{F\in\mathcal F,\; F\subset \partial Q}\int_{F} \left(\phi \partial_{\widetilde{\vn}}v - v \partial_{\widetilde{\vn}} \phi \right)\dS= 0.
	\end{equation}
	Since $\phi \in \mathcal D(Q)$, $\phi$ and $\partial_{\widetilde{\vn}} \phi$ vanish on $\partial Q$, and
	the integrals over $F\cap \partial Q$ vanish.
	Moreover, since $\partial_{\widetilde{\vn}}\phi$ and $v$ are continuous across any facet $F=\partial K^1_F\cap \partial K^2_F$,
	we have $v\partial_{\widetilde{\vn}} \phi|_{K^1_F} - v\partial_{\widetilde{\vn}}\phi|_{K^2_F}=0$.
Finally, \eqref{eq:femJumps} can be rewritten as
	\begin{equation*}
		\sum_{F\in\mathcal F,\; F\subset Q}\int_{F} \phi \big(\partial_{\widetilde{\vn}}v|_{K^1_F} - \partial_{\widetilde{\vn}}v|_{K^2_F} \big)\dS = 0
		 \qquad\forall\phi \in \mathcal D(Q),
	\end{equation*}
	which shows that $\partial_{\widetilde{\vn}} v$ is continuous over every $F\in\mathcal F$. To conclude, note that the tangential gradients $\Grad_F (v|_{K^1_F}) = \Grad_F (v|_{K^2_F})$ on $F$ because $v$ is continuous across elements.
\end{proof}

Lemma~\ref{lemma:C1conformity} imposes a regularity constraint on the Galerkin spaces that can be used with formulation~\eqref{eq:vpV}.
On the other hand, high-regularity discrete spaces, in the spirit of isogeometric analysis, have already proved particularly beneficial in the context of wave propagation, see for example \cite{HughesRealiSangalli2008, PavarinoZampieri2021, FraschiniLoliMoiolaSangalli2023}.

\begin{remark}[DOF count: space--time vs time-stepping]\label{rem:DOFcount}
The obvious drawback of space--time methods is that they require the solution of large linear systems.
We show that,
when approximating smooth solutions with uniform
maximal-regularity splines,
a space--time scheme and a single step of any time-stepping method achieve the same accuracy solving linear systems of comparable sizes, if, in the former, one takes a slightly higher polynomial degree and coarser elements.
Assume that $u$ is smooth,
	\begin{equation*}
		Q = \left(-\LI, \LI \right)^d\times(0,T), \qquad \LI\sim T\sim c\sim1,
	\end{equation*}
	\begin{equation*}
		N_x  = \frac{2\LI}{h_x}, \qquad N_t = \frac{T}{h_t}, \qquad V_N = \mathbb Q^p(\mathcal T_h)\cap C^k(Q), \qquad  1\le k\le p-1,
	\end{equation*}
	where $\mathcal T_h$ is a uniform tensor-product mesh on $Q$ with $N_x^d$ Cartesian-product elements in space and $N_t$ in time, and $\mathbb Q^p(\mathcal T_h)$ is the space of degree-$p$, tensor-product, piecewise polynomials on $\mathcal T_h$.
Recall that $\dim V_N\sim N_x^dN_t(p-k)^d$ and
	\begin{equation*}
		\|u - u_N\|_{H^1(Q)} \lesssim \max\{h_x,h_t\}^p \sim \max\{N_x^{-p}, N_t^{-p}\}.
	\end{equation*}
	Then, the number of degrees of freedom (DOFs) required to achieve a given accuracy $\epsilon>0$ is of order $\epsilon^{-\frac{d+1}{p}}(p-k)^d$.
	Using splines with maximal regularity (i.e.\ $k = p - 1$), to get $H^1$ accuracy proportional to $\epsilon$, a linear system with $\epsilon^{-\frac{d+1}{p}}$ DOFs needs to be solved.

We compare this with the cost of a single time-step of an implicit method, using the spline space
$\mathbb Q^{p_{TS}}(\mathcal T_{h_{TS}}^x)\cap C^{p_{TS}-1}(\Omega)$
with polynomial degree $p_{TS}$, space mesh size $h_{TS}$, and  $N_{TS}=2\LI/h_{TS}$ elements in each direction.
For simplicity, we assume that the time-step does not introduce any error other than the approximation in space.
To achieve the same error $\epsilon$, one needs $N_{TS}^d = \epsilon^{-\frac{d}{p_{TS}}}$ DOFs.
Choosing
$$
p\sim \frac{d+1}d p_{TS}\qquad\text{and}\qquad
N_x \sim N_t \sim N_{TS}^{\frac{p_{TS}}{p}}\sim N_{TS}^{\frac{d}{d+1}},
$$
then both the space--time method and a single time-step of the implicit method have accuracy proportional to $\epsilon$ and require the solution of a linear system of size $\sim \epsilon^{-\frac{d+1}p}\sim\epsilon^{-\frac{d}{p_{TS}}}$.

We note that for Galerkin approximation of the wave equation with high-regularity splines, explicit and implicit time-stepping schemes give linear systems with the same size and sparsity pattern, as pointed out in \cite[sect.~4]{PavarinoZampieri2021}.
\end{remark}

\begin{remark}[Condition number bound]
	\label{remark:condNumber}
Using inverse estimates for polynomial spaces, we provide an upper bound on the condition number of the Galerkin matrix of \eqref{eq:Galerkin}.
	Let $V_N \subset V$ be a space of piecewise polynomials of degree at most $p$ on a mesh that is tensor-product of quasi-uniform meshes in $\Omega$ and $(0,T)$.
	The elements of $V_N$ must be in $C^1(Q)$ by Lemma~\ref{lemma:C1conformity}.
	Let $\{\varphi_j\}_{j=1,\ldots,N}$ be a basis scaled such that $\|v_h\|_{Q}^2 \sim h_x^{d}h_t \|\vect v\|_2^2$ for all $v_h=\sum_{j=1}^N v_j\varphi_j \in V_N$, $\vect v=(v_1,\ldots,v_N)\in\mathbb R^N$, and let  $\mat B_{ij} = b(\phi_j, \phi_i)$. First, observe that by \eqref{eq:L2V}
	\begin{equation*}
		\|v_h\|_V^2 \geq (2T^2)^{-1} \|v_h\|_Q \sim (2T^2)^{-1} h_x^dh_t \|\vect v\|_2^2\qquad \forall v_h\in V_N.
	\end{equation*}
	Moreover, using inverse and trace inequalities \cite[Lemma~4.5.3, Thm.~1.6.6]{BrennerScottRidgway2008},
	the definition \eqref{eq:norm_def} of the norm $\|\cdot\|_V$ gives
	\begin{align*} \|v_h\|_V^2 \lesssim \big( &
			h_t^{-2} + c^2 h_x^{-2} + c^4 T^2 h_x^{-4} + T^2 h_t^{-4} \\ & + T h_t^{-3} + c^2 T h_x^{-2}h_t^{-1} +T^{-1}h_t^{-1}
			+ \LI h_t^{-2}h_x^{-1} + c^2 \LI h_x^{-3} \big) h_x^dh_t \|\vect v\|_2^2
			\qquad \forall v_h\in V_N.
	\end{align*}
	Then, arguing as in the proof of \cite[Prop.~5.4]{Moiola2014}, using $h_x\le \LI$, $h_t\le T$ and Young's inequality, the condition number $\kappa_2(\mat B) := \|\mat B\|_2 \|\mat B^{-1}\|_2$ is bounded by
	\begin{align*}
		\kappa_2(\mat B) & \lesssim \frac{C_b}{\alpha_b} {2T^2} \Big(
			h_t^{-2} + c^2 h_x^{-2} + c^4 T^2 h_x^{-4} + T^2 h_t^{-4} \\ &\qquad \qquad
			+ T h_t^{-3} + c^2 T h_x^{-2}h_t^{-1} +T^{-1}h_t^{-1}
			+ \LI h_t^{-2}h_x^{-1} + c^2 \LI h_x^{-3} \Big) \lesssim \frac{C_b}{\alpha_b}
\bigg(\frac T{h_t}+\frac{\LI+cT}{h_x}\bigg)^4,
	\end{align*}
where $\alpha_b$ ($C_b$ respectively) is the coercivity (continuity respectively) constant of $b$.
	This shows that the condition number grows at most with order 4, when the mesh is refined in either one of the directions.
	The fourth power of the mesh size appears because of the presence of the term $\|\mathcal W v\|_Q$ in the norm $\|v\|_V$, which is needed to ensure the continuity of $b(\cdot,\cdot)$.
	Analogous results can be derived for unstructured, non-tensor-product, quasi-uniform, space--time meshes.
\end{remark}

\begin{remark}[Energy considerations]
\label{remark:EnergyDecay}
	Define the energy of any $u\in W$ at time $t\in[0,T]$ as
	\begin{equation*}
	\mathcal E(t; u) := \frac12 \big(\|\TimeDer u\|_{\Omega_t}^2 + c^2\|\Grad u\|_{\Omega_t}^2\big).
	\label{eq:Energy}
	\end{equation*}
	If $u$ solves \eqref{eq:imp_varprob} with $f = 0$ and $\gI = 0$, integrating by parts and imposing the boundary conditions shows that energy does not increase: for $0\le t_1<t_2\le T$
	\begin{equation*}
	\begin{split}
		\mathcal E(t_2; u) - \mathcal E(t_1; u) =
		\int_{t_1}^{t_2}\TimeDer \mathcal E(t; u) \dt =
		\int_{t_1}^{t_2}\int_{\GD} c^2 \TimeDer u \NormDer u \dS \dt =
		- \frac{c}\theta \int_{t_1}^{t_2}\int_\GD (\TimeDer u)^2 \dS \le 0.
	\end{split}
	\end{equation*}
	For a general $u\in W$, bound \eqref{eq:est1} in the proof of Theorem~\ref{thm:Existence} is uniform in the parameter $N$ and thus holds with $u$ in place of $u^N$.
	This is an upper bound on $\mathcal E(t; u)$ which depends only on $\|\oW u\|_Q$, $\|\NormDer u+(\theta c)^{-1}\TimeDer u\|_\SI$, $\|\Grad u\|_\OZ$ and $\|u\|_\OZ$ and is uniform in $t$.
	Moreover, the right-hand side of such bound can be bounded using the $V$ norm of $u$:
	\begin{equation*}
\mathcal E(t; u)\le
		e^1\max\Big\{\frac1{2T},\;\frac{c}{\LI}\Big(\theta+\frac1\theta\Big)\Big\}\|u\|_V^2
		\qquad\forall t\in[0,T].
	\end{equation*}
This allows to control the energy $\mathcal E(t;u-u_N)$ of the Galerkin error using the approximation properties of $V_N$.
We can also control the error of the energy: for any $w\in W$ such that $\|u-w\|_V\le\|u\|_V$ (e.g.\ $w=u_N$ the Galerkin solution for a sufficiently accurate space $V_N$)
\begin{align*}
|\mathcal E(t;u) - \mathcal E(t; w)|
&=\frac12\int_{\Omega\times\{t\}}
\Big(\TimeDer(u+w)\TimeDer(u-w)+c^2\nabla(u+w)\cdot\nabla(u-w)\Big)\dx\\
&\le \sqrt{\mathcal E(t;u+w)\;\mathcal E(t;u-w)}\\
&\le e^1\max\Big\{\frac1{2T},\;\frac{c}{\LI}\Big(\theta+\frac1\theta\Big)\Big\}
\;\|u+w\|_V \;\|u-w\|_V\\
&\le 3\,e^1\max\Big\{\frac1{2T},\;\frac{c}{\LI}\Big(\theta+\frac1\theta\Big)\Big\}
\;\|u\|_V \;\|u-w\|_V.
\end{align*}
\end{remark}

\section{Scattering problem}\label{sec:Scattering}
We now consider problem \eqref{eq:IBVP} with non-empty Dirichlet boundary~$\SD$.
We follow the approach of \S\ref{sec:interior_problem} modifying $b$ and $F$ to take into account the Dirichlet boundary condition on $\SD$, which will be imposed in a weak sense. A least-squares term involving time derivatives will appear in $\ell$ and $F_\ell$.
To ensure coercivity, the sign of the product $\vx\cdot\vn$ on $\SD$ has to be opposite to that on $\SI$, see Assumption~\ref{ass:StarSD} and Figure~\ref{fig:StarShape}.

We assume that the initial and boundary data $u_0\in H^1(\OZ)$ and $\gD\in H^1(\SD)$ are compatible, in the sense that their traces on $\partial\OZ$ coincide.
If this were not the case, we could not expect the IBVP solution $u$ to belong even to $H^1(Q)$: if e.g.\ $u_0=1$ and $\gI=0$, the trace of $u$ on the Lipschitz boundary $\partial Q$ would not belong to $H^{1/2}(\partial Q)$, as it can be deduced from \cite[sect.~5.2]{Triebel2002}, so $u$ would not be in $H^1(Q)$ by the classical trace theorem \cite[Thm.~3.37]{McLean00}.

\subsection{Bilinear and linear form definition}
To define the variational problem we start again from \eqref{eq:integrated_identity}.
Recalling that $\Sigma = \overline{\SI \cup \SD}$, we separate the integral term on $\Sigma$ in a term on $\SI$, treated as in \S\ref{sec:interior_problem}, and a new one on $\SD$.
To split $m$ as in \eqref{eq:assum:boundary_identity} and define the penalty terms as in \eqref{eq:assum:ell=L}, we observe that,
whenever the Dirichlet boundary condition $u = \gD$ holds on $\SD$, then all the terms in the integral over $\SD$ are known, except for $\NormDer u$. Recalling the definition \eqref{eq:multiplier_def} of $\operator M$, this suggests to split the integral over $\Sigma$ in \eqref{eq:integrated_identity} in the following way:
\begin{align*}
&\int_\Sigma \big[c^2 \operator{M} u\NormDer{v} + c^2\NormDer{u} \operator{M} v
				+ \xi \vect{x} \cdot \vect{n} (- \TimeDer{u} \TimeDer{v}
				+ c^2 \,\Grad u \cdot \Grad v) \big] \dS\dt
= I^r_{\SI} + I^r_{\SD} + I^s_{\SI} + I^s_{\SD},
\\
I^r_{\SI} & := \int_{\SI} \big[c^2 \oM u\NormDer{v} - c\theta^{-1} \TimeDer u \oM v - \xi \vect{x} \cdot \vect{n} (\TimeDer{u} \TimeDer{v})+ c^2 \xi \vect{x} \cdot \vect{n} (\Grad u \cdot \Grad v)  \big]\dS\dt,\\
I^r_{\SD} & := \int_{\SD} c^2  \NormDer u \oM v \dS \dt,\\
I^s_{\SI} & := \int_{\SI} \big[c\theta^{-1}  \TimeDer u\oM v + c^2  \NormDer u\oM v \big]\dS \dt, \\
I^s_{\SD} & := \int_{\SD} \big[- c^2 \xi(\vx \cdot \GradGD u)\NormDer v + c^2 \beta (t-T^*) \TimeDer u \NormDer v
+ \xi \vect{x} \cdot \vect{n}(- \TimeDer u \TimeDer v+c^2 \GradGD u \cdot \GradGD v )\big] \dS \dt,
\end{align*}
where $I^r_{\bullet}$ and $I^s_{\bullet}$ are the terms to be included in the bilinear form $r$ and $s$, respectively.

To define the bilinear form for the scattering problem, we modify $r$, $s$ and $\ell$ in \eqref{eq:r_def}, \eqref{eq:s_def} and \eqref{eq:ell_def} to include
the terms on $\SD$.
We append the symbol ``$_\star$'' to all the bilinear forms and functionals of the scattering problem.
We set
\begin{align}
\nonumber r_{\star}(u, v) :=& \; r(u, v) - I_\SD^r, \\
\label{eq:rslScatter} s_{\star}(u, v) :=& \; s(u, v) - I_\SD^s, \\
\nonumber \ell_\star (u, v) :=&\; \ell(u, v) + A_\SD \LD \int_\SD \TimeDer u \TimeDer v \dS \dt,
\end{align}
for a parameter $A_\SD>0$, and where $\LD=\max\{|\vx|:\vx\in \GD\}$.
Recall that the integral over $\Sigma$ in \eqref{eq:integrated_identity} has a negative sign in front, and therefore the term $I_\SD^r$ ($I_\SD^s$, respectively) has to be included in $r_\star$ ($s_\star$, respectively) with a negative sign as well.

Recalling the definitions of $b$ and $F$ in \eqref{eq:def:bilin} and \eqref{eq:def:lin}, $b_\star$ and $F_\star$ are defined as
\begin{align}
b_\star(u,v)  := &\;\int_Q \oM u \oW v \dx \dt - r_\star(u,v) + \ell_\star(u, v) \nonumber\\
 = &\;\int_Q \oM u \oW v \dx\dt - r(u,v) + I_\SD^r + \ell (u, v) + A_\SD \LD \int_\SD \TimeDer u \TimeDer v \dS \dt
\label{eq:bScatter}\\
= &\; b(u, v) + \int_{\SD} \big(c^2 \NormDer u \oM v + A_\SD \LD \TimeDer u \TimeDer v \big)\dS \dt,\nonumber
\\
F_\star(v) := &\;
F(v) - I_\SD^s(v) + A_\SD \LD \int_\SD \TimeDer u \TimeDer v \dS \dt   \nonumber\\
\label{eq:FScatter}
= &\; F(v) + \int_{\SD} \big[c^2 \xi(\vx \cdot \GradGD \gD)\NormDer v - c^2 \beta (t-T^*) \TimeDer \gD \NormDer v \\
& \hspace{20mm}+ \xi \vect{x} \cdot \vect{n}( \TimeDer \gD \TimeDer v-c^2 \GradGD \gD \cdot \GradGD v )
+ A_\SD \LD \TimeDer \gD \TimeDer v \big]\dS \dt ,\nonumber
\end{align}
where $\operator{M}$ is defined in \eqref{eq:multiplier_def}
and $I_\SD^s(v)$ is the same as $I_\SD^s$ defined above except for $u$ is replaced by $\gD$.
The forms $b_\star$ and $F_\star$ are determined by the positive parameters
$A_Q$, $A_\OZ$, $A_\SD$, $\beta$, $\xi$, $T^*$.

\subsection{Norm, function spaces and variational problems}
Recalling the definition \eqref{eq:norm_def} of $\|\cdot\|_V$, we define a norm that controls all terms in $b_\star$ and $F_\star$: for $v \in C^\infty(\overline{Q})$,
\begin{equation}
\label{eq:norm_scattering}
\|v\|_\Vs^2 :=  \|v\|_V^2 + \LD \|\TimeDer{v}\|_{\SD}^2 + c^2 \LD \|\Grad v\|_{\SD}^2 .
\end{equation}
As in \eqref{eq:L2V}, $\|v\|_Q^2\le 2T^2\|v\|_\Vs^2$, so $\|\cdot\|_\Vs$ is indeed a norm.
As in \S\ref{sec:spaces}, we define two Hilbert spaces:
\begin{equation*}
\Vs := \overline{C^\infty(\overline{Q})}^{\|\cdot\|_\Vs}, \qquad
\Ws := \{v \in H^1(Q): \|v\|_\Vs < \infty\}.
\end{equation*}
The same considerations of Remark~\ref{remark:V=W?} apply in this case as well, in particular $H^2(Q) \subset \Vs \subset \Ws$, but it is not clear whether $\Vs = \Ws$.
We can thus write two, possibly equivalent, variational formulations of the scattering problem~\eqref{eq:IBVP}:

\begin{equation}
	\label{eq:vpVScatter}
	\text{Find}\; u \in \Vs\;\text{such that}\qquad b_\star(u,v) = F_\star(v) \qquad \forall v \in \Vs,
\end{equation}
and
\begin{equation}
	\label{eq:vpWScatter}
	\text{Find}\;u\in\Ws\;\text{such that}\qquad b_\star(u,v) = F_\star(v)\qquad \forall v\in\Ws,
\end{equation}
where $b_\star$ and $F_\star$ are defined in \eqref{eq:bScatter} and \eqref{eq:FScatter}.
Clearly, when $\GD=\emptyset$, problems \eqref{eq:vpVScatter} and \eqref{eq:vpWScatter} coincide with  \eqref{eq:vpV} and \eqref{eq:vpW}, respectively.

\subsection{Coercivity and continuity}
We prove the coercivity and the continuity of $b_\star$ with respect to the norm $\|\cdot\|_\Vs$, slightly adapting the proofs of \S\ref{sec:CC}.
The well-posedness of problems \eqref{eq:vpVScatter} and \eqref{eq:vpWScatter} follows.

To ensure the coercivity of $b_\star$, both boundaries $\GI$ and $\GD$ have to satisfy (opposite) conditions on the sign of $\vx\cdot\vn$ (recall that the unit normal $\vn$ on $\partial\Omega=\overline{\GI\cup\GD}$ points outwards of $\Omega$).
\begin{assumption}\label{ass:StarSD}
There exists $\deltaD > 0$, such that $-\vx\cdot\vn\geq\deltaD\LD$ for almost every $\vx\in\GD$.
\end{assumption}
If $\GI$ and $\GD$ are disjoint, Assumptions~\ref{ass:StarS} and \ref{ass:StarSD} coincide with the requirement that both $\GI$ and $\GD$ are boundaries of domains that are star-shaped with respect to concentric balls.
However, the case $\overline\GI\cap\overline\GD\ne\emptyset$ is allowed.
Two examples are shown in Figure~\ref{fig:StarShape}.

\begin{figure}[htb]
\begin{center}
\begin{tikzpicture}[scale=1.5]
\fill[lightgray](-.5,0)--(-.5,1.3)--(.5,1.3)--(.5,0);
\fill[lightgray] (0,0) ellipse (2 and .7);
\draw[dashed,thick,fill=white](.2,-.4)--(.3,.5)--(-.7,.1)--(.2,-.4);
\draw(2,.3)node{$\Gamma_I$};
\draw(-.4,-.2)node{$\Gamma_D$};
\draw(1,0)node{$\Omega$};
\fill(0,0)circle(.04)node[anchor=west]{$\vect0$};
\draw[->,thick](-.2,.3)--(-.2+0.08,.3-0.2)node[anchor=east]{$\vn$};
\draw[->,thick](.5,1.1)--(.75,1.1)node[anchor=north]{$\vn$};
\end{tikzpicture}
\hspace{20mm}
\begin{tikzpicture}[scale=1.5]
\fill[lightgray] (0,0) ellipse (1.5 and 1);
\fill[white] (-1.5,-.5)--(-.4,-.5)--(-.4,0)--(-.2,0)arc (180:0:.2)--(.4,0)--(.4,-.5)--(1.5,-.5)--(1.5,-1)--(-1.5,-1);
\fill(0,-.9)circle(.04)node[anchor=west]{$\vect0$};
\draw[->,thick](1,-.5)--(1,-.75)node[anchor=west]{$\vn$};
\draw[->,thick](1.5,0)--(1.75,0)node[anchor=north]{$\vn$};
\draw(-1.2,.8)node{$\Gamma_I$};
\draw(-.8,-.7)node{$\Gamma_D$};
\draw(1,0)node{$\Omega$};
\draw[dashed,thick](-1.3,-.5)--(-.4,-.5)--(-.4,0)--(-.2,0)arc (180:0:.2)--(.4,0)--(.4,-.5)--(1.3,-.5);
\end{tikzpicture}
\end{center}
\caption{Two space domains $\Omega\subset\mathbb R^2$ (in grey) with $\GD\ne\emptyset$ and that satisfy Assumptions~\ref{ass:StarS} and \ref{ass:StarSD} ($\vx\cdot\vn$ uniformly positive on $\GI$ and uniformly negative on $\GD$).
The dashed line represents the Dirichlet boundary $\GD$.
Left domain: the impedance and the Dirichlet boundaries are disjoint and they bound domains that are star-shaped with respect to the origin $\mathbf 0$.
Right domain: the impedance and the Dirichlet boundaries are not disjoint.
}
\label{fig:StarShape}
\end{figure}
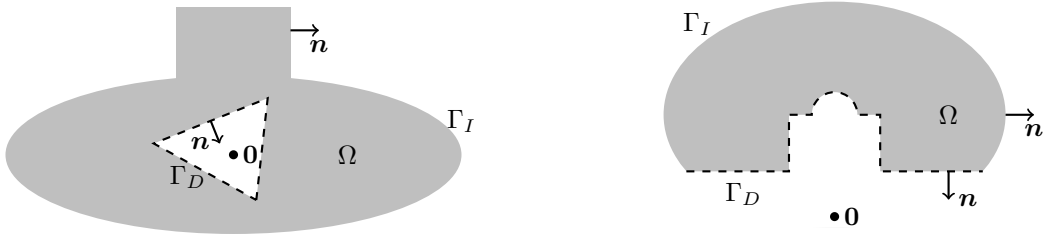

\begin{theorem}[Coercivity of $b_\star$]	\label{thm:coercBScat}
Under Assumptions~\ref{ass:StarS} and \ref{ass:StarSD} on $\Omega$, let $A_Q > 0$, $A_\OZ > 0$, $A_\SD \geq \xi > 0$, and let $\beta$ satisfy the lower bound \eqref{eq:coefficient_coercivity_cond}.
Then for all $u \in \Ws$
\begin{equation}\label{eq:CoercivityScat}
b_\star(u,u) \geq \alpha_{b_\star}\|u\|_\Vs^2 \qquad \text{with}\qquad
\alpha_{b_\star}:= \min\left\{\frac{\xi\deltaI}{4},\;\frac{\xi\deltaD}{2},\; A_Q,\; A_\OZ \right\}.
\end{equation}
\end{theorem}
\begin{proof}
For all $u \in \Ws$, we need to provide a lower bound for
\begin{equation*}
2 b_\star(u,u) = 2\ell_\star(u, u) - r_\star(u,u) + s_\star(u,u).
\end{equation*}
Using the definitions \eqref{eq:rslScatter}, the relation \eqref{eq:abstract_coercivity_condition}, and that $|\nabla u|^2=|\GradGD u|^2+|\NormDer u|^2$ on $\SD$,
we obtain
\begin{align*}
2 b_\star(u,u) & =
2\ell(u,u)- r(u, u) + s(u, u) +2A_\SD\LD\int_\SD |\TimeDer u|^2\dS\dt + I_\SD^r(u,u)- I_\SD^s(u,u)\\
&=2b(u,u) +  \int_\SD \big[2A_\SD \LD|\TimeDer u|^2 + c^2 \NormDer u \oM u \big]\dS \dt\\
&\quad + \int_{\SD} \big[c^2 \xi(\vx \cdot \GradGD u)\NormDer u - c^2 \beta (t-T^*) \TimeDer u \NormDer u
+ \xi \vect{x} \cdot \vect{n}(|\TimeDer u|^2 - c^2 |\GradGD u|^2)\big] \dS \dt \\
& = 2b(u,u) + 2A_\SD \LD \|\TimeDer u\|^2_\SD
+ \int_\SD \xi \vx\cdot\vn \big(|\TimeDer u|^2-c^2|\Grad u|^2 \big) \dS \dt,
\end{align*}
where $I_\SD^r(u, u)$ and $I_\SD^s(u,u)$ are defined as $I_\SD^r$ and $I_\SD^s$ respectively, except for $v=u$.
Using Assumption~\ref{ass:StarSD}, the fact that $A_\SD\geq\xi$ and that $\deltaD\le1$ observe that
\begin{align*}
2A_\SD \LD \|\TimeDer u\|^2_\SD + \int_\SD \xi \vx\cdot\vn \big(|\TimeDer u|^2-c^2|\Grad u|^2 \big) \dS \dt & \geq (2A_\SD - \xi)\LD \|\TimeDer u \|^2_\SD + c^2 \xi\deltaD\LD\|\Grad u\|^2_\SD \\
& \geq \xi \deltaD\big(\LD \|\TimeDer u \|^2_\SD + c^2\LD\|\Grad u\|^2_\SD\big).
\end{align*}
The term $2b(u,u)$ is bounded below by $2\alpha_b\|u\|_V^2$ by \eqref{eq:coercivity_condition}, since the proof of Theorem~\ref{thm:coerc_b} does not require integrations by parts and so it is applicable also in the presence of the Dirichlet boundary $\SD$.
Then, recalling the relation \eqref{eq:norm_scattering} between the norms $\|\cdot\|_V$ and $\|\cdot\|_\Vs$, assertion \eqref{eq:CoercivityScat} follows at once.
\end{proof}

\begin{lemma}[Continuity of $b_\star$ and $F_\star$]
The bilinear form $b_\star$ and the linear functional $F_\star$ are continuous in $\Ws$ with respect to the norm $\|\cdot\|_\Vs$ in \eqref{eq:norm_scattering}:
$$
|b_\star(u,v)|\le C_{b_\star} \|u\|_\Vs\|v\|_\Vs,
\qquad
|F_\star (v)|\leq C_{F_\star} \|(f, \gI, \gD, u_0, u_1)\|_{\mathtt{d}_\star} \|v\|_\Vs
\qquad
\forall u,v\in \Ws,
$$
where
\begin{equation*}C_{b_\star} := \max
\begin{cases}
C_b,\\
	\sqrt3\big(\beta  (\nu-1) \frac{c T}{\LD} + \xi\big),  \\
	\sqrt3 A_\SD,
\end{cases}
\qquad
C_{F_\star}:=\max \begin{cases}C_F,\\
	\xi+A_\SD,\\
	2\xi,\\
	\beta (\nu-1) \frac{cT}{\LD},
\end{cases}
\end{equation*}
and, recalling the data norm $\|\cdot\|_\mathtt{d}$ in \eqref{eq:dnorm},
\begin{align*}
\|(f, \gI, \gD, u_0, u_1)\|^2_{\mathtt d_\star} := \|(f, \gI, u_0, u_1)\|^2_\mathtt{d} + \LD \|\TimeDer \gD\|^2_\SD + c^2\LD \|\GradGD \gD\|^2_\SD.
\end{align*}
\end{lemma}
\begin{proof}
The proof follows exactly as that of Proposition~\ref{prop:continuity_b}, after extending the vectors and the block-diagonal matrix as
\begin{align*}
\vect{m}_\star(v)&:=\Big(\vect{m}(v),\quad L_D^{\frac12}\|\TimeDer v\|_\SD,\quad c L_D^{\frac12}\|\nabla v\|_\SD\Big)\in \mathbb{R}^{12},\\
\mathbf F_{F_\star}&:=\bigg(\mathbf F_F,\;
(\xi +A_\SD)L_D^{\frac12}\|\TimeDer \gD\|_\SD,\;
\Big(\frac{c^2T^2}{L_D}\beta^2(\nu-1)^2\|\TimeDer\gD\|_\SD^2 + 4\xi^2c^2L_D\|\GradGD\gD\|_\SD^2\Big)^{\frac12}\bigg)\in \mathbb{R}^{12},\\
\mat{M}_\star&:=
\left(\begin{matrix}\mat{M} &0 \\ 0 & \mat{M}_\SD\end{matrix}\right),\qquad
\mat{M}_\SD :=
\left(\begin{matrix}A_\SD & \beta (\nu-1) \frac{cT}{\LD}\\0 & \xi\end{matrix}\right),
\end{align*}
so that
$$
\|\vect{m}_\star(v)\|_2=\|v\|_\Vs, \qquad
|b_\star(u,v)|\le |\vect{m}_\star(v)^\top \mat{M}_\star\vect{m}_\star(u)|,\quad \text{and} \quad
|F_\star(v)|\le \mathbf F_{F_\star}\cdot\vect{m}_\star(v)
\quad \forall u,v\in \Ws.
$$
\end{proof}

\begin{corollary}
Assume that the hypotheses of Theorem~\ref{thm:coercBScat} hold. Then, from Lax--Milgram theorem, it follows that the solution $u$ to problem \eqref{eq:vpWScatter} satisfies the stability bound
\begin{align*}
		\|u\|_\Vs \leq C_{F_\star}
		 \max\left\{\frac1{A_Q},\;\frac1{A_\OZ},\;\frac2{\xi\deltaD},\;\frac4{\xi\deltaI}\right\}\|(f, \gI,\gD,u_0,u_1)\|_{\mathtt{d}_\star}.
	\end{align*}
	The bounding constant scales with the same powers of the parameters as that in Theorem~\ref{thm:Existence} and deteriorates for either $\deltaI\to0$ or $\deltaD\to0$.
\end{corollary}

The considerations made in \S\ref{sec:NumImplications} on the numerical aspects of the proposed formulation, such as quasi-optimality, conforming discretisations, conditioning and number of DOFs, immediately extend to the scattering problems \eqref{eq:vpVScatter} and \eqref{eq:vpWScatter}.

\section{Numerical experiments}
\label{sec:numExperiments}

We report some numerical results obtained from a simple spline discretisation of formulation \eqref{eq:vpV}.
We focus in particular on the choice of the parameters in the variational formulation, their robustness, the optimality of the convergence rates, the approximation of non-smooth solutions, the sharpness of the theoretical quasi-optimality bounds, the conservation of energy.
We consider the following simplified setting:
\begin{itemize}
\item $\GD=\emptyset$, i.e., the impedance cavity problem;
\item the space dimension is $d=1$, the space domain is $\Omega=(-1,1)$, and the final time $T=1$;
\item we use $N_x + 1$ equispaced nodes $x_j = -1 + j h_x$ in space, $j = 0, \dots, N_x$, and $N_t + 1$ equispaced nodes $t_i = i h_t$ in time, $i = 0, \dots, N_t$, with $h_x=2/N_x$ and $h_t=1/N_t$;
\item we let the discrete space $V_h$ be the tensor-product Hermite element space in $Q=(-1,1)\times(0,1)$ (i.e., the Bogner--Fox--Schmit element space in space--time).
More explicitly, the discrete space is
\begin{equation*}
V_h = \big\{v \in C^1(Q),\quad
	v|_{(x_j, x_{j+1})\times(t_i, t_{i+1})} \in \mathbb{Q}^3,\quad j = 0, \dots, N_x-1, \quad i = 0, \dots, N_t-1 \big\},
\end{equation*}
where $\mathbb Q^3$ is the space of polynomials of degree at most 3 in each variable;

\item we choose as basis functions the products $\varphi_{j,i}(x,t)=v_j(x)v_i(t)$,
for $j=0, \dots, 2N_x + 2$ and $ i=0, \dots, 2N_t + 2$, where $v_j$ and $v_i$ are the cubic Hermite basis functions of \cite[Ch.~IV, p.~48]{deBoor2001}.
Each basis function is supported in (at most) four elements of the space--time mesh, and is normalised so that one of $\varphi_{j,i},\partial_x\varphi_{j,i},\partial_t\varphi_{j,i},\partial_x\partial_t\varphi_{j,i}$ has value 1 at the mesh node at the centre of its support and the other three vanish at the same point.
For simplicity, let $N:= \dim(V_h) = (2N_t + 2)(2N_x + 2)$.
\end{itemize}
The method was implemented in Matlab R2023b, the linear systems were solved with the backslash direct solver, and all experiments were run on a laptop.\footnote{The code developed is available on \url{https://github.com/pbignardi/CoerciveWaveTests}.}

More extensive numerical experiments in higher space dimensions, with more general spline spaces, more exhaustive parameter sensitivity analysis, matrix compression, and comparisons with other methods have been implemented and will be subject of a separate report.

The Galerkin discretisation of formulation \eqref{eq:vpV} is
\begin{equation}
	\text{Find}\;u_h\in V_h \quad \text{such that}\quad b(u_h, v_h) = F(v_h) \qquad \forall v_h \in V_h,
	\label{eq:hGalerkin}
\end{equation}
which in matrix form reads
\begin{equation}
	\text{Find}\;\vect U\in\R^N \quad \text{such that}\quad \mat B \vect U = \vect{F}.
	\label{eq:matGalerkin}
\end{equation}

\let\prestretch\arraystretch
\renewcommand{\arraystretch}{1.30}
\begin{table}[htb]
	\begin{subtable}[t]{.5\linewidth}
	\begin{tabular}{|c|p{.85\linewidth}|}
		\hline
		\multirow{7}{*}{\rotatebox{90}{\textsc{Problem 1}}} & $c=1$ \\
		& $\theta=1$\\\cline{2-2}
& $f(x, t) =  (2\cos2t + \pi^2\sin^2 t)\cos\pi x+2\cos2t$ \\
		& $\gI(\pm1, t) = 0$\\
		& $u_0(x) = 0$ \\
		& $u_1(x) = 0$ \\\cline{2-2}
		& $u(x, t) = \sin^2t\;(\cos\pi x + 1)$ \\\hline
	\end{tabular}
	\caption{IBVP with homogeneous boundary and initial conditions.}
	\label{tab:P1}
	\end{subtable}\hfill
	\begin{subtable}[t]{.5\linewidth}
	\begin{tabular}{|c|p{.85\linewidth}|}
		\hline
		\multirow{7}{*}{\rotatebox{90}{\textsc{Problem 2}}} & $c=2$ \\
		& $\theta=10$\\\cline{2-2}
		& $f(x, t) = 0$ \\
& $\gI(-1, t) =
		\frac{-(\theta+1)^2 w'(-1-ct)+(\theta-1)^2w'(3-ct)}{\theta(\theta+1)}$\\
		& $\gI(1, t) = 0$\\
		& $u_0(x) = w(x) + \frac{\theta-1}{\theta+1} w(2 - x)$ \\
		& $u_1(x) = - c w'(x) - c \frac{\theta-1}{\theta+1} w'(2-x)$ \\\cline{2-2}
		& $u(x, t) =w(x - ct) + \frac{\theta-1}{\theta+1} w(2 - x - ct)$ \\\hline
	\end{tabular}
	\caption{IBVP for a smooth travelling wave with homogeneous source term.}
	\label{tab:P2}
	\end{subtable}
	\centering
	\begin{subtable}[t]{.5\linewidth}
	\begin{tabular}{|c|p{.85\linewidth}|}
		\hline
		\multirow{7}{*}{\rotatebox{90}{\textsc{Problem 3}}} & $c = 1$ \\
		    & $\theta = 1$ \\\cline{2-2}
		    & $f(x, t) = 0$ \\
		    & $\gI(\pm1, t) = 0$\\
		    & $u_0(x) = w(x + 1)$\\
		    & $u_1(x) = - c w'(x + 1)$\\\cline{2-2}
		    & $u(x, t) = w(x - ct + 1) \mathbf{1}_{\{x - ct + 1 > 0\}}$\\\hline
	\end{tabular}
	\caption{IBVP for a non-smooth travelling wave.}
	\label{tab:P3}
	\end{subtable}
	\caption{Data and solutions of the test impedance initial--boundary value problems.
	The solutions $u$ are plotted in Figure~\ref{fig:solPlots}.
	Here $w$ is the symmetric double-Gaussian wave profile $w(x):=e^{-20(x-0.1)^2}-e^{-20(x+0.1)^2}$.
	In Problem~2, $|g_I|\le 4\cdot10^{-6}$ on $\SI$ because of the Gaussian decay, so in practice we approximate it with 0.}
	\label{tab:Problems}
\end{table}
\renewcommand{\arraystretch}{\prestretch}

In the following we consider three IBVPs:
Table~\ref{tab:Problems} reports their data and exact solutions, and Figure~\ref{fig:solPlots} shows the solutions in $Q$.
The first two problems have smooth data and smooth solutions $u\in C^\infty(\overline Q)$.
Instead, Problem~3 has smooth data, but the solution is not smooth, in fact $u$ is not even in $H^2(Q)$ (more precisely, $u\in H^{3/2-\epsilon}(Q)\setminus H^{3/2}(Q)$ for all $\epsilon>0$), as the initial and boundary data $u_0,u_1,\gI$ (see Table~\ref{tab:P3}) fail to verify the compatibility condition \eqref{eq:compatibility} at the point $(-1,0)\in \partial\OZ$, which is required for the regularity Theorem~\ref{thm:t_regularity} to hold.
Indeed, in this case
\begin{align*}
	\NormDer u_0(-1)+(c\theta)^{-1}u_1(-1)&\neq\gI(-1).
\end{align*}
Such incompatibility generates a jump in the time and space derivatives of the solution along the line $\{x-ct+1=0\}$, and the individual second-order derivatives are not in $L^2(Q)$.
However, the distributional wave operator $\oW u$ vanishes, thus $u\in W$.

\begin{remark}[$u\in V$ for Problem 3]
The solution $u$ to the problem in Table~\ref{tab:P3} (shown in Figure~\ref{fig:Sol3Plot}) belongs not only to $W$, but also to $V$.
To show this, we construct a sequence of functions in $C^\infty(\overline{Q})$, that converges to $u$ in the $\|\cdot\|_V$ norm.
Consider the function $\tilde{u}_0\in C^0(\R)$ defined as
$\tilde u_0(x) = w(x+1) \mathbbm{1} _{[-1, +\infty)}(x)$.
For any $\epsilon>0$, let $\tilde u_0^\epsilon$ be the function $\tilde u_0$ shifted to the right by $\epsilon$, i.e.\
$\tilde u_0^\epsilon(x) = w (x+1 - \epsilon) \mathbbmss{1}_{[-1+\epsilon, +\infty)}(x)$.
Let $\varphi_\epsilon := \eta_{\epsilon/2}*\tilde u_0^\epsilon$, where $\eta_{\epsilon/2}\in C^\infty(\R)$ is the mollifier with support $[-\epsilon/2,\epsilon/2]$ defined in \cite[sect.\ C.5]{EvansPDE}.
Then $\varphi_\epsilon\in C^\infty(\R)$, is supported in $(-1+\epsilon/2, +\infty)$ and, for $\epsilon\rightarrow0$, $\varphi_\epsilon\rightarrow\tilde u_0$ in $H^1(\R)$.
Consider now $\tilde u_\epsilon(x, t) = \varphi_\epsilon(x-ct)$: clearly $\tilde u_\epsilon \in C^\infty(\R^2)$, and its support lies in $\{(x, t): x-ct+1\ge\epsilon/2\}$.
Moreover, $u_\epsilon:=\tilde u_\epsilon|_Q\rightarrow u$ in $H^1(Q)$
since $u(x, t) = \tilde u_0(x - ct)$: this is
the sequence of functions $u_\epsilon$ we were looking for.
Indeed, convergence in $V$ of $u_\epsilon$ follows noting that $\oW u_\epsilon =\oW u= 0$, and that, since $\varphi_\epsilon$, $\varphi'_\epsilon$ converge almost everywhere in $\R$ and using Lebesgue convergence theorem, also the traces of $\TimeDer u_\epsilon$ and $\Grad u_\epsilon$ on the boundary parts of $Q$ converge.
\end{remark}

\begin{figure}[htb]
\centering
\begin{subfigure}{.33\linewidth}
	\centering
	\includegraphics[width=\linewidth]{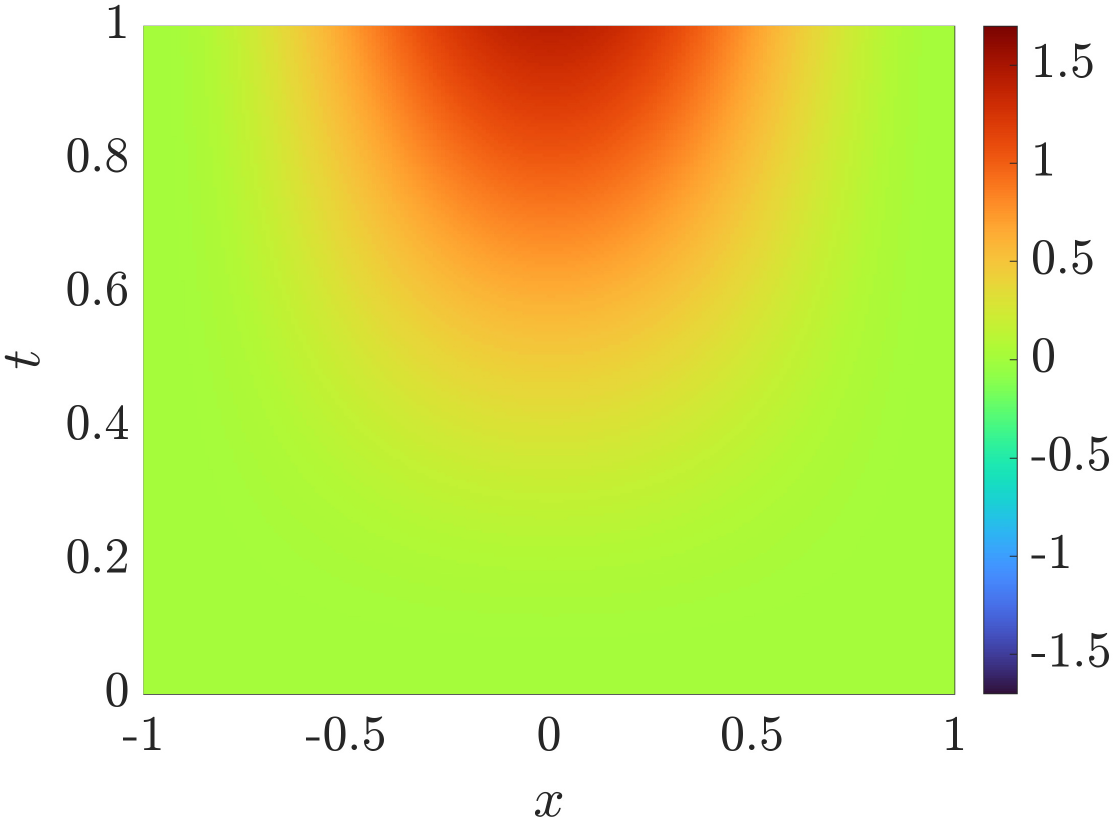}
	\caption{Solution of Problem~1.}
\end{subfigure}\begin{subfigure}{.33\linewidth}
	\centering
	\includegraphics[width=\linewidth]{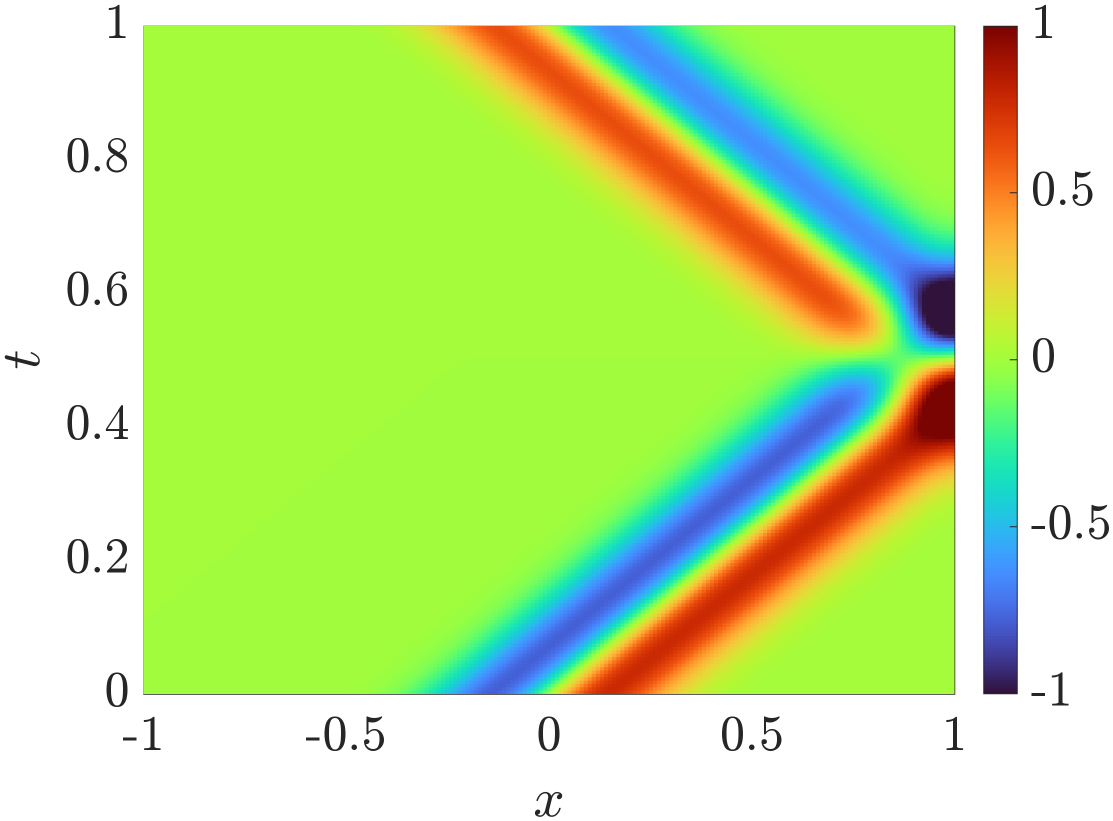}
	\caption{Solution of Problem~2.}
\end{subfigure}\begin{subfigure}{.33\linewidth}
	\centering
	\includegraphics[width=\linewidth]{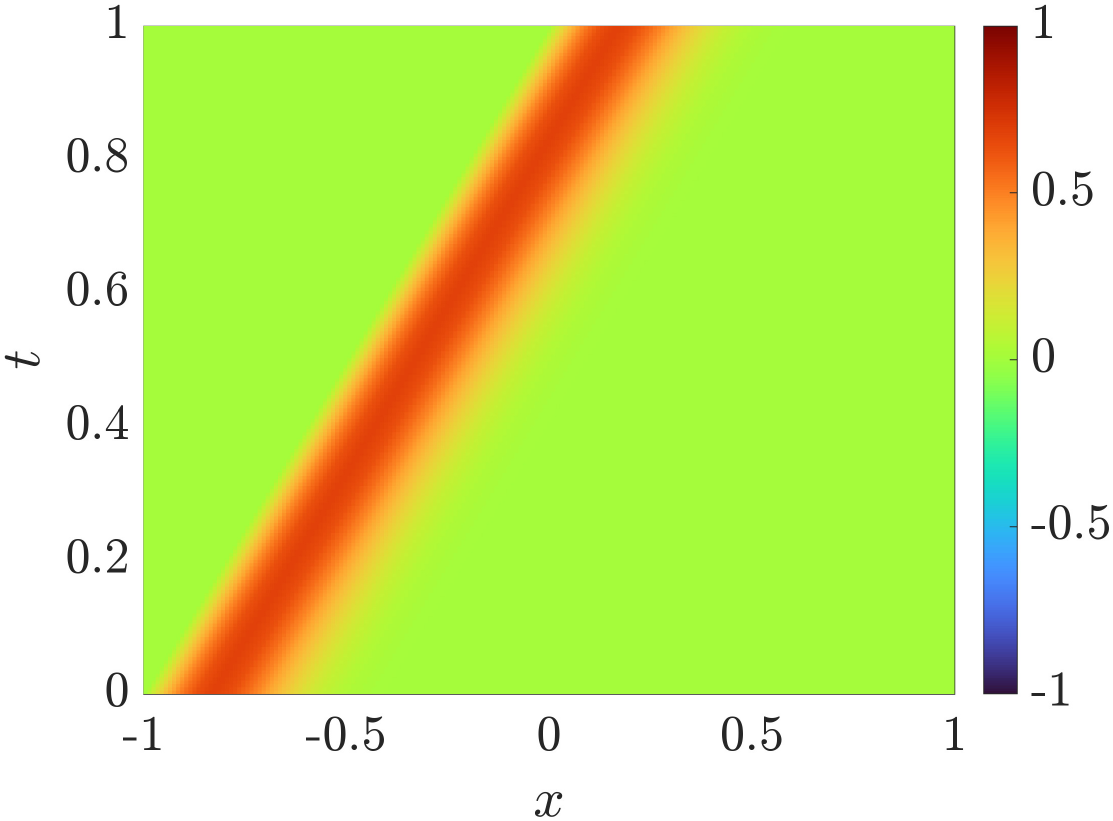}
	\caption{Solution of Problem~3.}
	\label{fig:Sol3Plot}
\end{subfigure}
\caption{The solutions to the problems in Table~\ref{tab:Problems}.
Left: solution to Problem~1 with a volume source term.
Centre: the solution of Problem~2 is a wave packet travelling through the domain and bouncing off the impedance boundary (since $\theta\ne1$).
Right: the solution $u\in V\setminus H^2(Q)$ of Problem~3 is non-smooth on the line $x-ct+1=0$.}
\label{fig:solPlots}
\end{figure}

\subsection{Testing formulation parameters}
\label{subsec:testA}
The bilinear form $b(\cdot,\cdot)$ and the linear functional $F(\cdot)$ in \eqref{eq:def:bilin}--\eqref{eq:def:lin} depend on the parameters $A_Q,A_\OZ,\beta,\xi$ and $\nu$ (recall that $T^*=\nu T$).
The main objective of this section is to study the sensitivity of the Galerkin solution with respect to their choice.
We show that a good choice of the parameters can improve accuracy, but no fine tuning is necessary for the stability and convergence of the method.
We consider Problem~1 as described in Table~\ref{tab:P1}; however, we observed similar results across the IBVPs of Table~\ref{tab:Problems}.

\subsubsection{Parameters \texorpdfstring{$A_Q$}{AQ} and \texorpdfstring{$A_\OZ$}{AOmegaZero}.}
We first assess the sensitivity with respect to the two least-squares parameters $A_Q$ and $A_\OZ$.
Let $N_x = N_t = 32$.
We approximate the solution to Problem~1 using \eqref{eq:hGalerkin}, with  $\xi=1$, $\beta=\beta^\#$, and $\nu=2$ as in \eqref{eq:ParamChoice}, and where $A_Q$ and $A_\OZ$ are picked from the sets
\begin{equation*}
\left\{10^{s_k} : s_k \;\text{are $75$ equispaced nodes in}\;[-15, 2]\right\}
\quad \text{and}\quad
\left\{10^{s_k} : s_k \;\text{are $75$ equispaced nodes in}\;[-7, 3]\right\},
\end{equation*}
respectively.
For each choice of the parameters $A_Q$ and $A_\OZ$ we compute the relative $L^2(Q)$ error and the condition number of the corresponding Galerkin matrix $\mat B$ in \eqref{eq:hGalerkin}.
Figure~\ref{fig:testA} shows the results of such computations, and suggests that $A_Q \sim 10^{-2}$ and $A_\OZ\sim1$ lead to the best accuracy.
Moreover, this parameter choice also leads to a reduction in the condition number.
In the following sections we compare the cases $A_Q=10^{-2}$ and $A_Q=1$, both with $A_\OZ=1$.

\begin{figure}[ht!]
	\begin{subfigure}{.5\linewidth}
		\centering
		\includegraphics[width=\linewidth]{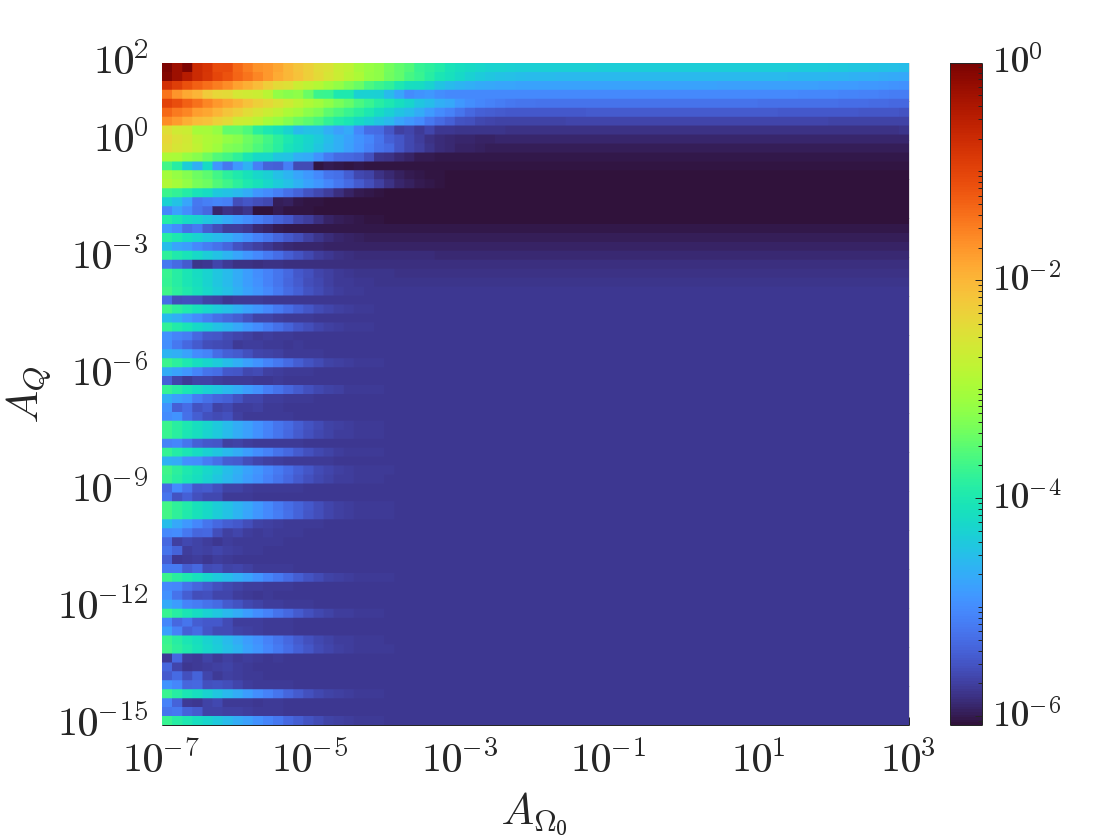}
        \caption{$L^2(Q)$ relative error.}
	\end{subfigure}\begin{subfigure}{.5\linewidth}
		\centering
		\includegraphics[width=\linewidth]{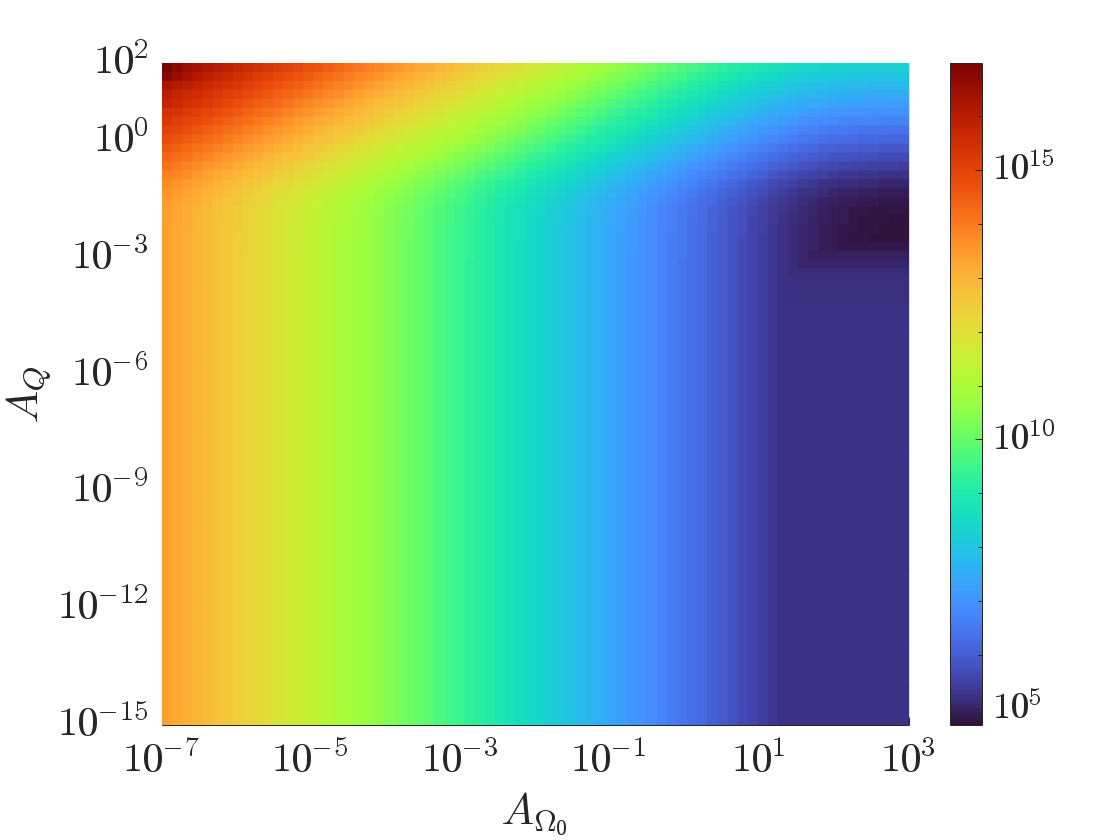}
		\caption{Condition number $\kappa_2(\mat B)$.}
	\end{subfigure}
    \caption{Problem 1: sensitivity with respect to the parameters $A_Q$ and $A_\OZ$. Left: $L^2(Q)$ relative error. Right: condition number of the Galerkin matrix.
}
	\label{fig:testA}
\end{figure}

Is the volume least-squares term $\int_Q \oW u \oW v\dx\dt$ necessary in the definition of the bilinear form $b$?
We can provide some insights by letting $A_Q \rightarrow 0$ while fixing $A_\OZ=1$, and looking at the corresponding $L^2(Q)$ error.
Figure~\ref{fig:fixedAs} shows that, while the optimal accuracy is achieved for $A_Q\approx0.02$,
the error for any positive value of $A_Q$ smaller than this value is larger than this optimal case by a small factor (at most 2.063).
This suggests that the least-squares term, required in the proof of coercivity in Theorem~\ref{thm:coerc_b}, is not necessary for a stable and accurate method.
However, when $A_Q = 0$ the method converges to the desired solution but the rates are sub-optimal, as shown in Figure~\ref{fig:degConvPlot} below.

The same reasoning cannot be applied to the term $\int_\OZ uv\dx$, because when this term is dropped (i.e.\ setting $A_\OZ = 0$) the Galerkin matrix $\mat B$ is no longer invertible as the constant functions belong to its kernel.
Indeed, from the red region on the left of both plots in Figure~\ref{fig:testA}, we see that for $A_\OZ\to0$ the method loses accuracy and stability.

\begin{figure}[ht!]
\centering
\begin{tikzpicture}
\pgfplotstableread[col sep=comma]{testA-p1-numEl32-bestA0errslice.dat}\fixedAzero
\begin{loglogaxis}[ xlabel={$A_Q$}, grid=major, width=0.75\linewidth, height=5cm, enlarge x limits=false, legend pos=north west,
	xtick={1e-16, 1e-14,1e-12,1e-10,1e-8,1e-6,1e-4, 1e-2, 1e0, 1e2}]
\addplot[blue, mark=none, line width=1pt] table[col sep=comma]{\fixedAzero};
\addlegendentry{$\|u-u_h\|_{L^2(Q)}/\|u\|_{L^2(Q)}$};
\end{loglogaxis}
\end{tikzpicture}
\caption{Relative error of the numerical solution to Problem 1 for $A_\OZ=1$ and variable $A_Q$.
The parameters $\beta=\beta^\#$, $\xi=1$, $\nu=2$ are chosen as in \eqref{eq:ParamChoice}.
}
\label{fig:fixedAs}
\end{figure}
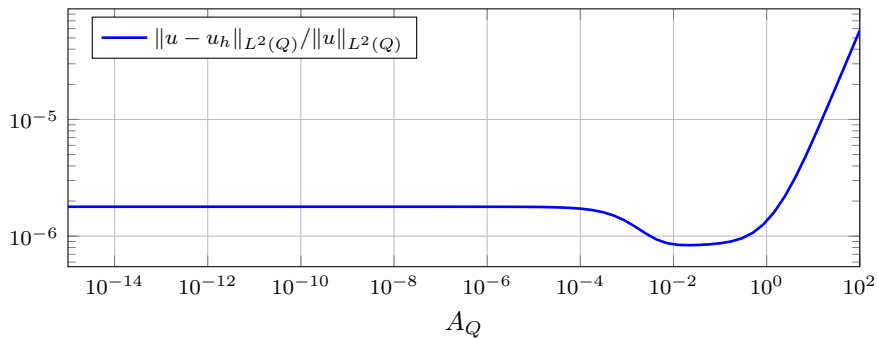

\subsubsection{Parameters \texorpdfstring{$\beta$}{beta}, \texorpdfstring{$\xi$}{xi} and \texorpdfstring{$\nu$}{nu}}
The theory developed in section~\ref{sec:CC} ensures the well-posedness and the stability of the variational problem \eqref{eq:vpV} and its Galerkin discretisation \eqref{eq:hGalerkin} if $\xi>0$, $\nu>1$, and $\beta$ satisfies the lower bound \eqref{eq:coefficient_coercivity_cond}.
To assess the sensitivity of the numerical solution with respect to these parameters, we let $N_x = N_t = 32$ and approximate the solution to Problem~1 by solving \eqref{eq:hGalerkin}
with $A_Q=10^{-2}$, $A_\OZ=1$, and $\beta$, $\xi$ and $\nu$ picked from the sets
\begin{align}\nonumber
	\beta\in\,&\big\{10^{s_k}: s_k\;\text{are $75$ equispaced nodes in}\;[-5, 5]\big\},\\
    \xi\in\,&\big\{10^{s_k}: s_k\;\text{are $75$ equispaced nodes in}\;[-3, 3]\big\},
	\label{eq:SetBetaXi}\\
	\nonumber
	\text{and}\;\;
	\nonumber
	\nu\in\,&\big\{10^{s_k}: s_k\;\text{are $10$ equispaced nodes in}\;[0.001, 2]\big\},
\end{align}
respectively.
We compute the $L^2(Q)$ relative error of the numerical solution and the condition number $\kappa_2(\mat B)$ in \eqref{eq:matGalerkin} for every choice of the parameters.

Such computations show that accuracy of the solution is not affected heavily by $\nu$.
Indeed, letting only $\nu$ vary, the maximum and the minimum of the $L^2$ error for fixed $\beta$ and $\xi$ differs at most by a factor $\approx30$ when $\beta$, $\xi$ and $\nu$ are picked to satisfy \eqref{eq:coefficient_coercivity_cond} and at most by a factor $\approx1.05$ under the additional assumption that $\xi\ge1$. Because of this, in the following we set $\nu=2$.

Clearly, for $\beta$ and $\xi$ such that \eqref{eq:coefficient_coercivity_cond} does not hold, which in Figure~\ref{fig:BetaXiErr} corresponds to the dotted region below the white line, the solution to \eqref{eq:hGalerkin} is not necessarily accurate because the problem is not necessarily well-posed, for coercivity is not guaranteed.
Despite this, some values of $\beta$ and $\xi$ in this region yield a good approximation of the exact solution, suggesting that \eqref{eq:coefficient_coercivity_cond} is a sufficient but not necessary condition for well-posedness.
Instead, when \eqref{eq:coefficient_coercivity_cond} is satisfied (region above the white line), the $L^2(Q)$ relative error is not much affected by the choice of $\beta$ and $\xi$. As noted in \S\ref{remark:condNumber}, the condition number of the matrix $\mat B$ is smallest when $\xi\sim1$ and $\beta$ satisfies~\eqref{eq:coefficient_coercivity_cond}.
In our numerical experiments, we found little dependence of the optimal parameters and their robustness on the $h_t$ and $h_x$ parameters, when the two mesh sizes decrease proportionally to one another.

\begin{figure}[ht!]
    \label{fig:betaxi_test}
	\begin{subfigure}{.5\linewidth}
		\includegraphics[width=\linewidth]{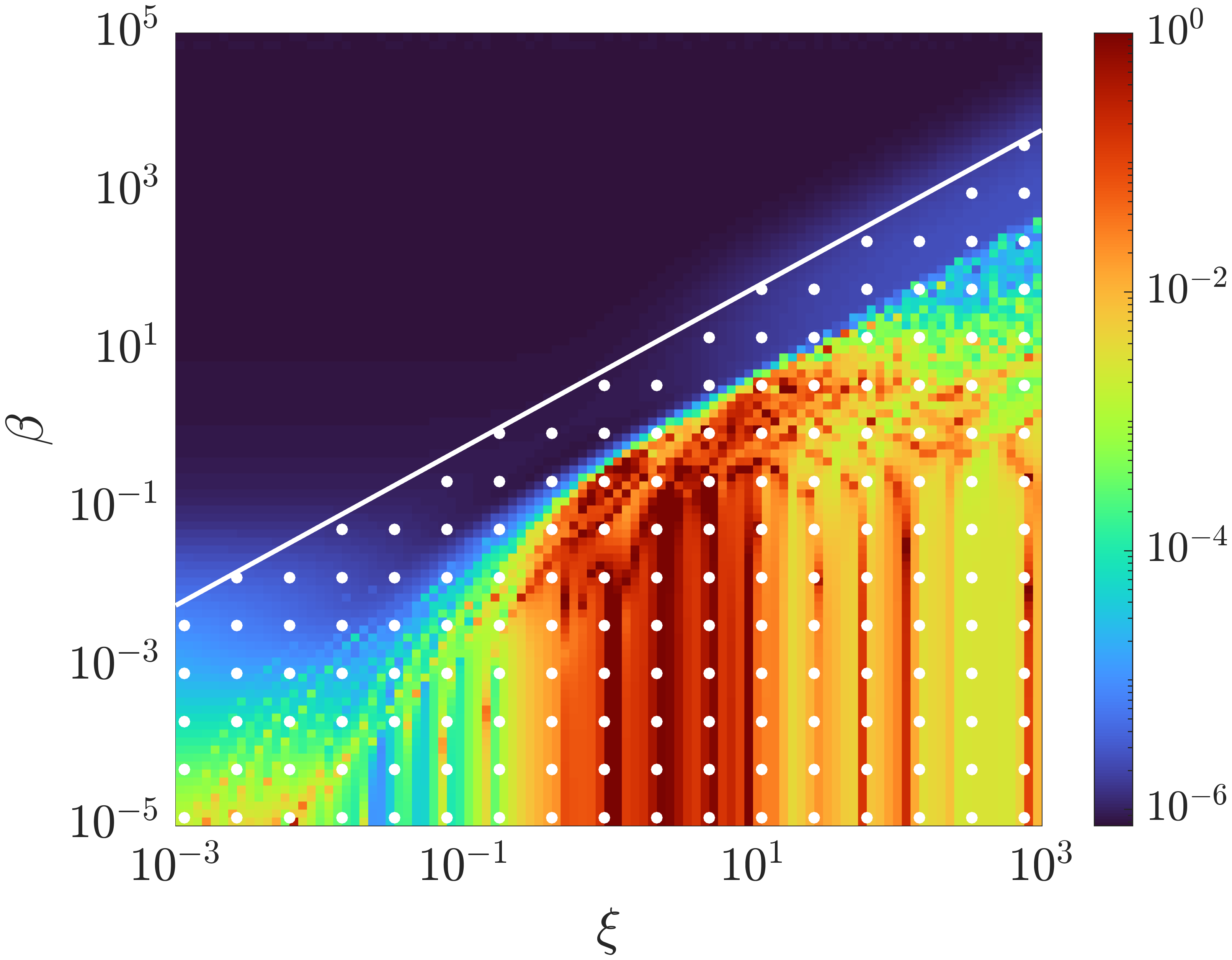 }
		\caption{$L^2(Q)$ relative error}
	\end{subfigure}\begin{subfigure}{.5\linewidth}
	\centering
	\includegraphics[width=\linewidth]{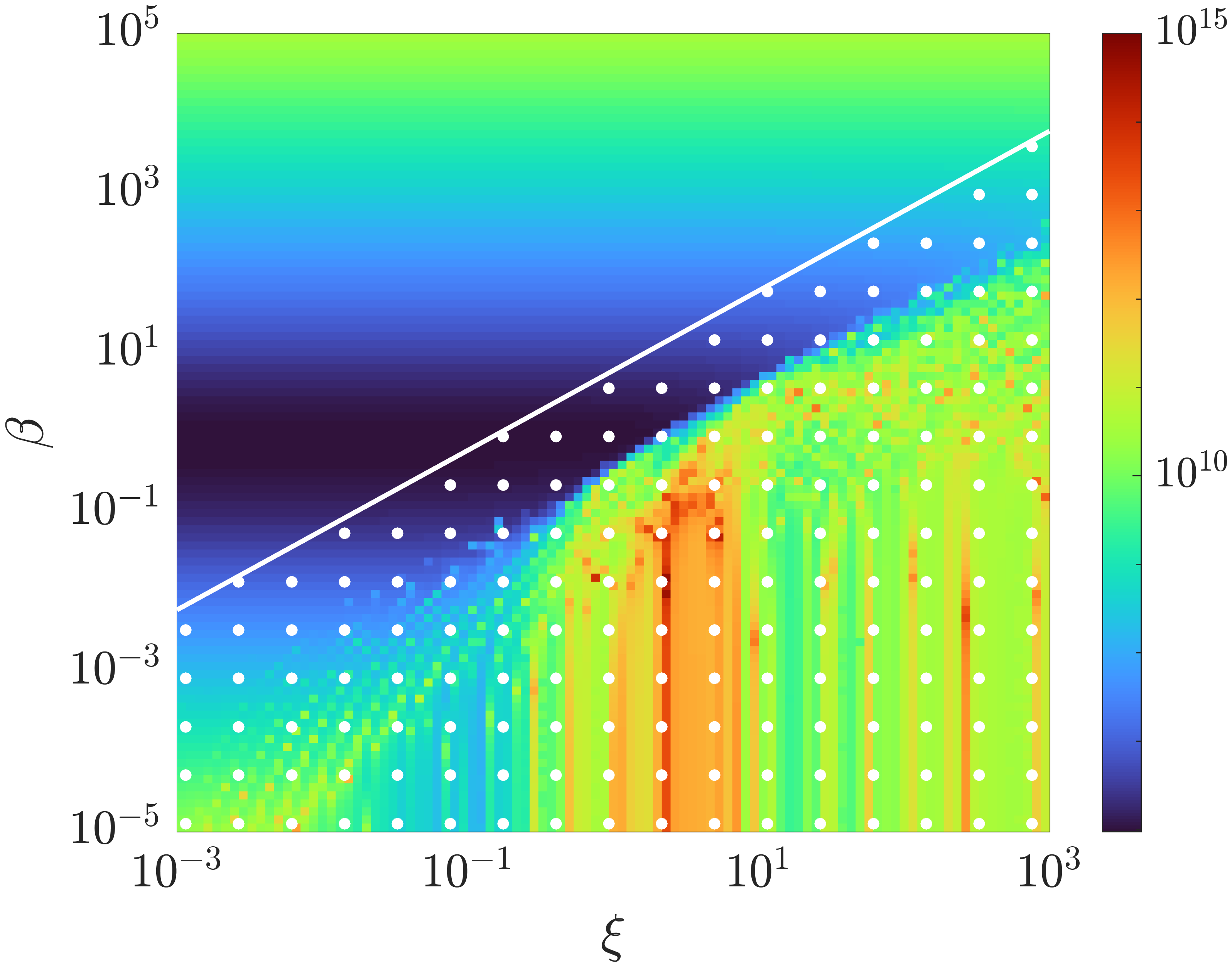 }
	\caption{Condition number of $\mat B$}
\end{subfigure}
\caption{$L^2(Q)$ relative error of the Galerkin solution and $\kappa_2(\mat B)$, for $N_x=N_t=32$, $A_Q = 10 ^ {-2}$, $A_\OZ = 1$ and $\nu=2$, while $\beta$ and $\xi$ vary in the sets \eqref{eq:SetBetaXi}.
The dotted region under the white line represents the pairs of $(\xi,\beta)$ for which \eqref{eq:coefficient_coercivity_cond} is violated.}
\label{fig:BetaXiErr}
\end{figure}

\subsection{Unconditional stability}

We demonstrate that, as proved in \S\ref{sec:interior_problem}, the formulation \eqref{eq:hGalerkin} is unconditionally stable, in the sense that no CFL condition is needed.
In other words we check that the Galerkin error remains bounded even when $h_x\ll h_t$.
To confirm this, compute the $L^2(Q)$ and $H^1(Q)$ relative errors\footnote{The $H^1(Q)$ norm is scaled to be dimensionally homogeneous as
$\|u\|_{H^1(Q)}^2 = T^{-2}\|u\|_Q^2 + \|\TimeDer u\|_Q^2 + c^2\|\Grad u\|_Q^2$.}
for Problem~1 and Problem~2 for $N_t = 8$ and $N_x \in \{2^n,\;n=1,\dots, 11\}$---with number of degrees of freedom ranging between $108$ and $73\:764$---and the parameters $A_Q=10^{-2}$, $A_\OZ=1$, $\xi=1$, $\beta=\beta^\#$, $\nu=2$ as in \eqref{eq:ParamChoice}.
As Figure~\ref{fig:CFLhx} shows, the $L^2(Q)$ and $H^1(Q)$ errors do not depend on the ratio $h_t/h_x$: indeed for large values of $h_t/h_x$ the error remains stable and is determined only by the time-mesh size $h_t$.
This is to be expected, as the formulation is well-posed for any conforming discrete space, regardless of the shape of the space--time elements and, in particular, the ratio $h_t / h_x$.

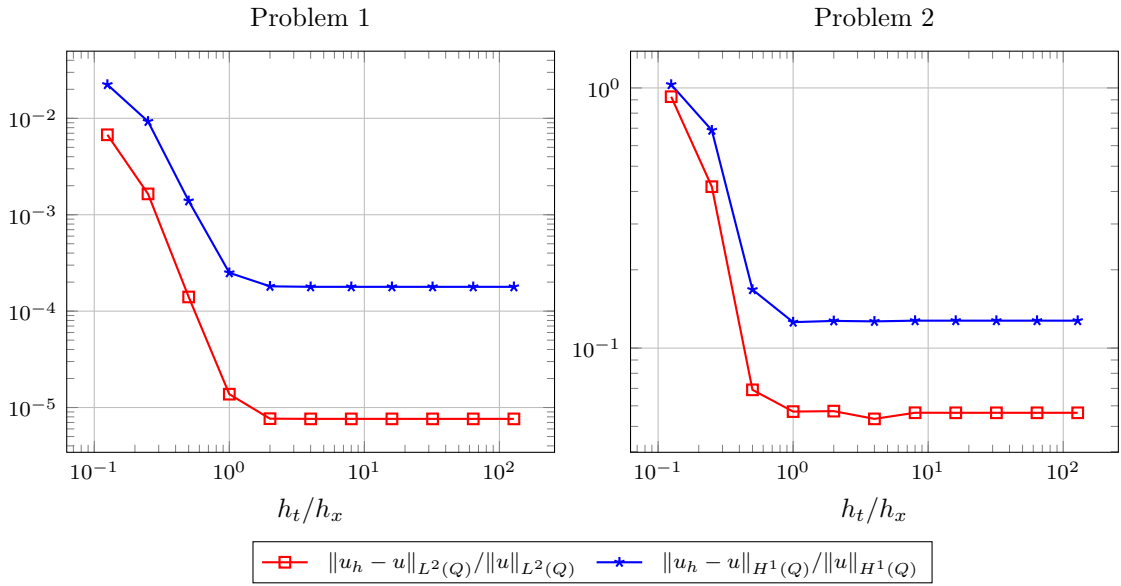
\begin{figure}[ht!]
\centering
\begin{tikzpicture}
\pgfplotstableread[col sep=comma]{p1-ptypeOPT-etypeRELATIVE-testcfl-nt8--convTable.dat}\datatabletwo;
\pgfplotstableread[col sep=comma]{p2-ptypeOPT-etypeRELATIVE-testcfl-nt8--convTable.dat}\datatablethree
\begin{groupplot}[
    group style={group size=2 by 1},
    width=0.5\linewidth,
    xlabel=$h_t/h_x$,
    legend style={column sep=0.5em},
    grid=major,
    height=6cm
    ]
	\nextgroupplot[ymode=log, xmode=log,  legend columns=2,
		legend to name=CFLlegend, title={Problem~1}]
	\addplot[L2error style] table[
		create on use/hthx/.style={create col/expr=\thisrow{Ht}/\thisrow{Hx}},
		columns={hthx, L2errors},
		x=hthx,
		y=L2errors,
		col sep=comma] {\datatabletwo};
	\addlegendentry{$\|u_h-u\|_{L^2(Q)}/\|u\|_{L^2(Q)}$};
	\addplot[H1error style] table[
		create on use/hthx/.style={create col/expr=\thisrow{Ht}/\thisrow{Hx}},
		x=hthx,
		y=H1errors,
		col sep=comma] {\datatabletwo};
	\addlegendentry{$\|u_h-u\|_{H^1(Q)}/\|u\|_{H^1(Q)}$};
\nextgroupplot[ymode=log, xmode=log, title={Problem~2}]
	\addplot[L2error style] table[
		create on use/hthx/.style={create col/expr=\thisrow{Ht}/\thisrow{Hx}},
		columns={hthx, L2errors},
		x=hthx,
		y=L2errors,
		col sep=comma] {\datatablethree};
	\addplot[H1error style] table[
		create on use/hthx/.style={create col/expr=\thisrow{Ht}/\thisrow{Hx}},
		x=hthx,
		y=H1errors,
		col sep=comma] {\datatablethree};
\end{groupplot}
\node [yshift=-3em, anchor=north] at ($(group c1r1.south)!0.5!(group c2r1.south)$) {\pgfplotslegendfromname{CFLlegend}};
\end{tikzpicture}
\caption{Unconditional stability test: $L^2(Q)$ and $H^1(Q)$ relative errors of the Galerkin solution of Problems~1 and~2 for $N_x \in \{2^n, n = 1,\dots,11\}$ and $N_t = 8$ ($h_t=0.125$).}
\label{fig:CFLhx}
\end{figure}

\subsection{Convergence analysis}

We study the $h$-convergence of the Galerkin error in $L^2(Q)$, $H^1(Q)$ and $V$ norms for the three problems in Table~\ref{tab:Problems}.
Figure~\ref{fig:ConvPlots} shows the errors for $N_x = N_t \in \{2^n,\; n=1,\ldots,7\}$, i.e.\ with numbers of degrees of freedom ranging between $36$ and $66\:564$.
The parameters $\xi=1$, $\beta=\beta^\#$, $\nu=2$ are picked as in \eqref{eq:ParamChoice}, $A_\OZ=1$ and we present results for both $A_Q=10^{-2}$ (left) and $A_Q=1$ (right).
The plots also show the best-approximation errors in each of the three norms (dashed lines).

Recalling that the $V$ norm is controlled by the $H^2(Q)$ norm, the quasi-optimality \eqref{eq:QOest} and using the approximation properties of cubic splines in $H^2(Q)$ (see e.g.\ \cite[Lemma~3.1]{BazilevsBeiraoCottrellHughesSangalli2006}), the $V$ norm of the Galerkin error decays quadratically in the mesh size $h:=\sqrt{h_x^2+h_t^2}=\sqrt5\, h_t$,
as long as the solution is sufficiently regular:
\begin{equation}\label{eq:RatesV}
\|u - u_h\|_V
\lesssim\inf_{v_h\in V_h}\|u - v_h\|_V
\lesssim\inf_{v_h\in V_h}\|u-v_h\|_{H^2(Q)}
\lesssim h^2 \|u\|_{H^4(Q)}.
\end{equation}
Optimal convergence in $V$ norm (green line) is observed in the plots for Problems~1 and 2, which admit smooth solutions.
We observe optimal convergence rates also in $H^1(Q)$ (rate $h^3$, blue lines) and $L^2(Q)$ norms (rate $h^4$, red lines); these are not ensured by the theory, which only guarantees $h^2$ rates from \eqref{eq:norm_def}, \eqref{eq:L2V} and \eqref{eq:RatesV}.

The bottom panels of Figure~\ref{fig:ConvPlots} show the convergence plot for Problem~3, whose solution $u\in H^{3/2-\epsilon}(Q)\setminus H^{3/2}(Q)$ for all $\epsilon>0$.
While the best-approximation errors in $H^1(Q)$ and $L^2(Q)$ (dashed blue and red lines) show the optimal convergence rates $h^{1/2}$ and $h^{3/2}$, respectively, the Galerkin solution (continuous lines) is slightly suboptimal.

For all three IBVPs, the $H^1(Q)$ and $L^2(Q)$ errors are smaller for $A_Q=10^{-2}$ (left plots) than for $A_Q=1$ (right plots), while for the $V$-norm error (green lines) the comparison is reversed.
This apparently happens because a small value of $A_Q$ does not control sufficiently strongly the residual term $\|\oW u_h\|_Q$, which is present in the $V$ norm.
This is also the reason why we observe a Galerkin error rate slightly lower than the best-approximation error rate for Problem 3 when $A_Q=10^{-2}$ in Figure~\ref{fig:ConvPlots} (see the green curves in the lower left plot).
Indeed, if this term is left out of the $V$ norm, then we have observed that the error for $A_Q=10^{-2}$ is smaller than that for $A_Q=1$ for all norms (plots not reported here).
The differences in accuracy between the two values of $A_Q$ are negligible for the source-driven Problem 1, and more substantial for the homogeneous Problems~2 and 3.

Analogous error plots for the formulation with $A_Q=0$, i.e.\ without the least-squares term, for which we cannot prove the coercivity, are shown in Figure~\ref{fig:degConvPlot}.
We observe convergence at slightly lower rates of order at least 3, 2, and 1 for the $L^2(Q)$, $H^1(Q)$, and $V$-norm errors, respectively, in both Problems~1 and~2.

Figure~\ref{fig:CondNumPlots} shows that for $A_Q=1$ the condition number $\kappa_2(\mat B)$ grows as $h^4$, confirming Remark~\ref{remark:condNumber}.
For $A_Q=10^{-2}$, the volume least-square term $\int_Q\oW u\oW w\dx\dt$ in $b(\cdot,\cdot)$ is less dominant and the rate is lower than $h^4$ for the range of parameters considered.
For the finest mesh considered,
when $A_Q=1$ the condition number is two orders of magnitude larger than when $A_Q=10^{-2}$, and is large enough to
cause an increase of the $L^2(Q)$ Galerkin error in the top-right plot of Figure~\ref{fig:ConvPlots}.

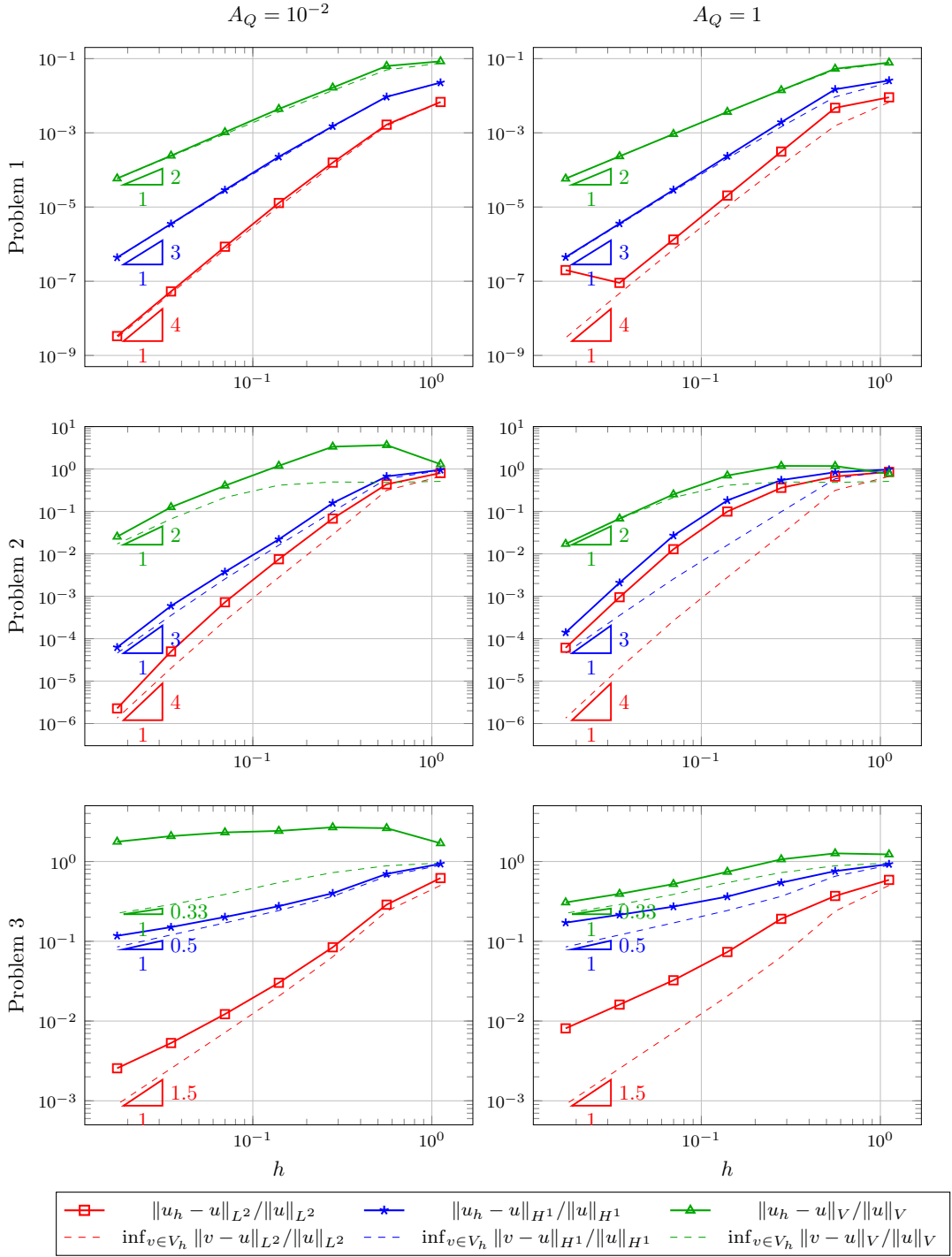
\begin{figure}
\centering
\pgfplotstableread[col sep=comma]{p1-ptypeOPT-etypeRELATIVE--convTable.dat}\OPTdataone
\pgfplotstableread[col sep=comma]{p1-ptypeGEN-etypeRELATIVE--convTable.dat}\GENdataone
\pgfplotstableread[col sep=comma]{p2-ptypeOPT-etypeRELATIVE--convTable.dat}\OPTdatatwo
\pgfplotstableread[col sep=comma]{p2-ptypeGEN-etypeRELATIVE--convTable.dat}\GENdatatwo
\pgfplotstableread[col sep=comma]{p3-ptypeOPT-etypeRELATIVE--convTable.dat}\OPTdatathree
\pgfplotstableread[col sep=comma]{p3-ptypeGEN-etypeRELATIVE--convTable.dat}\GENdatathree
\begin{tikzpicture}[scale=0.9]
\begin{groupplot}[group style={group size=2 by 3}, width=0.5\linewidth,  legend style={column sep=0.5em}, grid=major
]
\nextgroupplot[ymode=log, xmode=log,  legend columns=3,
legend to name=p1grouplegend,
title={$A_Q=10^{-2}$},
ylabel=Problem~1,
ytick = {1e-9, 1e-7, 1e-5, 1e-3, 1e-1},
ymin=0.5e-9, ymax=0.2e0
]
\addplot[L2error style] table[x=H, y=L2errors]{\OPTdataone};
\addlegendentry{$\|u_h-u\|_{L^2}/\|u\|_{L^2}$};
\addplot[H1error style] table[x=H, y=H1errors]{\OPTdataone};
\addlegendentry{$\|u_h-u\|_{H^1}/\|u\|_{H^1}$};
\addplot[Verror style] table[x=H, y=Verrors]{\OPTdataone};
\addlegendentry{$\|u_h-u\|_{V}/\|u\|_{V}$};
\addplot[L2bestapprox style] table[x=H, y=L2projErrors]{\OPTdataone};
\addlegendentry{$\inf_{v\in V_h}\|v-u\|_{L^2}/\|u\|_{L^2}$};
\addplot[H1bestapprox style] table[x=H, y=H1projErrors]{\OPTdataone};
\addlegendentry{$\inf_{v\in V_h}\|v-u\|_{H^1}/\|u\|_{H^1}$};
\addplot[Vbestapprox style] table[x=H, y=VprojErrors]{\OPTdataone};
\addlegendentry{$\inf_{v\in V_h}\|v-u\|_{V}/\|u\|_{V}$};
\logLogSlopeTriangle{0.2}{0.1}{0.08}{4}{L2error style};
\logLogSlopeTriangle{0.2}{0.1}{0.32}{3}{H1error style};
\logLogSlopeTriangle{0.2}{0.1}{0.57}{2}{Verror style};
\nextgroupplot[ymode=log, xmode=log,
title={$A_Q=1$},
ytick = {1e-9, 1e-7, 1e-5, 1e-3, 1e-1},
ymin=0.5e-9, ymax=0.2e0
]
\addplot[L2error style] table[x=H, y=L2errors]{\GENdataone};
\addplot[H1error style] table[x=H, y=H1errors]{\GENdataone};
\addplot[Verror style] table[x=H, y=Verrors]{\GENdataone};;
\addplot[L2bestapprox style] table[x=H, y=L2projErrors]{\GENdataone};
\addplot[H1bestapprox style] table[x=H, y=H1projErrors]{\GENdataone};
\addplot[Vbestapprox style] table[x=H, y=VprojErrors]{\GENdataone};
\logLogSlopeTriangle{0.2}{0.1}{0.08}{4}{L2error style};
\logLogSlopeTriangle{0.2}{0.1}{0.32}{3}{H1error style};
\logLogSlopeTriangle{0.2}{0.1}{0.57}{2}{Verror style};
\nextgroupplot[ymode=log, xmode=log,  legend columns=3,
	legend to name=p2grouplegend,
ylabel=Problem~2,
	ytick = {1e-6, 1e-5, 1e-4, 1e-3, 1e-2, 1e-1, 1e0, 1e1},
	ymin=0.3e-6, ymax=1e1
	]
\addplot[L2error style] table[x=H, y=L2errors]{\OPTdatatwo};
\addlegendentry{$\|u_h-u\|_{L^2}/\|u\|_{L^2}$};
\addplot[H1error style] table[x=H, y=H1errors]{\OPTdatatwo};
\addlegendentry{$\|u_h-u\|_{H^1}/\|u\|_{H^1}$};
\addplot[Verror style] table[x=H, y=Verrors]{\OPTdatatwo};
\addlegendentry{$\|u_h-u\|_{V}/\|u\|_{V}$};
\addplot[L2bestapprox style] table[x=H, y=L2projErrors]{\OPTdatatwo};
\addlegendentry{$\inf_{v\in V_h}\|v-u\|_{L^2}/\|u\|_{L^2}$};
\addplot[H1bestapprox style] table[x=H, y=H1projErrors]{\OPTdatatwo};
\addlegendentry{$\inf_{v\in V_h}\|v-u\|_{H^1}/\|u\|_{H^1}$};
\addplot[Vbestapprox style] table[x=H, y=VprojErrors]{\OPTdatatwo};
\addlegendentry{$\inf_{v\in V_h}\|v-u\|_{V}/\|u\|_{V}$};
\logLogSlopeTriangle{0.2}{0.1}{0.08}{4}{L2error style};
\logLogSlopeTriangle{0.2}{0.1}{0.29}{3}{H1error style};
\logLogSlopeTriangle{0.2}{0.1}{0.63}{2}{Verror style};
\nextgroupplot[ymode=log, xmode=log,
ytick = {1e-6, 1e-5, 1e-4, 1e-3, 1e-2, 1e-1, 1e0, 1e1},
ymin=0.3e-6, ymax=1e1
]
\addplot[L2error style] table[x=H, y=L2errors]{\GENdatatwo};
\addplot[H1error style] table[x=H, y=H1errors]{\GENdatatwo};
\addplot[Verror style] table[x=H, y=Verrors]{\GENdatatwo};
\addplot[L2bestapprox style] table[x=H, y=L2projErrors]{\GENdatatwo};
\addplot[H1bestapprox style] table[x=H, y=H1projErrors]{\GENdatatwo};
\addplot[Vbestapprox style] table[x=H, y=VprojErrors]{\GENdatatwo};
\logLogSlopeTriangle{0.2}{0.1}{0.08}{4}{L2error style};
\logLogSlopeTriangle{0.2}{0.1}{0.29}{3}{H1error style};
\logLogSlopeTriangle{0.2}{0.1}{0.63}{2}{Verror style};
\nextgroupplot[ymode=log, xmode=log,  legend columns=3,
	legend to name=p3grouplegend,
ylabel=Problem~3,
	xlabel=$h$,
	ymin=0.5e-3, ymax=0.5e1
	]
\addplot[L2error style] table[x=H, y=L2errors]{\OPTdatathree};
\addlegendentry{$\|u_h-u\|_{L^2}/\|u\|_{L^2}$};
\addplot[H1error style] table[x=H, y=H1errors]{\OPTdatathree};
\addlegendentry{$\|u_h-u\|_{H^1}/\|u\|_{H^1}$};
\addplot[Verror style] table[x=H, y=Verrors]{\OPTdatathree};
\addlegendentry{$\|u_h-u\|_{V}/\|u\|_{V}$};
\addplot[L2bestapprox style] table[x=H, y=L2projErrors]{\OPTdatathree};
\addlegendentry{$\inf_{v\in V_h}\|v-u\|_{L^2}/\|u\|_{L^2}$};
\addplot[H1bestapprox style] table[x=H, y=H1projErrors]{\OPTdatathree};
\addlegendentry{$\inf_{v\in V_h}\|v-u\|_{H^1}/\|u\|_{H^1}$};
\addplot[Vbestapprox style] table[x=H, y=VprojErrors]{\OPTdatathree};
\addlegendentry{$\inf_{v\in V_h}\|v-u\|_{V}/\|u\|_{V}$};
\logLogSlopeTriangle{0.2}{0.1}{0.06}{1.5}{L2error style};
\logLogSlopeTriangle{0.2}{0.1}{0.55}{0.5}{H1error style};
\logLogSlopeTriangle{0.2}{0.1}{0.66}{0.33}{Verror style};
\nextgroupplot[ymode=log, xmode=log,
xlabel=$h$,
ymin=0.5e-3, ymax=0.5e1,
]
\addplot[L2error style] table[x=H, y=L2errors]{\GENdatathree};
\addplot[H1error style] table[x=H, y=H1errors]{\GENdatathree};
\addplot[Verror style] table[x=H, y=Verrors]{\GENdatathree};;
\addplot[L2bestapprox style] table[x=H, y=L2projErrors]{\GENdatathree};
\addplot[H1bestapprox style] table[x=H, y=H1projErrors]{\GENdatathree};
\addplot[Vbestapprox style] table[x=H, y=VprojErrors]{\GENdatathree};
\logLogSlopeTriangle{0.2}{0.1}{0.06}{1.5}{L2error style};
\logLogSlopeTriangle{0.2}{0.1}{0.55}{0.5}{H1error style};
\logLogSlopeTriangle{0.2}{0.1}{0.66}{0.33}{Verror style};
\end{groupplot}
\node [yshift=-2.5em, anchor=north] at ($(group c1r3.south)!0.5!(group c2r3.south)$) {\pgfplotslegendfromname{p1grouplegend}};
\end{tikzpicture}
\caption{Relative errors for the problems in Table~\ref{tab:Problems}.
The continuous lines represent the Galerkin relative errors, the dashed lines the best-approximation relative errors.
Each color corresponds to one of the three norms.
Left: $A_Q=10^{-2}$. Right: $A_Q=1$.}
\label{fig:ConvPlots}
\end{figure}

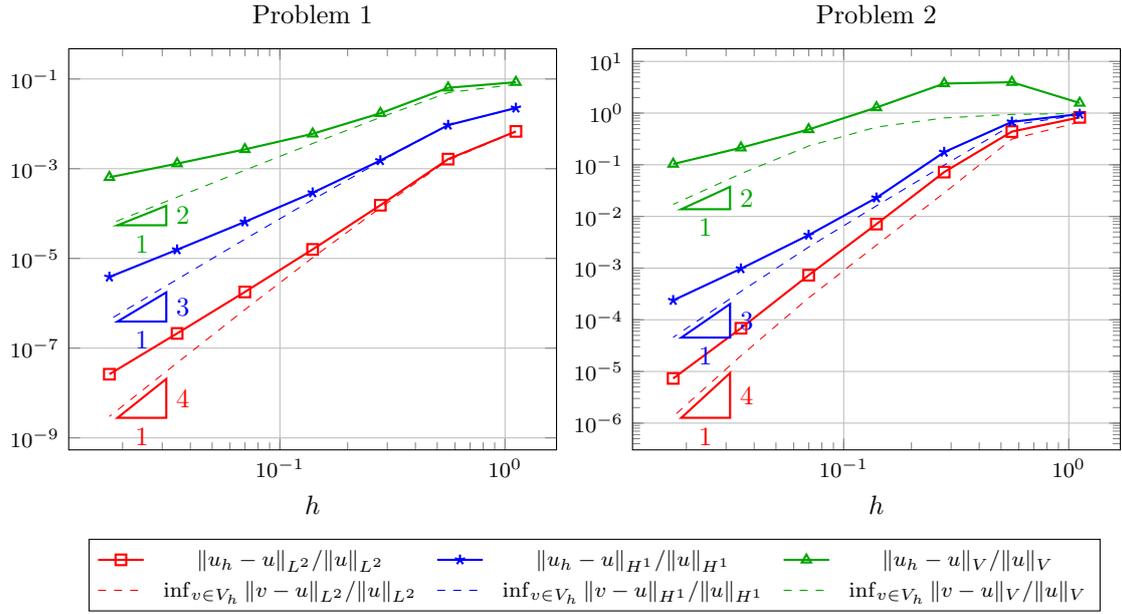
\begin{figure}[htb]
\centering
\begin{tikzpicture}
\pgfplotstableread[col sep=comma]{p1-ptypeCUSTOM-etypeRELATIVE--convTable.dat}{\degConvTwo}
\pgfplotstableread[col sep=comma]{p2-ptypeCUSTOM-etypeRELATIVE--convTable.dat}{\degConvThree}
\begin{groupplot}[
    group style={group size=2 by 1},
    width=0.5\linewidth,
    xlabel=$h$,
    legend style={column sep=0.5em},
    grid=major,
    height=6.5cm,
    ]
\nextgroupplot[ymode=log, xmode=log,  legend columns=3,
	legend to name=degConvLegend, title={Problem~1}]
\addplot[L2error style] table[x=H, y=L2errors] {\degConvTwo};
\addlegendentry{$\|u_h-u\|_{L^2}/\|u\|_{L^2}$};
\addplot[H1error style] table[x=H, y=H1errors] {\degConvTwo};
\addlegendentry{$\|u_h-u\|_{H^1}/\|u\|_{H^1}$};
\addplot[Verror style] table[x=H, y=Verrors] {\degConvTwo};
\addlegendentry{$\|u_h-u\|_{V}/\|u\|_{V}$};
\addplot[L2bestapprox style] table[x=H, y=L2projErrors] {\degConvTwo};
\addlegendentry{$\inf_{v\in V_h}\|v-u\|_{L^2}/\|u\|_{L^2}$};
\addplot[H1bestapprox style] table[x=H, y=H1projErrors] {\degConvTwo};
\addlegendentry{$\inf_{v\in V_h}\|v-u\|_{H^1}/\|u\|_{H^1}$};
\addplot[Vbestapprox style] table[x=H, y=VprojErrors] {\degConvTwo};
\addlegendentry{$\inf_{v\in V_h}\|v-u\|_{V}/\|u\|_{V}$};
\logLogSlopeTriangle{0.2}{0.1}{0.08}{4}{L2error style};
\logLogSlopeTriangle{0.2}{0.1}{0.32}{3}{H1error style};
\logLogSlopeTriangle{0.2}{0.1}{0.56}{2}{Verror style};
\nextgroupplot[ymode=log, xmode=log, title={Problem~2}]
\addplot[L2error style] table[x=H, y=L2errors] {\degConvThree};
\addplot[H1error style] table[x=H, y=H1errors] {\degConvThree};
\addplot[Verror style] table[x=H, y=Verrors] {\degConvThree};
\addplot[L2bestapprox style] table[x=H, y=L2projErrors] {\degConvThree};
\addplot[H1bestapprox style] table[x=H, y=H1projErrors] {\degConvThree};
\addplot[Vbestapprox style] table[x=H, y=VprojErrors] {\degConvThree};
\logLogSlopeTriangle{0.2}{0.1}{0.08}{4}{L2error style};
\logLogSlopeTriangle{0.2}{0.1}{0.28}{3}{H1error style};
\logLogSlopeTriangle{0.2}{0.1}{0.60}{2}{Verror style};
\end{groupplot}
\node [yshift=-2.4em, anchor=north] at ($(group c1r1.south)!0.5!(group c2r1.south)$) {\pgfplotslegendfromname{degConvLegend}};
\end{tikzpicture}
\caption{Convergence plots for Problem~1 and Problem~2 when $A_Q=0$ and $A_\OZ=1$ (i.e.\ without the volume least-squares term).
While rates are slightly suboptimal, convergence is stil achieved. Left: convergence plot of Problem~1. Right: convergence plot of Problem~2.}
\label{fig:degConvPlot}
\end{figure}

\begin{figure}[htb]
\centering
\pgfplotstableread[ col sep=comma,]{p1-GEN-conditionTable.dat}{\GENdata}
\pgfplotstableread[ col sep=comma,]{p1-OPT-conditionTable.dat}{\OPTdata}
\begin{tikzpicture}
\begin{loglogaxis}[
	xlabel=$h$,
	ylabel=$\kappa_2(\mat B)$,
	grid=major,
	width=.6\linewidth,
	height=5.5cm
	]
	\addplot[teal, mark=star, line width=1pt] table[x=H, y=Kconds]{\GENdata};
	\addlegendentry{$\kappa_2(\mat B)$ when $A_Q=1$};
	\addplot[orange, mark=square, line width=1pt] table[x=H, y=Kconds]{\OPTdata};
	\addlegendentry{$\kappa_2(\mat B)$ when $A_Q=10^{-2}$}
	\logLogCondSlopeTriangle{0.1}{-0.1}{0.75}{-4}{line width=1pt, teal};
	\logLogCondSlopeTriangle{0.1}{-0.1}{0.51}{-4}{line width=1pt, orange};
\end{loglogaxis}
\end{tikzpicture}
\caption{Condition number of the matrix $\mat B$ in \eqref{eq:matGalerkin} for Problem~1.
}
\label{fig:CondNumPlots}
\end{figure}
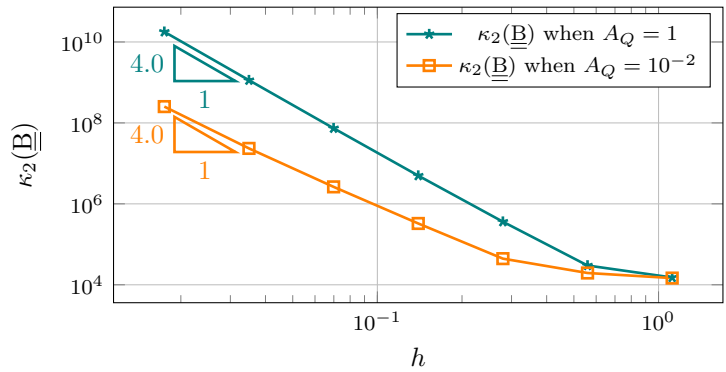

\subsection{Quasi-optimality}
Since the focus of the present work is on the design of a stable space--time formulation that can accommodate a range of discrete spaces, we study the quasi-optimality ratio of its solution.
The quasi-optimality ratio is also a measure of the dispersion and numerical pollution properties of the scheme, \cite{Moiola2014}.

In Table~\ref{tab:NumericalQO} we compare the values of the theoretical quasi-optimality bound $C_{qo}$ proved in Proposition~\ref{lemma:cea} against the ratios $\frac{\|u-u_h\|_{V}}{\inf_{v_h\in V_h}\|u-v_h\|_{V}}$ computed numerically.
We consider Problems~1 and~2 with parameters chosen as in the experiments of Figure~\ref{fig:ConvPlots}.
We also give the values of the quasi-optimality ratio in the $L^2(Q)$ and $H^1(Q)$ norms (first two columns).
We observe that the ratios obtained numerically are considerably better than the upper bound \eqref{eq:Cqo} (last column).
In particular, the numerical ratios in the third column are remarkably close to 1 for Problem~1, and only slightly larger for Problem~2.

\begin{table}[htb]
\centering
\begin{subtable}{\linewidth}
\centering
\pgfplotstableread[ col sep=comma]{p1-ptypeOPT-etypeRELATIVE--convTable.dat}{\OPTdata}
\pgfplotstableread[ col sep=comma]{p1-ptypeGEN-etypeRELATIVE--convTable.dat}{\GENdata}
\pgfplotstablenew[
col sep=comma,
create on use/H/.style={create col/copy column from table={\OPTdata}{H}},
create on use/L2Eopt/.style={create col/copy column from table={\OPTdata}{L2errors}},
create on use/L2PEopt/.style={create col/copy column from table={\OPTdata}{L2projErrors}},
create on use/L2QOopt/.style={create col/expr={\thisrow{L2Eopt}/\thisrow{L2PEopt}}},
create on use/L2Egen/.style={create col/copy column from table={\GENdata}{L2errors}},
create on use/L2PEgen/.style={create col/copy column from table={\GENdata}{L2projErrors}},
create on use/L2QOgen/.style={create col/expr={\thisrow{L2Egen}/\thisrow{L2PEgen}}},
create on use/H1Eopt/.style={create col/copy column from table={\OPTdata}{H1errors}},
create on use/H1PEopt/.style={create col/copy column from table={\OPTdata}{H1projErrors}},
create on use/H1QOopt/.style={create col/expr={\thisrow{H1Eopt}/\thisrow{H1PEopt}}},
create on use/H1Egen/.style={create col/copy column from table={\GENdata}{H1errors}},
create on use/H1PEgen/.style={create col/copy column from table={\GENdata}{H1projErrors}},
create on use/H1QOgen/.style={create col/expr={\thisrow{H1Egen}/\thisrow{H1PEgen}}},
create on use/VEopt/.style={create col/copy column from table={\OPTdata}{Verrors}},
create on use/VPEopt/.style={create col/copy column from table={\OPTdata}{VprojErrors}},
create on use/VQOopt/.style={create col/expr={\thisrow{VEopt}/\thisrow{VPEopt}}},
create on use/VEgen/.style={create col/copy column from table={\GENdata}{Verrors}},
create on use/VPEgen/.style={create col/copy column from table={\GENdata}{VprojErrors}},
create on use/VQOgen/.style={create col/expr={\thisrow{VEgen}/\thisrow{VPEgen}}},
create on use/QOconstGen/.style={create col/copy column from table={\GENdata}{QOconstEst}},
create on use/QOconstOpt/.style={create col/copy column from table={\OPTdata}{QOconstEst}},
columns={H,L2Eopt, L2PEopt, L2QOopt,L2Egen, L2PEgen, L2QOgen,H1Eopt, H1PEopt, H1QOopt,H1Egen, H1PEgen, H1QOgen,VEopt, VPEopt, VQOopt,VEgen, VPEgen, VQOgen, QOconstGen, QOconstOpt }
]{\pgfplotstablegetrowsof{\OPTdata}}\ALLdata
\pgfplotstabletypeset[
col sep=comma,
every head row/.style={
before row={
\toprule $h$
& \multicolumn{2}{c|}{$\frac{\|u-u_h\|_{L^2(Q)}}{\inf_{v_h\in V_h}\|u-v_h\|_{L^2(Q)}}$} & \multicolumn{2}{c|}{$\frac{\|u-u_h\|_{H^1(Q)}}{\inf_{v_h\in V_h}\|u-v_h\|_{H^1(Q)}}$} & \multicolumn{2}{c|}{$\frac{\|u-u_h\|_{V}}{\inf_{v_h\in V_h}\|u-v_h\|_{V}}$} & \multicolumn{2}{c}{$\frac{C_b}{\alpha_b}$} \\
\hline
}, after row=\hline},
every last row/.style={after row=\bottomrule},
columns={H, L2QOopt, L2QOgen, H1QOopt, H1QOgen, VQOopt, VQOgen, QOconstOpt, QOconstGen},
every column/.style={fixed, fixed zerofill, precision=2, set thousands separator={}},
columns/H/.style={column name={},sci, sci zerofill},
columns/L2QOopt/.style={column name={\small$A_Q=10^{-2}$}, column type=|c},
columns/L2QOgen/.style={column name={\small$A_Q=1$}, column type=c|},
columns/H1QOopt/.style={column name={\small$A_Q=10^{-2}$}},
columns/H1QOgen/.style={column name={\small$A_Q=1$}},
columns/VQOopt/.style={column name={\small$A_Q=10^{-2}$}, column type=|c},
columns/VQOgen/.style={column name={\small$A_Q=1$}},
columns/QOconstOpt/.style={column name={\small$A_Q=10^{-2}$}, column type=|c, precision=1 },
columns/QOconstGen/.style={column name={\small$A_Q=1$}, column type= c, precision=1},
]{\ALLdata}
\caption{Problem~1}
\end{subtable}
\begin{subtable}{\linewidth}
\centering
\pgfplotstableread[ col sep=comma]{p2-ptypeOPT-etypeRELATIVE--convTable.dat}{\OPTdata}
\pgfplotstableread[ col sep=comma]{p2-ptypeGEN-etypeRELATIVE--convTable.dat}{\GENdata}
\pgfplotstablenew[
col sep = comma,
create on use/H/.style={create col/copy column from table={\OPTdata}{H}},
create on use/L2Eopt/.style={create col/copy column from table={\OPTdata}{L2errors}},
create on use/L2PEopt/.style={create col/copy column from table={\OPTdata}{L2projErrors}},
create on use/L2QOopt/.style={create col/expr={\thisrow{L2Eopt}/\thisrow{L2PEopt}}},
create on use/L2Egen/.style={create col/copy column from table={\GENdata}{L2errors}},
create on use/L2PEgen/.style={create col/copy column from table={\GENdata}{L2projErrors}},
create on use/L2QOgen/.style={create col/expr={\thisrow{L2Egen}/\thisrow{L2PEgen}}},
create on use/H1Eopt/.style={create col/copy column from table={\OPTdata}{H1errors}},
create on use/H1PEopt/.style={create col/copy column from table={\OPTdata}{H1projErrors}},
create on use/H1QOopt/.style={create col/expr={\thisrow{H1Eopt}/\thisrow{H1PEopt}}},
create on use/H1Egen/.style={create col/copy column from table={\GENdata}{H1errors}},
create on use/H1PEgen/.style={create col/copy column from table={\GENdata}{H1projErrors}},
create on use/H1QOgen/.style={create col/expr={\thisrow{H1Egen}/\thisrow{H1PEgen}}},
create on use/VEopt/.style={create col/copy column from table={\OPTdata}{Verrors}},
create on use/VPEopt/.style={create col/copy column from table={\OPTdata}{VprojErrors}},
create on use/VQOopt/.style={create col/expr={\thisrow{VEopt}/\thisrow{VPEopt}}},
create on use/VEgen/.style={create col/copy column from table={\GENdata}{Verrors}},
create on use/VPEgen/.style={create col/copy column from table={\GENdata}{VprojErrors}},
create on use/VQOgen/.style={create col/expr={\thisrow{VEgen}/\thisrow{VPEgen}}},
create on use/QOconstGen/.style={create col/copy column from table={\GENdata}{QOconstEst}},
create on use/QOconstOpt/.style={create col/copy column from table={\OPTdata}{QOconstEst}},
columns={H,L2Eopt, L2PEopt, L2QOopt,L2Egen, L2PEgen, L2QOgen,H1Eopt, H1PEopt, H1QOopt,H1Egen, H1PEgen, H1QOgen,VEopt, VPEopt, VQOopt,VEgen, VPEgen, VQOgen, QOconstGen, QOconstOpt }
]{\pgfplotstablegetrowsof{\OPTdata}}\ALLdata
\pgfplotstabletypeset[
col sep=comma,
every head row/.style={
before row={
\toprule $h$
& \multicolumn{2}{c|}{$\frac{\|u-u_h\|_{L^2(Q)}}{\inf_{v_h\in V_h}\|u-v_h\|_{L^2(Q)}}$} & \multicolumn{2}{c|}{$\frac{\|u-u_h\|_{H^1(Q)}}{\inf_{v_h\in V_h}\|u-v_h\|_{H^1(Q)}}$} & \multicolumn{2}{c|}{$\frac{\|u-u_h\|_{V}}{\inf_{v_h\in V_h}\|u-v_h\|_{V}}$} & \multicolumn{2}{c}{$\frac{C_b}{\alpha_b}$} \\
\hline
}, after row=\hline},
every last row/.style={after row=\bottomrule},
columns={H, L2QOopt, L2QOgen, H1QOopt, H1QOgen, VQOopt, VQOgen, QOconstOpt, QOconstGen },
every column/.style={fixed, fixed zerofill, precision=2, set thousands separator={}},
columns/H/.style={column name=,sci, sci zerofill},
columns/L2QOopt/.style={column name={\small$A_Q=10^{-2}$}, column type=|c},
columns/L2QOgen/.style={column name={\small$A_Q=1$}, column type=c|},
columns/H1QOopt/.style={column name={\small$A_Q=10^{-2}$}},
columns/H1QOgen/.style={column name={\small$A_Q=1$}},
columns/VQOopt/.style={column name={\small$A_Q=10^{-2}$}, column type=|c},
columns/VQOgen/.style={column name={\small$A_Q=1$}},
columns/QOconstOpt/.style={column name={\small$A_Q=10^{-2}$}, column type=|c, precision=1},
columns/QOconstGen/.style={column name={\small$A_Q=1$}, column type= c, precision=1},
]{\ALLdata}
\caption{Problem~2}
\end{subtable}
\caption{Numerical quasi-optimality constants for Problems~1 and 2.
Ratios between the Galerkin solution errors and the best-approximation errors in $L^2(Q)$, $H^1(Q)$ and $V$ norms, for $A_Q\in\{10^{-2},1\}$,
$A_\OZ=1$, and $\xi=1$, $\beta=\beta^\#$, $\nu=2$ as in \eqref{eq:ParamChoice}.
The last two columns display the theoretical quasi-optimality bound in the $V$ norm proved in \eqref{eq:Cqo}.}
\label{tab:NumericalQO}
\end{table}

\subsection{Energy conservation}
Remark~\ref{remark:EnergyDecay} ensures the convergence to 0 of \emph{(i)} the difference $|\mathcal E(t;u)-\mathcal E(t;u_h)|$ between the energies of the IBVP and the Galerkin solutions, and of \emph{(ii)} the error energy $\mathcal E(t;u-u_h)$.
The first of these quantities is bounded by a multiple $\|u-u_h\|_V$ and the second by a multiple of $\|u-u_h\|_V^2$.

Figure~\ref{fig:EnergyError} reports some numerical results concerning these two measures for Problem~2.
Recall that the solution of this IBVP is a wave packet hitting the boundary (with impedance parameter $\theta\ne1$) and being partially reflected in the domain; its energy decreases considerably approximately between the time instants $0.3$ and $0.7$, and is roughly constant otherwise.

The top panel of Figure~\ref{fig:EnergyError} plots the energy error in $(0,T)$ for Problem~2 (divided at each time $t$ by the exact solution energy to take into account its evolution), where the Galerkin solution is computed with $N_x=N_t=128$.
We observe that this error is bounded uniformly in time and does not increase.

The lower panel of Figure~\ref{fig:EnergyError} shows the convergence of the $L^\infty(0,T)$ norm of the error energy
for a sequence of refined meshes.
We observe that the convergence of this quantity is much faster than the rate expected from the bound with $\|u-u_h\|_V^2$.

\begin{figure}[htb]
\centering
\begin{tikzpicture}
\pgfplotstableread[col sep=comma]{EnergyTableTS-p2-N128.dat}{\EnergyTable}
\pgfplotstableread[col sep=comma]{p2-ptypeOPT-etypeRELATIVE--energyConvTable.dat}{\EnergyConvTable}
\centering
\begin{groupplot}[group style={group size=1 by 2}]
\nextgroupplot[enlarge x limits=false, ymode=log, width=\linewidth, height=5cm, ytick={0.0001, 0.00001, 0.000001, 0.0000001, 0.00000001}, xlabel=$t$, grid=both]
\addplot[mark=none, line width=0.5pt, line join=bevel, blue] table[x=Ts, y=error]{\EnergyTable};
\addlegendentry{$|\mathcal E (t; u_h) - \mathcal E(t; u)|/\mathcal E(t; u)$}

\nextgroupplot[ymode=log, xmode=log, width=\linewidth, height=5cm, legend pos=south east, grid=both, xlabel=$h$,
ytick={1, 0.01, 0.0001, 0.000001, 0.00000001, 0.0000000001 }, ymin=0.000000000001]
\addplot[mark=star, blue, line width=1pt] table[x=H, y=EnergyNormErrors]{\EnergyConvTable};
\addlegendentry{$\sup_{t\in(0,T)}\mathcal E(t; u-u_h)/\sup_{t\in(0,T)}\mathcal E(t; u)$};
\addplot[Vbestapprox style, line width=1pt] table[ x=H, y=VprojErrors]{\EnergyConvTable};
\addlegendentry{$(\inf_{v_h\in V}\|u-v_h\|_V/\|u\|_V)^2$};
\logLogSlopeTriangle{0.25}{0.15}{0.53}{4}{Verror style};
\logLogSlopeTriangle{0.25}{0.15}{0.15}{7}{blue};
\end{groupplot}
\end{tikzpicture}
\caption{Top: relative error of the Galerkin solution energy for the approximation of Problem~2,  for $N_x = N_t = 128$, $A_Q=10^{-2}$ and $A_\OZ=1$, plotted for $768$ equispaced time instants in $[0,T]$.
Bottom: convergence of the numerical energy error in $L^\infty(0, T)$ norm for $N_x=N_t\in\{2^n,\; n=1,\ldots,8\}$.}
\label{fig:EnergyError}
\end{figure}
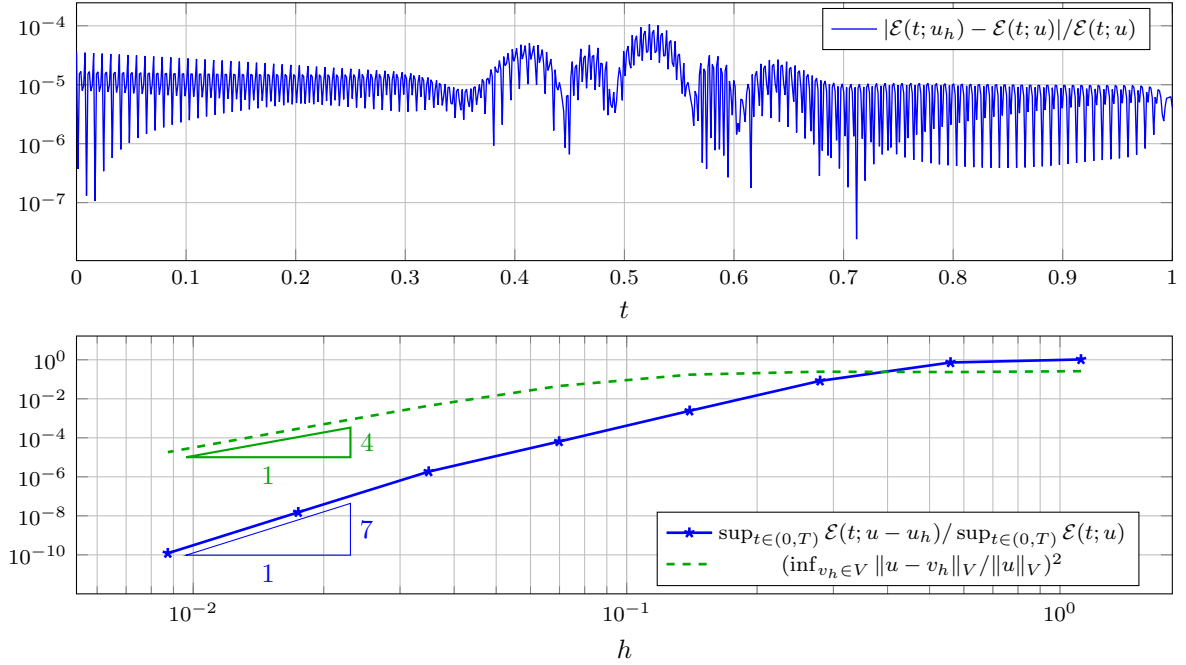

\section{Conclusions and further work}\label{sec:Conclusions}

We have derived a space--time variational formulation for a class of impedance and mixed impedance--Dirichlet initial--boundary value problems for the acoustic wave equation that is continuous and coercive in a norm stronger than $H^1$ on the space--time cylinder.
The main assumption is that the scalar product $\vx\cdot\vn$ of the position vector and the outgoing unit normal is positive on the impedance boundary and negative on the Dirichlet boundary; this includes the case of a star-shaped impedance domain, possibly containing a star-shaped sound-soft scatterer.
The derivation of the formulation relies on the use of classical Morawetz multipliers, following the analogue result for the Helmholtz equation of \cite{Moiola2014}.
The bilinear and the linear forms only include space and time partial derivatives of test and trial functions, together with linear (in $\vx$ and $t$) coefficients and a handful of parameters that can be easily chosen.
The proof of the coercivity and continuity estimates only requires elementary vector-calculus tools.
The formulation can be discretised with any space--time $H^2$-conforming space, leading in all cases to well-posed and quasi-optimal Galerkin methods.

We list a few improvements and extensions of the proposed method that might be considered, some of which are already underway, see \cite{Bignardi2025}.
\begin{itemize}
	\item The extension to a class of non-trapping, non-constant material parameters, using the Morawetz techniques of, e.g., \cite{Moiola2019} and \cite{Graham2019}.

\item The extension to more general classes of bounded domains, possibly using different Morawetz multipliers.
The extension to unbounded domains by means of high-order absorbing boundary conditions (ABC), perfectly-matched layers (PML), Dirichlet-to-Neumann maps (DtN).

\item The extension to vector wave problems, in particular electromagnetic and elastic waves.

\item The proof of optimal convergence rates in $L^2(Q)$ and $H^1(Q)$ norms.

\item The proof or disproof of the equality between the function spaces $V$ and $W$; see Remark~\ref{remark:V=W?}.
\item The precise characterisation of the IBVP data for which the solution of the proposed formulation coincides with the desired one (those in Theorem~\ref{thm:t_regularity} are sufficient but possibly not necessary).

\item A systematic numerical study of the proposed formulation in space dimensions higher than $1$, with general spline spaces and other discrete spaces on simplicial and unstructured meshes, including a parameter-sensitivity and dispersion analysis, and a comparison against established methods.

\item The lifting of the $H^2(Q)$-conformity requirement on the Galerkin space, by use of continuous-interior penalty (CIP) or by devising a mixed formulation that exploits Morawetz multipliers.

\item The development of techniques such as matrix-compression, preconditioners, a-posteriori estimators and adaptivity, that make use of coercivity to reduce the computational cost of the method.
\end{itemize}

\section*{Acknowledgements}
The authors are grateful to
Martin Berggren (Ume\aa{}),
Th\'eophile Chaumont-Frelet (INRIA),
Matteo Fornoni (Pavia),
Linus H\"agg (Ume\aa{}),
Michael Multerer (USI),
Ilaria Perugia (Vienna),
Andrea Signori (Milano Politecnico),
Euan Spence (Bath), and
Pietro Zanotti (Milano)
for helpful discussions.
The authors also acknowledge the support from the PRIN projects NA-from-PDEs, ASTICE (202292JW3F) and COSMIC (2022A79M75), GNCS-INDAM and PNRR-M4C2-I1.4-NC-HPC-Spoke6, funded by the European Union - Next Generation EU.

\bibliographystyle{abbrv}
\bibliography{references}
\end{document}